\documentclass[reqno]{amsart}
\usepackage[a4paper, margin=2cm]{geometry}
\usepackage{graphicx}
\usepackage{amssymb,amsmath}
\usepackage{amsthm,thmtools}
\usepackage{mathrsfs}
\usepackage{array}
\usepackage{comment}
\usepackage{mathtools}
\usepackage{tikz}
\usetikzlibrary{arrows,arrows.meta,decorations.pathreplacing,decorations.markings,shapes,calc,cd}
\tikzset{labelsize/.style={font=\scriptsize}}
\tikzset{2cell/.style={-implies,double,double equal sign distance}}
\usepackage[colorlinks=true,pagebackref]{hyperref}
\usepackage{cleveref}

\hypersetup{%
    colorlinks=true,
    allcolors=violet
}

\newcolumntype{N}{@{}m{0pt}@{}}

\usepackage{etoolbox}
\makeatletter
\patchcmd{\@counteralias}
 {\@ifdefinable{c@#1}}
 {\expandafter\@ifdefinable\csname c@#1\endcsname}
 {}{}
\makeatother

\makeatletter
\newcommand{\labeleditem}[1]{
\item[\text{#1}]\protected@edef\@currentlabel{\text{#1}}\phantomsection
}
\makeatother

\numberwithin{equation}{subsection}

\declaretheorem[style=plain,sibling=equation,name=Theorem]{theorem}
\declaretheorem[style=plain,sibling=theorem,name=Lemma]{lemma}
\declaretheorem[style=plain,sibling=theorem,name=Proposition]{proposition}
\declaretheorem[style=plain,sibling=theorem,name=Corollary]{corollary}

\declaretheorem[style=definition,qed=$\blacksquare$,sibling=theorem,name=Definition]{definition}

\declaretheorem[style=definition,qed=$\blacksquare$,sibling=theorem,name=Remark]{remark}
\declaretheorem[style=definition,qed=$\blacksquare$,sibling=theorem,name=Construction]{construction}

\crefname{theorem}{Theorem}{Theorems}
\crefname{section}{Section}{Sections}
\crefname{subsection}{Subsection}{Subsections}
\crefname{definition}{Definition}{Definitions}
\crefname{notation}{Notation}{Notations}
\crefname{example}{Example}{Examples}
\crefname{remark}{Remark}{Remarks}
\crefname{appendix}{Appendix}{Appendices}
\crefname{equation}{}{}
\crefname{construction}{Construction}{Constructions}
\crefname{corollary}{Corollary}{Corollaries}
\crefname{proposition}{Proposition}{Propositions}
\crefname{lemma}{Lemma}{Lemmas}

\mathchardef\mhyphen="2D
\newcommand{\Set}{\mathbf{Set}}
\newcommand{\cat}[1]{{\mathcal{#1}}}
\newcommand{\N}{\mathbb{N}}
\newcommand{\G}{\mathbb{G}}
\newcommand{\op}{\mathrm{op}}
\newcommand{\ppair}[1]{\langle #1\rangle} 

\newcommand{\undg}[1]{{{\left|{#1}\right|}}} 
\newcommand{\unds}[1]{{#1}_{\mathrm{s}}} 
\newcommand{\Kcomp}{\mathbin{\odot}} 
\newcommand{\ar}{\mathop{\mathrm{ar}}}
\newcommand{\StrCats}[1]{{\mathbf{Str}\mhyphen{#1}\mhyphen\mathbf{Cat}_\mathrm{s}}} 
\newcommand{\WkCats}[1]{{\mathbf{Wk}\mhyphen{#1}\mhyphen\mathbf{Cat}_\mathrm{s}}} 
\newcommand{\WkCat}[1]{{\mathbf{Wk}\mhyphen{#1}\mhyphen\mathbf{Cat}}} 
\newcommand{\GSet}{\mathbf{GSet}}
\newcommand{\id}[3]{\mathrm{id}_{#1}^{#2}(#3)} 
\newcommand{\newid}{\mathrm{id}}
\newcommand{\comp}[2]{\mathbin{{\ast}_{#1}^{#2}}}  
\newcommand{\spi}{\mathrm{sp}}
\newcommand{\OO}{G}
\newcommand{\kk}{\boldsymbol{k}}
\newcommand{\kkd}{\boldsymbol{k'}}
\newcommand{\uu}{\boldsymbol{u}}
\newcommand{\vv}{\boldsymbol{v}}

\newcommand{\ww}{\boldsymbol{w}}
\newcommand{\ATF}{\mathbf{ATF}}
\newcommand{\AATF}{\mathbb{ATF}}
\newcommand{\SSplEpi}{\mathbb{S}\mathbf{plEpi}}
\newcommand{\cod}{\mathrm{cod}}

\begin{document}
\title{$\omega$-weak equivalences between weak $\omega$-categories}
\author{Soichiro Fujii}
\address{School of Mathematical and Physical Sciences, Macquarie University, NSW 2109, Australia}
\address{Department of Mathematics and Statistics, Faculty of Science, Masaryk University, Kotl\'a\v{r}sk\'a 2, 611 37 Brno, Czech Republic}
\email{s.fujii.math@gmail.com}

\author{Keisuke Hoshino}
  \address{Research Institute for Mathematical Sciences, Kyoto University, Kyoto, Japan}
  \email{hoshinok@kurims.kyoto-u.ac.jp}

\author{Yuki Maehara}
  \address{Research Institute for Mathematical Sciences, Kyoto University, Kyoto, Japan}
  \email{ymaehar@kurims.kyoto-u.ac.jp}

\date{June 2, 2025}


\begin{abstract}
We study $\omega$-weak equivalences between weak $\omega$-categories in the sense of Batanin--Leinster. 
Our $\omega$-weak equivalences are strict $\omega$-functors satisfying essential surjectivity in every dimension, and when restricted to those between strict $\omega$-categories, they coincide with the weak equivalences in the model category of strict $\omega$-categories defined by Lafont, M\'etayer, and Worytkiewicz.
We show that the class of $\omega$-weak equivalences has the 2-out-of-3 property. 
We also consider a generalisation of $\omega$-weak equivalences, defined as weak $\omega$-functors (in the sense of Garner) satisfying essential surjectivity, and show that this class also has the 2-out-of-3 property. 
\end{abstract}

\maketitle

\section{Introduction}
A \emph{weak $\omega$-category} in the sense of Leinster \cite{Leinster_book}, based on an earlier definition by Batanin \cite{Batanin_98}, is an Eilenberg--Moore algebra for the ``universal weakening'' (or more precisely, the universal cofibrant replacement; see \cite{Garner_univ}) of the monad for strict $\omega$-categories on the category of globular sets.
In our previous work \cite{FHM1}, we studied weakly invertible cells \emph{in} such a weak $\omega$-category.
The purpose of this paper is to investigate weakly invertible $\omega$-functors \emph{between} weak $\omega$-categories.

Although one might be tempted to define the weakly invertible $\omega$-functors between weak $\omega$-categories as the weakly invertible 1-cells in a putative ``weak $\omega$-category of all (small) weak $\omega$-categories,'' the latter has not yet been defined; hence, we adopt a more hands-on definition in this paper.
We first define \emph{$\omega$-weak equivalences} as the strict $\omega$-functors between weak $\omega$-categories which are \emph{essentially surjective} in every dimension.
One of the main purposes of this paper is to prove that these $\omega$-weak equivalences enjoy the \emph{2-out-of-3 property} (\cref{thm:2-out-of-3}); that is, given a composable pair $f \colon X \to Y$, $g \colon Y \to Z$ of strict $\omega$-functors between weak $\omega$-categories, if any two of $f$, $g$, and $gf$ are $\omega$-weak equivalences then so is the third.
We believe that this result is a necessary step towards the construction of a (left semi-)model structure on the category of weak $\omega$-categories and strict $\omega$-functors, in which the weak equivalences are the $\omega$-weak equivalences.

In the case of \emph{strict} $\omega$-categories, such a model structure was constructed by Lafont, M\'etayer, and Worytkiewicz \cite{Lafont_Metayer_Worytkiewicz_folk_model_str_omega_cat}.
In particular, they established the 2-out-of-3 property of the $\omega$-weak equivalences between \emph{strict} $\omega$-categories; our result is its generalisation to the context of \emph{weak} $\omega$-categories.
Although our proof is partly inspired by theirs, we also face a brand new kind of complication that arises from the fact that we are dealing with \emph{weak} $\omega$-categories.
For example, there is a seemingly innocuous step in their argument which uses the fact that a $2$-cell $w$ (or more generally an $n$-cell $w$, with $n\geq 2$, whose $0$-source is $x$) in a strict $\omega$-category, as in the left diagram below, is the same as the composite $\id{}{}{x}\comp{0}{}w$, as in the right diagram below.
\[
\begin{tikzpicture}[baseline=-\the\dimexpr\fontdimen22\textfont2\relax ]
      \node(20) at (0,0) {$x$};
      \node(21) at (2,0) {$y$};

      \draw [->,bend left=30]  (20) to node[auto, labelsize] {$u$} (21);
      \draw [->,bend right=30] (20) to node[auto, swap,labelsize] {$v$} (21);

      \draw [2cell]  (1,0.25) to node[auto,labelsize] {$w$} (1,-0.25);
\end{tikzpicture}
\qquad\qquad
\begin{tikzpicture}[baseline=-\the\dimexpr\fontdimen22\textfont2\relax ]
      \node(1) at (-2,0) {$x$};
      \node(20) at (0,0) {$x$};
      \node(21) at (2,0) {$y$};

      \draw [->]  (1) to node[auto, labelsize] {$\id{}{}{x}$} (20);
      \draw [->,bend left=30]  (20) to node[auto, labelsize] {$u$} (21);
      \draw [->,bend right=30] (20) to node[auto, swap,labelsize] {$v$} (21);

      \draw [2cell]  (1,0.25) to node[auto,labelsize] {$w$} (1,-0.25);
\end{tikzpicture}
\]
When generalising such an argument to the context of \emph{weak} $\omega$-categories, it is natural to attempt to formalise the intuition that $\id{}{}{x} \comp{0}{} w$ is still ``essentially just $w$.''
However, in doing so, one readily realises that $w \colon u \to v$ and $\id{}{}{x}\comp{0}{}w \colon \id{}{}{x}\comp{0}{}u \to \id{}{}{x}\comp{0}{}v$ may not even be parallel, which makes unclear what the formal statement should be.

In order to deal with this sort of problem, we introduce the \emph{padding} construction (\cref{con:padding}).
Armed with a kind of naturality (\cref{lem:padding-a-pasting}), this construction allows us to go back and forth between non-parallel composites arising from the same pasting diagram.
Since ``applying'' the unit, associativity, or interchange law to compositions of high codimension (i.e., composition of $n$-cells along $k$-cells for $k < n-1$) can always result in such non-parallel composites, we believe the padding construction to be of interest even beyond the scope of this paper.

After establishing the 2-out-of-3 property, we extend the scope of study to the more general class of \emph{weak $\omega$-weak equivalences}, where the first ``weak'' indicates that these maps are not necessarily strictly $\omega$-functorial.
The notion of weak $\omega$-functor that we adopt is due to Garner \cite{Garner_homomorphisms}, which is essentially a span of strict $\omega$-functors with the left-pointing leg a trivial fibration.
In addition to generalising the 2-out-of-3 property to the larger class (\cref{prop:2-out-of-3-for-weak-omega-weak-eq}), we characterise the weak $\omega$-weak equivalences in terms of the essential $\omega$-surjectivity of their ``underlying'' globular maps (\cref{prop:weak-omega-weak-eq-via-und-glob}), establishing an analogy with the accepted notions of ``equivalence'' in low-dimensional cases such as biequivalences and triequivalences.

\subsection*{Related work}
At the conference \emph{$(\infty,n)$-Categories and Their Applications} held in Utrecht in April 2024, where the third-named author presented this work, we learnt from Cl\'emence Chanavat that a similar technique to our padding construction was being developed independently in the setting of diagrammatic sets.
This work later appeared in her joint paper with Amar Hadzihasanovic \cite{Chanavat_Hadzihasanovic_equivalences}.

\subsection*{Outline of the paper}
In \cref{sec:weak-omega-cat}, we recall the Batanin--Leinster notion of weak $\omega$-category and collect necessary results from our previous work on weakly invertible cells \cite{FHM1}.
We introduce the notion of $\omega$-weak equivalence and prove that they enjoy the 2-out-of-3 property in \cref{sec:2-out-of-3}.
This result is extended to weakly $\omega$-functorial maps in \cref{sec:weak-weak}.

\section{Preliminaries}
\label{sec:weak-omega-cat}
In this section, we recall Leinster's definition of weak $\omega$-category~\cite{Leinster_book} (which was inspired by an earlier definition by Batanin~\cite{Batanin_98}) and collect necessary results from our previous work~\cite{FHM1}.
This is a shortened version of material from \cite{FHM1}, to which we refer for more details and motivation.

\subsection{Globular sets}
We will write $\G$ for the category generated by the graph
\[
\begin{tikzpicture}[baseline=-\the\dimexpr\fontdimen22\textfont2\relax ]
      \node(0) at (0,0) {$0$};
      \node(1) at (2,0) {$1$};
      \node(d) at (4,0) {$\cdots$};
      \node(n) at (6,0) {$n$};
      \node(d2) at (8,0) {$\cdots$};
      
      \draw [->,transform canvas={yshift=3pt}] (0) to node[auto, labelsize] 
      {$\sigma_0$} (1); 
      \draw [->,transform canvas={yshift=-3pt}] (0) to node[auto, 
      swap,labelsize] 
      {$\tau_0$} (1); 
      \draw [->,transform canvas={yshift=3pt}] (1) to node[auto, labelsize] 
      {$\sigma_1$} (d); 
      \draw [->,transform canvas={yshift=-3pt}] (1) to node[auto, 
      swap,labelsize] 
      {$\tau_1$} (d); 
      \draw [->,transform canvas={yshift=3pt}] (d) to node[auto, labelsize] 
      {$\sigma_{n-1}$} (n); 
      \draw [->,transform canvas={yshift=-3pt}] (d) to node[auto, 
      swap,labelsize] 
      {$\tau_{n-1}$} (n); 
      \draw [->,transform canvas={yshift=3pt}] (n) to node[auto, labelsize] 
      {$\sigma_{n}$} (d2); 
      \draw [->,transform canvas={yshift=-3pt}] (n) to node[auto, 
      swap,labelsize] 
      {$\tau_{n}$} (d2);
\end{tikzpicture}
\]
subject to the relations
\[
\sigma_{n+1}\circ \sigma_{n} =\tau_{n+1}\circ \sigma_{n},\qquad \sigma_{n+1}\circ \tau_{n}=\tau_{n+1}\circ \tau_{n} \qquad (\forall n\in \N).
\]

By a \emph{globular set} we mean a functor $\G^\op\to\Set$, and the category $\GSet$ of globular sets is the presheaf category $[\G^\op,\Set]$.
Given a globular set $X$, we write $X_n$ for the set $Xn$, and we call its elements \emph{$n$-cells} of $X$. 
If $m<n$, we write $s^X_m$ for the composite
\[
\begin{tikzpicture}[baseline=-\the\dimexpr\fontdimen22\textfont2\relax ]
      \node(n) at (0,0) {$X_n$};
      \node(n-1) at (2.2,0) {$X_{n-1}$};
      \node(d) at (4.4,0) {$\cdots$};
      \node(m) at (6.6,0) {$X_m$};
      
      \draw [->] (n) to node[auto, labelsize] 
      {$X\sigma_{n-1}$} (n-1);
      \draw [->] (n-1) to node[auto, labelsize] 
      {$X\sigma_{n-2}$} (d);
      \draw [->] (d) to node[auto, labelsize] 
      {$X\sigma_{m}$} (m);
\end{tikzpicture}
\]
and similarly $t^X_m$ for the composite $X(\tau_m \circ \dots \circ \tau_{n-1})$.
Two $n$-cells $x$ and $y$ of $X$ are \emph{parallel} if
\begin{itemize}
    \item $n=0$, or
    \item $n \ge 1$, $s_{n-1}^X(x)=s_{n-1}^X(y)$ and $t_{n-1}^X(x)=t_{n-1}^X(y)$.
\end{itemize}
For $u \in X_n$ with $n\geq 1$, we write $u\colon x\to y$ to mean $s_{n-1}^X(u)=x$ and $t_{n-1}^X(u)=y$.

We write $\OO^{n}$ for the representable globular set $\G(-,n)$ and $\iota_n \colon \partial \OO^n \to \OO^n$ for the inclusion of its boundary (i.e., the largest proper globular subset).
So $\partial \OO^n$ represents parallel pairs of $(n-1)$-cells for $n \ge 1$, and $\partial \OO^0$ is the initial globular set $\emptyset$.
For any parallel pair $(x,y)$ of $(n-1)$-cells in $X$, we denote the corresponding map by $\ppair{x,y} \colon \partial \OO^n \to X$. 
Note that for $n\geq 1$, an $n$-cell $u$, and a parallel pair $(x,y)$ of $(n-1)$-cells in $X$, we have $u\colon x\to y$ in $X$ if and only if the diagram 
\[
\begin{tikzpicture}[baseline=-\the\dimexpr\fontdimen22\textfont2\relax ]
      \node(00) at (0,1) {$\partial\OO^n$};
      \node(01) at (2,1) {$X$};
      \node(10) at (0,-1) {$\OO^n$};
      
      \draw [->] (00) to node[auto, labelsize] {$\ppair{x,y}$} (01); 
      \draw [->] (00) to node[auto,swap,labelsize] {$\iota_n$} (10);   
      \draw [->] (10) to node[swap,auto,labelsize] {$u$} (01);
\end{tikzpicture}
\]
in $\GSet$ commutes.

\subsection{The monads \texorpdfstring{$T$}{T} and \texorpdfstring{$L$}{L}}
\label{subsec:T}
The category $\StrCats{\omega}$ of small strict $\omega$-categories and strict $\omega$-functors is monadic over $\GSet$, and moreover the induced monad $T=(T,\eta^T,\mu^T)$ is \emph{cartesian}~\cite[Theorem~F.2.2]{Leinster_book}; that is, the functor $T$ preserves pullbacks, and the natural transformations $\eta^T$ and $\mu^T$ are \emph{cartesian}, which means that all their naturality squares are pullbacks.
After reviewing the action of this monad $T$ more explicitly, we will define the monad $L$ for weak $\omega$-categories in terms of how it relates to $T$.

Write $1$ for the terminal globular set. 
The $n$-cells in the globular set $T1$ may be identified with the \emph{pasting schemes} of dimension $n$ in the following sense (see \cite[Section 2.2]{FHM1} for a proof). 

\begin{definition}\label{def:pasting-scheme}
    A \emph{pasting scheme} is a table of non-negative integers
    \[
    \kk = \begin{bmatrix}
    k_0 & & k_1 & & \dots & & k_r\\
    & \underline{k}_1 & & \underline{k}_2 & \dots & \underline{k}_r &
    \end{bmatrix}
    \]
    with $r\ge 0$ and $k_{i-1} > \underline k_i < k_i$ for all $1 \le i \le r$.
    For $n\ge 0$, a \emph{pasting scheme of dimension $n$} is a pasting scheme $\kk$ as above which moreover satisfies $k_i \le n$ for all $0 \le i \le r$.
\end{definition}

Each pasting scheme $\kk$ may be visualised as the arrangement of $k_i$-cells for $0 \le i \le r$ where the $(i-1)$st and the $i$th cells are composed along their $\underline{k}_i$-dimensional boundary.
For example, the pasting scheme
\[
\kk = \begin{bmatrix}
    2 & & 1 & & 2 & & 2\\
    & 0 & & 0 & & 1 & 
\end{bmatrix}
\]
corresponds to the arrangement:
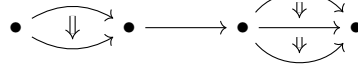
\[
\begin{tikzpicture}[baseline=-\the\dimexpr\fontdimen22\textfont2\relax ]
      \node(20) at (-1.5,0) {$\bullet$};
      \node(21) at (0,0) {$\bullet$};
      \node(22) at (1.5,0) {$\bullet$};
      \node(23) at (3,0) {$\bullet$};
      
      \draw [->]  (21) to (22);
      
      \draw [->,bend left=30]  (20) to node (2u) {} (21);      
      \draw [->,bend right=30] (20) to node (2b) {} (21); 
      \draw [2cell] (2u) to (2b);

      \draw [->,bend left=50]  (22) to node (3u) {} (23);
      \draw [->] (22) to node (3m) {} (23);
      \draw [->,bend right=50] (22) to node (3b) {} (23);
      \draw [2cell] (3u) to (3m);
      \draw [2cell] (3m) to (3b);
\end{tikzpicture}
\]

Observe that, given a pasting scheme $\kk$ of dimension $n$, the same table $\kk$ can also present a pasting scheme of dimension $n+1$ (or greater).
Whenever this ambiguity can cause confusion, we write $\kk^{(n)}$ to indicate that it is regarded as one of dimension $n$.

The unit map $\eta^T_1 \colon 1 \to T1$ sends the unique $n$-cell of $1$ to the pasting scheme
\[
\begin{bmatrix}
    n \\ {}
\end{bmatrix},
\]
which we often denote simply as $[n]$.

For an arbitrary globular set $X$, the $n$-cells in $TX$ may be identified with the \emph{pasting diagrams} of dimension $n$ in $X$, defined as follows.

\begin{definition}
\label{def:pasting-diagram}
    Let $X$ be a globular set and $\kk$ be a pasting scheme as in the previous definition.
    A \emph{pasting diagram of shape $\kk$ in $X$} is a table
    \[
    \uu = \begin{bmatrix}
    u_0 & & u_1 & & \dots & & u_r\\
    & \underline{u}_1 & & \underline{u}_2 & \dots & \underline{u}_r &
    \end{bmatrix}
    \]
    of cells $u_i \in X_{k_i}$ for $0 \le i \le r$ and $\underline u_i \in X_{\underline k_i}$ for $1 \le i \le r$, such that
    \[
    t_{\underline k_i}^X(u_{i-1}) =\underline u_i = s_{\underline k_i}^X(u_i)
    \]
    for all $1 \le i \le r$.
    The pasting diagram $\uu$ is \emph{of dimension $n$} if its shape $\kk$ is so.
\end{definition}
A pasting diagram $\uu$ of shape $\kk$ in $X$ may be understood as the pasting scheme $\kk$ labelled with the cells $u_i$ (and $\underline{u}_i$).
For example, the pasting diagram
\[
\uu=
\begin{bmatrix}
    \alpha & & h & & \beta & & \gamma\\
    & b & & c & & j & 
\end{bmatrix}
\]
(whose shape is the example of $\kk$ from above) may be visualised as
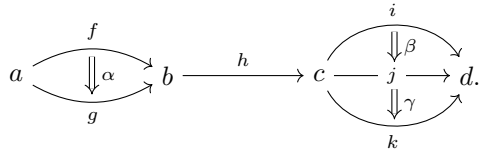
\begin{equation*}
	\begin{tikzpicture}[baseline=-\the\dimexpr\fontdimen22\textfont2\relax ]
		\node(20) at (0,0) {$a$};
		\node(21) at (2,0) {$b$};
		\node(22) at (4,0) {$c$};
		\node(23) at (6,0) {$d$.};
		
		\draw [->,bend left=30]  (20) to node[auto, labelsize] {$f$} (21);
		\draw [->,bend right=30] (20) to node[auto, swap,labelsize] {$g$} (21); 
		\draw [->]  (21) to node[auto, labelsize] {$h$} (22);
		\draw [->,bend left=60]  (22) to node[auto, labelsize] {$i$} (23);
		\draw [->]  (22) to node[midway,fill=white,labelsize] {$j$} (23);
		\draw [->,bend right=60]  (22) to node[auto, swap,labelsize] {$k$} (23);
		
		\draw [2cell]  (1,0.25) to node[auto,labelsize] {$\alpha$} (1,-0.25);
		\draw [2cell]  (5,0.58) to node[auto,labelsize] {$\beta$} (5,0.18);
		\draw [2cell]  (5,-0.18) to node[auto,labelsize] {$\gamma$} (5,-0.58);
	\end{tikzpicture}
\end{equation*}

Similarly to pasting schemes, we write $\uu^{(n)}$ when we wish to explicitly indicate that a pasting diagram $\uu$ is regarded as one of dimension $n$.

The unit map $\eta^T_X \colon X \to TX$ sends an $n$-cell $x$ of $X$ to the pasting diagram
\[
\begin{bmatrix}
    x \\ {}
\end{bmatrix},
\]
which we often denote simply as $[x]$.
For more details on $T$ (such as how the source and target operations are defined on $TX$), see \cite[Section 2.2]{FHM1}.

We will define the category $\WkCats{\omega}$ of \emph{weak $\omega$-categories} and \emph{strict $\omega$-functors} between them as the Eilenberg--Moore category of a monad $L$ on $\GSet$. 
This $L$ is the \emph{initial cartesian monad over $T$ with contraction} in the following sense.
(The intuition behind this definition will be explained in \cref{subsec:basic-operations}; see also \cite[Section 2.3]{FHM1}.)

Firstly, $L = (L,\eta^L, \mu^L)$ is itself a monad.
Being \emph{cartesian over $T$} means having a monad map $\ar \colon L \to T$ (called the \emph{arity}) which is cartesian as a natural transformation.\footnote{The existence of such $\ar$ implies that $L$ too is a cartesian monad.
The proof of this fact proceeds by considering the suitable squares involving $L$ that one wishes to show to be pullbacks, connecting them to the corresponding squares involving $T$ via $\ar$, and then using the pasting lemma for pullback squares and the fact that $T$ is a cartesian monad; cf.~\cite[Proof of Proposition~6.2.1]{Leinster_book}.}
It moreover has a \emph{contraction} $\kappa$, which is a function that assigns, to each commutative square as in the solid part below (with $n\geq 1$), a diagonal lift as indicated:
\begin{equation}
\label{eqn:contraction}
\begin{tikzpicture}[baseline=-\the\dimexpr\fontdimen22\textfont2\relax ]
      \node(00) at (0,1) {$\partial\OO^n$};
      \node(01) at (2,1) {$L1$};
      \node(10) at (0,-1) {$\OO^n$};
      \node(11) at (2,-1) {$T1$.};
      
      \draw [->] (00) to node[auto, labelsize] {$\ppair{\phi,\phi'}$} (01); 
      \draw [->] (01) to node[auto, labelsize] {$\ar_1$} (11); 
      \draw [->] (00) to node[auto,swap,labelsize] {$\iota_n$} (10); 
      \draw [->] (10) to node[auto,swap,labelsize] {$\kk$} (11);   
      \draw [->, dashed] (10) to node[midway,fill=white,labelsize] {$\kappa(\ppair{\phi,\phi'},\kk)$} (01);
\end{tikzpicture}
\end{equation}
Finally, $L$ is \emph{initial} in a suitable category of such monads, whose existence is guaranteed by \cite[Proposition~9.2.2]{Leinster_book}.

\begin{definition}[{\cite{Leinster_book}}]\label{def:weak-omega-cats}
    A \emph{weak $\omega$-category} is an Eilenberg--Moore algebra  $(X,\xi\colon LX\to X)$ of the monad $L=(L,\eta^L,\mu^L)$ on $\GSet$.
\end{definition}

We can immediately observe the following fact (mentioned in, e.g., \cite[Section~5]{Garner_homomorphisms}). 
\begin{proposition}
  \label{L-is-finitary}
    The functor $L \colon \GSet \to \GSet$ preserves filtered colimits.
\end{proposition}
\begin{proof}
    Up to natural isomorphism, we can factorise $L$ as
    \[
    \begin{tikzcd}
        \GSet
        \arrow [r] &
        \GSet/T1
        \arrow [r, "\mathrm{ar}_1^*"] &
        \GSet/L1
        \arrow [r] &
        \GSet
    \end{tikzcd}
    \]
    where the first factor sends $X \in \GSet$ to $T! \colon TX \to T1$ (which preserves filtered colimits because $T$ does by \cite[Theorem F.2.2]{Leinster_book} or \cite[Proposition 14.2.8]{Ara_etal_polygraphs}), the second factor pulls back along $\ar_1$ (which preserves filtered colimits because it is a left adjoint; note that $\GSet$, being a presheaf category, is locally cartesian closed), and the last factor is the canonical projection.
\end{proof}

\subsection{Basic operations in a weak \texorpdfstring{$\omega$}{ω}-category}
\label{subsec:basic-operations}
The two most basic operations in a categorical structure is taking the identity on a cell and composing a pair of cells.
The purpose of this subsection is to identify and analyse the weak $\omega$-categorical version of these operations.

We start by explaining the sense in which the globular set $L1$ parametrises operations in a weak $\omega$-category.
From now on, we will abuse the notation and write $\ar \colon L1 \to T1$ for the component $\ar_1$ of the natural transformation $\ar \colon L \to T$ at the terminal globular set $1$.

\begin{definition}
    By a \emph{pasting instruction} of dimension $n$, we mean an $n$-cell in $L1$.
\end{definition}

The intuition behind this terminology is that we may informally think of the fibre of $\ar \colon L1 \to T1$ over $\kk \in (T1)_n$ as the set of all possible ``instructions'' on exactly how to paste given cells arranged in shape $\kk$ in a weak $\omega$-category.
For example, there are (at least) two distinct $1$-cells in $L1$ of arity
\[
\begin{bmatrix}
    1 & & 1 & & 1\\
    & 0 & & 0 &
\end{bmatrix}
\]
corresponding to $(f,g,h) \mapsto (f \comp{0}{} g) \comp{0}{}h$ and $(f,g,h) \mapsto f \comp{0}{} (g \comp{0}{}h)$; the formal definition of the symbol $\comp{0}{}$ will be given in \cref{def:id-and-comp-in-weak-omega-cat} below.
From this viewpoint, the contraction $\kappa$ (\ref{eqn:contraction}) may be interpreted as encoding a sort of pasting theorem: given a pasting scheme $\kk$ and instructions $\ppair{\phi,\phi'}$ on how to paste its boundary, one can extend it to a pasting instruction $\kappa\bigl(\ppair{\phi,\phi'},\kk\bigr)$ on the whole of $\kk$.

Although there are many operations of given arity $\kk$ in general, it is convenient to fix one among such as our default choice.
In the case of $\kk = [n]$, there is an intuitively natural choice.

\begin{definition}
    For each $n\geq 0$, we will write $\widetilde{e}_n \in (L1)_n$ for the image of the unique $n$-cell in $1$ under the unit map $\eta^L_1 \colon 1 \to L1$.
\end{definition}

We extend this family of choices to arbitrary $\kk$ using the contraction $\kappa$.

\begin{definition}[{\cite[Definition~2.5.1]{FHM1}}]
    For any $n\geq 0$ and $\kk\in(T1)_n$, we define $\spi(\kk)\in(L1)_n$ with $\ar(\spi(\kk))=\kk$ inductively (on $n$) as follows. 
    \begin{itemize}
        \item If $\kk=[n]\in (T1)_n$, then $\spi([n])=\widetilde{e}_n$.
        \item Otherwise $\kk$ is of the form $\kk\colon \kkd\to \kkd$ for some $\kkd\in(T1)_{n-1}$. (Note that necessarily $n\geq 1$.) We define $\spi(\kk)=\kappa\bigl(\ppair{\spi(\kkd),\spi(\kkd)},\kk\bigr)$.
    \end{itemize}
    We call $\spi(\kk)$ the \emph{standard pasting instruction} of arity $\kk$.
\end{definition}
 
By construction, this assignment defines a globular map $\spi\colon T1\to L1$ which is moreover a section of $\ar \colon L1 \to T1$ and commutes with $\eta^T_1$ and $\eta^L_1$.

Next, we describe how to turn a pasting instruction into an actual operation in a weak $\omega$-category.
Observe that, since $\ar\colon L\to T$ is cartesian, the globular set $LX$ for any $X\in\GSet$ can be computed as the pullback
\begin{equation}
\label{eqn:LX}
\begin{tikzpicture}[baseline=-\the\dimexpr\fontdimen22\textfont2\relax ]
      \node(00) at (-0.5,1) {$LX$};
      \node(01) at (1.5,1) {$L1$};
      \node(10) at (-0.5,-1) {$TX$};
      \node(11) at (1.5,-1) {$T1$.};
      
      \draw [->] (00) to node[auto, labelsize] {$L!$} (01); 
      \draw [->] (01) to node[auto, labelsize] {$\ar$} (11); 
      \draw [->] (00) to node[auto,swap,labelsize] {$\ar_X$} (10); 
      \draw [->] (10) to node[auto,swap,labelsize] {$T!$} (11); 
\end{tikzpicture}
\end{equation}
This fact allows us to identify the $n$-cells in $LX$ with the pairs $(\phi,\uu)$ consisting of a pasting instruction $\phi \in (L1)_n$ and a pasting diagram $\uu\in (TX)_n$ such that $\uu$ is of shape $\ar(\phi)$.
Thus, given a weak $\omega$-category $(X,\xi)$ and a pasting instruction $\phi \in (L1)_n$, we can define the ``paste according to the instruction $\phi$'' operation on $X$ as $\uu \mapsto \xi(\phi,\uu)$.
Here the argument $\uu$ must be a pasting diagram of shape $\ar(\phi)$ in $X$; in other words, this operation has arity $\ar(\phi)$.

Note that the naturality of $\ar \colon L \to T$ implies that the functor $L$ sends $f\colon X\to Y$ to $Lf\colon LX\to LY$ given by $(Lf)(\phi,\uu) = \bigl(\phi, (Tf)(\uu)\bigr)$.
It follows that any strict $\omega$-functor $f\colon (X,\xi)\to (Y,\nu)$ between weak $\omega$-categories preserves the ``paste according to $\phi$'' operations because, for any pasting diagram $\uu$ in $X$ of shape $\ar(\phi)$, we have 
\[
        f\bigl(\xi(\phi,\uu)\bigr)
        = \nu\bigl((Lf)(\phi,\uu)\bigr)
        = \nu\bigl(\phi,(Tf)(\uu)\bigr).
\]

Finally, we define the identity and binary composition operations as the ``paste according to $\spi(\kk)$'' operations of suitable arities $\kk$.
More precisely, in \cref{def:id-and-comp-in-weak-omega-cat} below (which generalises \cite[Definition~2.5.2]{FHM1}), we define 
\begin{enumerate}
    \item an operation that returns an identity $n$-cell on a given $m$-cell, and
    \item an operation that takes an $m_1$-cell and an $m_2$-cell, composes them along a $k$-cell, and regards the result as an $n$-cell.
\end{enumerate}
(These cover precisely all $\kk$ as in \cref{def:pasting-scheme} with $r \le 1$.)

\begin{definition}
\label{def:id-and-comp-in-weak-omega-cat}
    Let $(X,\xi)$ be a weak $\omega$-category.
    \begin{enumerate}
        \item Given natural numbers $m,n$ with $m<n$ and an $m$-cell $x$ of $X$, we define the $n$-cell 
    \[
    \id{n}{X}{x}=
    \xi\bigl( \spi([m]^{(n)}),[x]^{(n)}\bigr)
    \]
    of $X$.
        \item Given natural numbers $k,m_1,m_2,n$ with $k<m_1\leq n$ and $k<m_2\leq n$, an $m_1$-cell $u$ of $X$, and an $m_2$-cell $v$ of $X$ such that $t_k^X(u)=s_k^X(v)$, we define the $n$-cell 
            \[
    u\comp{k,n}{X}v=
    \xi\biggl( \spi\biggl({\begin{bmatrix}
    m_1 & & m_2 \\
    & k &
    \end{bmatrix}^{(n)}}\biggr),    \begin{bmatrix}
    u & & v \\
    & x &
    \end{bmatrix}^{(n)}\biggr)
    \]
    of $X$, where $x=t_k^X(u)=s_k^X(v)$. When $n=\max\{m_1,m_2\}$, we also write $u\comp{k,n}{X}v$ as $u\comp{k}{X}v$.
    \end{enumerate}
    When $X$ is clear from the context, we omit the superscript.
\end{definition}

Observe that, since $\eta^L$ is natural (in particular with respect to the unique map $X \to 1$) and $\ar \colon L \to T$ is a monad morphism, the component of the unit $\eta^L_X\colon X\to LX$ at arbitrary $X \in \GSet$ maps each $n$-cell $x\in X_n$ to the $n$-cell $(\widetilde{e}_n,[x]) \in (LX)_n$.
It follows that we have
\begin{align*}
    \xi\bigl( \spi([n]^{(n)}),[x]\bigr)
    &=\xi\left( \widetilde{e}_n,[x]\right)\\
    &=\xi\circ\eta^L_X(x)\\
    &=x
\end{align*}
for each $n$-cell $x$ in a weak $\omega$-category $(X,\xi)$, which is why we did not include the case $m=n$ in (1) of the above definition (although the expression still makes perfect sense).

These operations satisfy the following source and target formulas, generalising \cite[Proposition~2.5.3]{FHM1}.
Note that, since the pullback square (\ref{eqn:LX}) is one in $\GSet$, the source and target operations on $LX$ are induced by those on $L1$ and $TX$.

\begin{proposition}
    \label{prop:s-and-t-for-id-and-ast}
    Let $(X,\xi)$ be a weak $\omega$-category.
    \begin{enumerate}
        \item Let $m,n$ be natural numbers with $m<n$ and $x$ an $m$-cell of $X$. Let $\ell$ be a natural number.
        \begin{enumerate}
            \item If $\ell< m$, then we have 
            \[
            s_\ell\bigl(\id{n}{}{x}\bigr)=s_\ell(x)\quad\text{and}\quad t_\ell\bigl(\id{n}{}{x}\bigr)=t_\ell(x).
            \]
            \item We have 
            \[
            s_m\bigl(\id{n}{}{x}\bigr)=x=t_m\bigl(\id{n}{}{x}\bigr).
            \]
            \item If $m< \ell< n$, then we have 
            \[
            s_\ell\bigl(\id{n}{}{x}\bigr)=\id{\ell}{}{x}=t_\ell\bigl(\id{n}{}{x}\bigr).
            \]
        \end{enumerate}
        \item Let $k,m_1,m_2,n$ be natural numbers with $k<m_1\leq n$ and $k<m_2\leq n$, $u$ an $m_1$-cell of $X$, and $v$ an $m_2$-cell of $X$ such that $t_k(u)=s_k(v)$ holds. Let $\ell$ be a natural number.  
        \begin{enumerate}
            \item If $\ell<k$, then we have 
            \[
            s_\ell(u\comp{k,n}{}v)=s_\ell(u)=s_\ell(v)\quad\text{and}\quad
            t_\ell(u\comp{k,n}{}v)=t_\ell(u)=t_\ell(v).
            \]
            \item We have 
            \[
            s_k(u\comp{k,n}{}v)=s_k(u)\quad\text{and}\quad
            t_k(u\comp{k,n}{}v)=t_k(v).
            \]
            \item If $k<\ell< \min\{m_1,m_2\}$, then we have 
            \[
            s_\ell(u\comp{k,n}{}v)=s_\ell(u)\comp{k,\ell}{}s_\ell(v)\quad\text{and}\quad
            t_\ell(u\comp{k,n}{}v)=t_\ell(u)\comp{k,\ell}{}t_\ell(v).
            \]
            \item If $m_1\leq \ell<m_2$, then we have 
            \[
            s_\ell(u\comp{k,n}{}v)=u\comp{k,\ell}{}s_\ell(v)\quad\text{and}\quad
            t_\ell(u\comp{k,n}{}v)=u\comp{k,\ell}{}t_\ell(v).
            \]
            \item If $m_2\leq \ell<m_1$, then we have 
            \[
            s_\ell(u\comp{k,n}{}v)=s_\ell(u)\comp{k,\ell}{}v\quad\text{and}\quad
            t_\ell(u\comp{k,n}{}v)=t_\ell(u)\comp{k,\ell}{}v.
            \]
            \item If $\max\{m_1,m_2\}\leq \ell<n$, then we have 
            \[
            s_\ell(u\comp{k,n}{}v)=u\comp{k,\ell}{}v=t_\ell(u\comp{k,n}{}v).
            \]
        \end{enumerate}
    \end{enumerate}
\end{proposition}
\begin{proof}
These are all straightforward consequences of the description of the source and target operations of $TX$ in \cite[Section~2.2]{FHM1} and the fact that the standard pasting instructions form a globular subset of $L1$ (so that we have $s^{L1}_\ell\bigl(\spi(\kk)\bigr)= \spi\bigl(s^{T1}_\ell(\kk)\bigr)$ and $t^{L1}_\ell\bigl(\spi(\kk)\bigr)= \spi\bigl(t^{T1}_\ell(\kk)\bigr)$ for any $\kk\in (T1)_n$ and $0\leq \ell< n$). For example, the first equation in (d) of (2) can be proved as follows, writing the $k$-cell $t_k^X(u)=s_k^X(v)$ as $x$:
\begin{align*}
s_\ell^X(u\comp{k,n}{X}v) &= s_\ell^X\biggl(
    \xi\biggl( \spi\biggl({\begin{bmatrix}
    m_1 & & m_2 \\
    & k &
    \end{bmatrix}^{(n)}}\biggr),    \begin{bmatrix}
    u & & v \\
    & x &
    \end{bmatrix}^{(n)}\biggr)\biggr)\\
    &= 
    \xi\biggl( s_\ell^{L1}\biggl(\spi\biggl({\begin{bmatrix}
    m_1 & & m_2 \\
    & k &
    \end{bmatrix}^{(n)}}\biggr)\biggr), s^{TX}_\ell\biggl(   \begin{bmatrix}
    u & & v \\
    & x &
    \end{bmatrix}^{(n)}\biggr)\biggr)\\
    &= 
    \xi\biggl( \spi\biggl({\begin{bmatrix}
    m_1 & & \ell \\
    & k &
    \end{bmatrix}^{(\ell)}}\biggr),    \begin{bmatrix}
    u & & s_\ell^X(v) \\
    & x &
    \end{bmatrix}^{(\ell)}\biggr)\\
    & = u\comp{k,\ell}{X}s_\ell^X(v).\qedhere
\end{align*}
\end{proof}

In later sections, whenever we use the notation $\id{n}{X}{x}$ for an $m$-cell $x$ of $X$, we have $m=n-1$, and whenever we use the notation $u\comp{k,n}{X}v$ for an $m_1$-cell $u$ and an $m_2$-cell $v$ of $X$, we have $n=\max\{m_1,m_2\}$ (so it may be written as $u\comp{k}{X}v$) and $k=\min\{m_1,m_2\}-1$.

The pasting instructions for the identity and binary composition operations are themselves given by the identity and binary composition in the free weak $\omega$-category $L1=(L1,\mu^L_1)$ in the following sense.
\begin{proposition}
\label{prop:pastings-in-L1}
    For any pasting instruction $\phi \in (L1)_n$ of arity
    \[
    \kk = \begin{bmatrix}
    k_0 & & k_1 & & \dots & & k_r\\
    & \underline{k}_1 & & \underline{k}_2 & \dots & \underline{k}_r &
    \end{bmatrix},
    \]
    we have $\phi = \mu^L_1(\phi,\widetilde{e}_{\kk})$ where
    \[
    \widetilde{e}_{\kk} = 
    \begin{bmatrix}
    \widetilde{e}_{k_0} & & \widetilde{e}_{k_1} & & \dots & & \widetilde{e}_{k_r}\\
    & \widetilde{e}_{\underline{k}_1} & & \widetilde{e}_{\underline{k}_2} & \dots & \widetilde{e}_{\underline{k}_r} &
    \end{bmatrix}.
    \]
    Consequently, given a weak $\omega$-category $(X,\xi)$,
    \begin{enumerate}
        \item we have $\id{n}{X}{x} = \xi\bigl(\id{n}{L1}{\widetilde{e}_m},[x]^{(n)}\bigr)$ for all natural numbers $m < n$ and all $x \in X_m$, and
        \item we have
        \[
        u \comp{k,n}{X} v = \xi\left(\widetilde{e}_{m_1} \comp{k,n}{L1} \widetilde{e}_{m_2}, \begin{bmatrix}
        u & & v \\
        & x &
        \end{bmatrix}^{(n)}\right)
        \]
        for all natural numbers $k,m_1,m_2,n$ with $k<m_1\leq n$ and $k<m_2\leq n$, and all $u \in X_{m_1}$ and $v \in X_{m_2}$ with $t^X_k(u)=s^X_k(v) = x$.
    \end{enumerate}
\end{proposition}
\begin{proof}
    Let us first make the following observations.
    \begin{itemize}
        \item [(i)] By calling the unique $n$-cell in the terminal globular set $1$ by the name ``$n$,'' we may regard each pasting scheme $\kk$ as a pasting diagram $\uu$ in the globular set $1$, and this is unambiguous in the sense that both $\kk$ and $\uu$ specify the same $n$-cell in $T1$.
        \item [(ii)] When we apply the description of the cells in $LX$ as pairs $(\phi,\uu)$ to the case $X = 1$, we end up identifying each pasting instruction $\phi$ of arity $\kk$ with the pair $(\phi,\kk)$.
    \end{itemize}
    It follows that, for any pasting instruction $\phi \in (L1)_n$ of arity $\kk$ as in the statement, we have
    \begin{align*}
        \phi &= (\mu^L_1 \circ L\eta^L_1)(\phi) 
        \tag*{\text{($\mu^L\circ L\eta^L=1_L$)}}\\
        &= (\mu^L_1 \circ L\eta^L_1)(\phi, \kk) 
        \tag*{\text{(ii)}}\\
        &= \mu^L_1 \bigl(\phi, (T\eta_1^L)(\kk)\bigr). 
        \tag*{\text{(action of $Lf$ in the case $f=\eta_1^L$)}}
    \end{align*}
    Using (i), we can compute $(T\eta_1^L)(\kk)$ as
    \[
        (T\eta_1^L)\left(\begin{bmatrix}
        k_0 & & k_1 & & \dots & & k_r\\
        & \underline{k}_1 & & \underline{k}_2 & \dots & \underline{k}_r &
        \end{bmatrix}\right)
        = \begin{bmatrix}
        \eta^L_1(k_0) & & \eta^L_1(k_1) & & \dots & & \eta^L_1(k_r)\\
        & \eta^L_1(\underline{k}_1) & & \eta^L_1(\underline{k}_2) & \dots & \eta^L_1(\underline{k}_r) &
        \end{bmatrix}\\
        = \widetilde{e}_{\kk}.
     \]
     This proves the first assertion.
     Moreover, the case $\phi = \spi\bigl([m]^{(n)}\bigr)$ yields
     \[
     \spi\bigl([m]^{(n)}\bigr) = \mu_1^L\Bigl(\spi\bigl([m]^{(n)}\bigr), [\widetilde{e}_m]^{(n)}\Bigr) = \id{n}{L1}{\widetilde{e}_m}
     \]
    where the second equality is simply (1) of \cref{def:id-and-comp-in-weak-omega-cat} with $(X,\xi) = (L1,\mu^L_1)$.
    This completes the proof of the assertion (1), and similarly (2) follows from the case $\phi = \spi\biggl(\begin{bmatrix}
         m_1 & & m_2 \\
         & k & 
     \end{bmatrix}^{(n)}\biggr)$.
\end{proof}

\begin{remark}
\label{rmk:id-in-TX}
    For each globular set $X$, we obtain the free strict $\omega$-category $(TX,\mu^T_X)$, which then induces the weak $\omega$-category $(TX,\mu^T_X\circ \ar_{TX})$. 
    In this weak $\omega$-category $TX$, for each $m<n$ and $\uu\in (TX)_m$, the identity $n$-cell $\id{n}{TX}{\uu}\in (TX)_n$ is the $m$-dimensional pasting diagram $\uu$ in $X$ regarded as an $n$-dimensional one, i.e., $\id{n}{TX}{\uu}=\uu^{(n)}$.
    (In particular, for any $m<n$ and $\kk\in (T1)_m$, we have $\id{n}{T1}{\kk}=\kk^{(n)}$.) 
    Using this notation, (1) of \cref{prop:pastings-in-L1} can be expressed as $\id{n}{X}{x}=\xi\bigl(\id{n}{L1}{\widetilde e_m},\id{n}{TX}{[x]}\bigr)$. Here is an alternative proof of this:
    \begin{align*}
        \xi\bigl(\id{n}{L1}{\widetilde e_m},\id{n}{TX}{[x]}\bigr)
        &= \xi\bigl(\id{n}{LX}{\widetilde e_m,[x]}\bigr)
        \tag*{\text{($LX$ is a pullback of $L1$ and $TX$ in $\WkCats{\omega}$)}}\\
        &= \newid^X_n\bigl(\xi(\widetilde e_m,[x])\bigr)
        \tag*{\text{($\xi\colon (LX,\mu^L_X)\to (X,\xi)$ is a strict $\omega$-functor)}}\\
        &= \id{n}{X}{x}.
        \tag*{\text{($\xi\circ \eta^L_X=1_X$)}}
    \end{align*}
    
    Similarly, in the situation of (2) of \cref{prop:pastings-in-L1}, we have 
    \[
    [u]\comp{k,n}{TX}[v]= 
    \begin{bmatrix}
        u & & v \\
        & x &
        \end{bmatrix}^{(n)}
    \]
    (cf.~\cite[Section~2.2]{FHM1}), 
    which allows us to express the main claim there as 
    $u \comp{k,n}{X} v = \xi\bigl(\widetilde{e}_{m_1} \comp{k,n}{L1} \widetilde{e}_{m_2}, [u]\comp{k,n}{TX}[v]\bigr)$. This can be proved similarly:
    \begin{align*}
        \xi\bigl(\widetilde{e}_{m_1} \comp{k,n}{L1} \widetilde{e}_{m_2}, [u]\comp{k,n}{TX}[v]\bigr)
        &= \xi\bigl((\widetilde{e}_{m_1},[u]) \comp{k,n}{LX} (\widetilde{e}_{m_2},[v])\bigr)\\
        &= \xi\bigl(\widetilde{e}_{m_1},[u]\bigr) \comp{k,n}{X} \xi\bigl(\widetilde{e}_{m_2},[v]\bigr)\\
        &= u\comp{k,n}{X}v.\qedhere
    \end{align*}
\end{remark}

For later reference, we remark on the compatibility of identity and binary composition operations and the construction of the \emph{hom weak $\omega$-category} $X(x,y)$ of a weak $\omega$-category $X$ between objects $x,y\in X_0$ defined in \cite[Section~9.3]{Leinster_book} and \cite{Cottrell_Fujii_hom}, and recalled in \cite[Section~2.5]{FHM1}. Just as \cite[Corollary~2.5.6]{FHM1}, the following is an immediate consequence of \cite[Proposition~2.5.5]{FHM1}.

\begin{corollary}
\label{cor:composition-in-hom}
    Let $X$ be a weak $\omega$-category and $x,y\in X_0$. 
    \begin{enumerate}
        \item Given natural numbers $m,n$ with $m< n$ and an $m$-cell $z$ of $X(x,y)$, we have
        \[
        \id{n}{X(x,y)}{z}=\id{n+1}{X}{z}.
        \]
        \item  Given natural numbers $k,m_1,m_2,n$ with $k<m_1\leq n$ and $k<m_2\leq n$, an $m_1$-cell $u$ of $X(x,y)$, and an $m_2$-cell $v$ of $X(x,y)$ such that $t_k^{X(x,y)}(u)=s_k^{X(x,y)}(v)$, we have
        \[
        u\comp{k,n}{X(x,y)}v=u\comp{k+1,n+1}{X}v.
        \]
    \end{enumerate}
\end{corollary}

\subsection{Invertible cells}
In this subsection, we recall the main results from our previous work \cite{FHM1} concerning \emph{invertible} cells in the following sense.

\begin{definition}[{\cite{Cheng_dual,Lafont_Metayer_Worytkiewicz_folk_model_str_omega_cat}}]
\label{def:invertible}
    An $n$-cell $u \colon x \to y$ (with $n \ge 1$) in a weak $\omega$-category $X$ is \emph{invertible} if there exist
    \begin{itemize}
        \item an $n$-cell $\check u \colon  y \to x$,
        \item an invertible $(n+1)$-cell $p\colon u\comp{n-1}{}\check u\to \id{n}{}{x}$, and
        \item an invertible $(n+1)$-cell $q\colon \check u\comp{n-1}{}u\to \id{n}{}{y}$
    \end{itemize}
    in $X$.
    In this situation, we say that $\check u$ is an \emph{inverse} of $u$.
    For $n$-cells $x$ and $y$ (with $n\geq 0$), we write $x \sim y$ if there exists an invertible $(n+1)$-cell $u\colon x \to y$.
\end{definition}

\begin{remark}
    In \cite{FHM1}, invertible cells are called \emph{weakly invertible} cells and inverses \emph{pseudo inverses}. Since in this paper we do not consider strictly invertible cells, we adopt the above more concise terminology.
\end{remark}

\begin{remark}
\label{rmk:coinduction}
The above notion of invertible cell in a weak $\omega$-category $X$ is defined \emph{coinductively}. 
Here we explain what this means in more detail and derive a useful proof principle, following \cite[Remark 3.1.2]{FHM1}.

We start with a general observation. 
Let $L$ be a complete lattice and $\Psi\colon L\to L$ a monotone map. 
We say that an element $s\in L$ is a \emph{post-fixed point} of $\Psi$ if $s\leq \Psi(s)$ holds. The set $\mathrm{Post}(\Psi)=\{\,s\in L\mid s\leq \Psi (s)\,\}$ of all post-fixed points of $\Psi$ is closed under joins in $L$, and hence is also a complete lattice. In particular, $\mathrm{Post}(\Psi)$ has a greatest element $t$. 
(In fact, $t$ is a \emph{fixed point} of $\Psi$, i.e., $t=\Psi(t)$ holds, since we have $\Psi(t)\in\mathrm{Post}(\Psi)$.) 
Notice that to show $s\leq t$ for some $s\in L$, it suffices to show that $s$ is a post-fixed point of $\Psi$.
The element $t$ is called the greatest (post-)fixed point of $\Psi$, and is denoted by $\nu\Psi$.

Now let $X$ be a weak $\omega$-category. \cref{def:invertible} can be understood by means of the monotone map $\Phi^X\colon \mathcal{P}\bigl(\coprod_{n\in\mathbb{N}} X_n\bigr)\to \mathcal{P}\bigl(\coprod_{n\in\mathbb{N}} X_n\bigr)$ on the powerset lattice $\mathcal{P}\bigl(\coprod_{n\in\mathbb{N}} X_n\bigr)$ of the set of all cells of $X$, defined by 
\begin{multline*}
\Phi^X(S)=\bigl\{\,(u\colon x\to y)\in X_n\,\big\vert\, n\geq 1,\ \ \exists (\check u\colon y\to x)\in X_n,\\
\exists \bigl({p}\colon u\comp{n-1}{}\check u\to \id{n}{}{x}\bigr)\in S\cap X_{n+1},\ \ \exists \bigl({q}\colon \check u\comp{n-1}{}u\to \id{n}{}{y}\bigr)\in S\cap X_{n+1}\,\bigr\}
\end{multline*}
for each $S\subseteq\coprod_{n\in\mathbb{N}} X_n$. 
In words, $\Phi^X(S)$ is the set of all cells in $X$ which are ``invertible up to $S$.''
We define the set of all invertible cells of $X$ to be the greatest (post-)fixed point $\nu\Phi^X$ of $\Phi^X$. 
It follows that if a set $W \subseteq \coprod_{n \in \N} X_n$ 
satisfies $W\subseteq \Phi^X(W)$, then all cells in $W$ are invertible. This is how we typically show invertibility of cells in a weak $\omega$-category.
\end{remark}

\begin{remark}[{\cite[Remark~3.1.3]{FHM1}}]
\label{rmk:hom-weak-omega-cat}
    Let $X$ be a weak $\omega$-category.
    By \cref{cor:composition-in-hom}, for $n\geq 2$, an $n$-cell $u$ of $X$ is invertible if and only if it is invertible as an $(n-1)$-cell in the hom weak $\omega$-category $X\bigl(s_0^X(u),t_0^X(u)\bigr)$.
\end{remark}

In \cite{FHM1}, we proved the following facts about these invertible cells.

\begin{proposition}[{\cite[Proposition 3.2.1]{FHM1}}]
\label{prop:strict-omega-functor-preserves-invertible-cells}
Any strict $\omega$-functor between weak $\omega$-categories preserves invertible cells.
\end{proposition}

\begin{proposition}[{Coherence, \cite[Proposition 3.2.5]{FHM1}}]
\label{cor:coherence}
Let $n\geq 0$ and $\phi,\phi'\in (L1)_n$ be parallel $n$-cells with $\ar(\phi)=\ar(\phi')=\kk\in (T1)_n$.
Then for any weak $\omega$-category $(X,\xi)$ and any pasting diagram $\uu$ of shape $\kk$ in $X$, the $(n+1)$-cell 
\[
\xi\Bigl(\kappa\bigl(\ppair{\phi,\phi'},\id{n+1}{T1}{\kk}\bigr),\id{n+1}{TX}{\uu}\Bigr)\colon \xi(\phi,\uu)\to\xi(\phi',\uu)
\]
in $X$ is invertible. 
In particular, we have $\xi(\phi,\uu)\sim\xi(\phi',\uu)$ in $X$. 
\end{proposition}
See \cref{rmk:id-in-TX} for a description of the cells $\id{n+1}{T1}{\kk}$ and $\id{n+1}{TX}{\uu}$ in the above proposition.

\begin{proposition}[{Unit Law, \cite[Proposition 3.3.5]{FHM1}}]
    \label{lem:unit-law}
    Let $X$ be a weak $\omega$-category, and let $u \colon x \to y$ be an $n$-cell in $X$ with $n \ge 1$.
    Then we have
    \[
    \id{n}{X}{x} \comp{n-1}{X} u \sim u \sim u \comp{n-1}{X} \id{n}{X}{y}.
    \]
\end{proposition}

\begin{theorem}[{\cite[Theorem 3.3.7]{FHM1}}]\label{thm:pasting-invertible-cells}
    Let $(X,\xi)$ be a weak $\omega$-category and let $(\phi,\uu) \in (LX)_n$ with $n\geq 1$ and
    \[
    \ar(\phi) = (T!)(\uu) = \begin{bmatrix}
    k_0 & & \dots & & k_r\\
    & \underline k_1 & \dots & \underline k_r &
    \end{bmatrix}.
    \]
    Suppose that, for each $0 \le i \le r$ with $k_i = n$, the $n$-cell $u_i$ appearing in $\uu$ is invertible.
    Then $\xi(\phi,\uu)$ is invertible.
\end{theorem}

\begin{corollary}[{\cite[Corollary 3.3.15]{FHM1}}]
\label{cor:eq-rel}
    Let $X$ be a weak $\omega$-category. Then $\sim$ is an equivalence relation on the set of cells of $X$. 
\end{corollary}

\begin{corollary}[{\cite[Corollary 3.3.17]{FHM1}}]
\label{cor:invariance}
    Let $X$ be a weak $\omega$-category, $n\geq 1$, and $u,v\colon x\to y$ be a parallel pair of $n$-cells in $X$ such that $u \sim v$.
    Suppose that $u$ is invertible.
    Then $v$ is invertible too.
\end{corollary}

\begin{remark}
    There are two kinds of approaches to infinite-dimensional categorical structures, namely \emph{inductive} and \emph{coinductive} ones.
    The purpose of this remark is to address why we adopt the latter instead of the former in this paper; if the reader is not already familiar with this dichotomy, they may safely skip to the next section.

    The main difference between the two kinds of approaches is that, in an inductive setting, the notion of ``invertible cell'' is a primitive one.
    Implemented in our setting (analogously to the case of \emph{strict} $\omega$-categories treated in \cite{Henry_Loubaton_inductive}), this translates to considering pairs $(X,E)$ where $X$ is a weak $\omega$-category in the sense of \cref{def:weak-omega-cats} and $E \subseteq \coprod_{n \ge 1}X_n$ is a set of \emph{marked} cells, which are to be thought of as ``abstract invertible cells.''
    Of course the set $E$ must satisfy some conditions in order to present a reasonable notion of invertibility, and our best guess for the ``correct'' conditions (cf.~\cite[Definition 2.15 and Theorem 3.38]{Henry_Loubaton_inductive}) is that
    \begin{enumerate}
        \item \label{cond:closed-under-composition} $E$ is ``closed under all pastings'' in the sense analogous to \cref{thm:pasting-invertible-cells}, and
        \item \label{cond:saturated} $\Phi^X(E) = E$ holds.
    \end{enumerate}
    Note that (\ref{cond:closed-under-composition}) in particular implies that all identities and coherence cells 
    (as in \cref{cor:coherence}) 
    are always marked.
    Both directions of (\ref{cond:saturated}) are non-trivial, where $E \subseteq \Phi^X(E)$ says that each marked cell is coinductively invertible (with all witness cells also marked), and the inclusion $\Phi^X(E) \subseteq E$ may be regarded as a sort of \emph{completeness} (as in Segal spaces) or \emph{saturation} (as in complicial sets).

    In fact, we suspect that it is possible to prove (though with additional non-trivial work) marked analogues of the results in \cref{sec:2-out-of-3} and possibly those in \cref{sec:weak-weak} after making the following modifications.
    \begin{itemize}
        \item The relation ``$x \sim y$'' should be interpreted as ``there exists a \emph{marked} (as opposed to \emph{invertible}) cell $x \to y$.''
        \item In \cref{def:omega-weak-eq}, define an \emph{$\omega$-weak equivalence} to be a marking-preserving strict $\omega$-functor that is essentially $\omega$-surjective (in the sense of \cref{def:ess-surj} but with the modified meaning of ``$x \sim y$'') and further reflects marked cells (cf.\ \cite[Proposition 3.33]{Henry_Loubaton_inductive} and \cref{lem:weak-eq-refl-inv}).
        \item In \cref{defn:TheQ}, replace the adjunction with the suitable marked variant, add to $\mathcal{I'}$ the inclusion of the standard (unmarked) $n$-globe into the standard marked $n$-globe for $n \ge 1$, and add their images under the left adjoint to $\mathcal{I}$. 
        (These modifications make sure that the corresponding ``trivial fibrations'' reflect marked cells.)
    \end{itemize}

    The caveat is that, if one is to actually pursue the inductive approach, one must consider pairs $(X,E)$ which do not necessarily satisfy the condition (\ref{cond:saturated}) in order to obtain a well-behaved category so that e.g., the adjunction in the third bullet point above actually exists.
    (Whether one chooses to impose (\ref{cond:closed-under-composition}) does not matter too much in this respect.)
    It is expected that one can then single out those pairs that do satisfy ((\ref{cond:closed-under-composition}) and) (\ref{cond:saturated}) as the fibrant objects using a left semi-model structure analogous to that in 
    \cite[Theorem~2.43]{Henry_Loubaton_inductive}.
    Since our arguments below rely on (\ref{cond:saturated}) (or more precisely, an analogue of \emph{prefibrancy} in the sense of \cite[Definition~3.18]{Henry_Loubaton_inductive}), the best we can prove in the marked setting is the 2-out-of-3 property only for $\omega$-weak equivalences between putative \emph{fibrant} objects 
    (cf.\ \cite[Proposition~3.33]{Henry_Loubaton_inductive}).
    In fact, for the above definition of $\omega$-weak equivalence, all three parts of the 2-out-of-3 property (\ref{item:A}, \ref{item:B}, and \ref{item:C} in the next section) can be shown to fail between not-necessarily-fibrant marked weak $\omega$-categories;
    it seems to us that the ``correct'' notion of weak equivalence should be defined, and also the 2-out-of-3 property for that class should be proved, using a more abstract means as done in \cite{Henry_Loubaton_inductive}.
    Thus, instead of generalising our work to include such partial results, we simply leave a detailed development of the inductive approach for future work.
\end{remark}

\section{The 2-out-of-3 property for \texorpdfstring{$\omega$}{ω}-weak equivalences}
\label{sec:2-out-of-3}
In this section, we define and study the class of \emph{$\omega$-weak equivalences} between weak $\omega$-categories.
In particular, we show that this class enjoys the \emph{2-out-of-3 property}; that is, for strict $\omega$-functors $f\colon X\to Y$ and $g\colon Y\to Z$ between weak $\omega$-categories, if two of $f$, $g$, and $gf$ are $\omega$-weak equivalences, then so is the third.
In other words, it is the conjunction of the following three statements.
\begin{itemize}
    \labeleditem{(A)}
    \label{item:A} If $f$ and $g$ are $\omega$-weak equivalences, then so is $gf$.
    \labeleditem{(B)}
    \label{item:B} If $g$ and $gf$ are $\omega$-weak equivalences, then so is $f$.
    \labeleditem{(C)}
    \label{item:C} If $f$ and $gf$ are $\omega$-weak equivalences, then so is $g$. 
\end{itemize}

After defining $\omega$-weak equivalences, we prove the easy two thirds (\ref{item:A} and \ref{item:B}) of the 2-out-of-3 property in \cref{subsec:def-omega-weak-eq}. 
In order to prove the other third \ref{item:C} of the 2-out-of-3 property, we study whiskering maps in a weak $\omega$-category in \cref{subsec:whiskering}, which culminates in a proof of a key intermediate result (\cref{lem:invertible-u-ess-surj}).
Using this, we complete our proof of the 2-out-of-3 property in \cref{subsec:2-out-of-3}.
Finally, we mention some additional properties of $\omega$-weak equivalences in \cref{subsec:other-properties-omega-weak-eq}. 

\subsection{Essential \texorpdfstring{$n$}{n}-surjectivity and \texorpdfstring{$\omega$}{ω}-weak equivalences}
\label{subsec:def-omega-weak-eq}
The following definition concerns arbitrary globular maps between weak $\omega$-categories (and not just strict $\omega$-functors) because we shall apply it also to e.g.\ \emph{whiskering maps}, which are not strict $\omega$-functors in general. 

\begin{definition}[Cf.~{\cite[Definition~4]{Baez-Shulman}}]
\label{def:ess-surj}
    Let $X$ and $Y$ be weak $\omega$-categories and $f\colon X\to Y$ be a globular map between the underlying globular sets of $X$ and $Y$.
    \begin{itemize}
        \item $f$ is \emph{essentially $0$-surjective} if for each $y\in Y_0$ there exists $x\in X_0$ such that $fx\sim y$. 
        \item $f$ is \emph{essentially $(n+1)$-surjective} ($n\in\N$) if $f$ is essentially $0$-surjective and for each $x,x'\in X_0$, the induced globular map $f_{x,x'}\colon X(x,x')\to Y(fx,fx')$ is essentially $n$-surjective.
        \item $f$ is \emph{essentially $\omega$-surjective} if it is essentially $n$-surjective for all $n\in\N$.\qedhere
    \end{itemize}
\end{definition}
Unravelling the induction, a globular map $f\colon X\to Y$ between weak $\omega$-categories is essentially $n$-surjective ($n\geq 0$) if and only if
\begin{itemize}
    \item for each $y\in Y_0$, there exists $x\in X_0$ such that $fx\sim y$, and
    \item for any $1\leq k\leq n$, any parallel pair $(u,v)$ of $(k-1)$-cells in $X$, and any $k$-cell $w\colon fu\to fv$ in $Y$, there exists a $k$-cell $\overline w\colon u\to v$ in $X$ such that $f\overline w\sim w$. (The case $k=2$ may be visualised as below.)
\end{itemize}
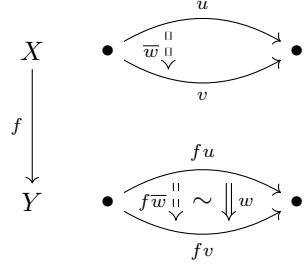
\[
\label{eqn:omega-weak-eq-diagram}
    \begin{tikzpicture}[baseline=-\the\dimexpr\fontdimen22\textfont2\relax ]
      \node(0) at (0,1) {$\bullet$};
      \node(1) at (2.5,1) {$\bullet$};
      \draw [->,bend left=30]  (0) to node[auto, labelsize] {$u$} (1);
      \draw [->,bend right=30] (0) to node[auto, swap,labelsize] {$v$} (1);
      \draw [2cell,dashed]  (0.8,1.25) to node[auto,swap,labelsize] {$\overline w$} (0.8,0.75);
      \node(2) at (0,-1) {$\bullet$};
      \node(3) at (2.5,-1) {$\bullet$};
      \draw [->,bend left=30]  (2) to node[auto, labelsize] {$fu$} (3);
      \draw [->,bend right=30] (2) to node[auto, swap,labelsize] {$fv$} (3);
      \draw [2cell,dashed]  (0.9,-0.75) to node[auto,swap,labelsize] {$f\overline w$} (0.9,-1.25);
      \draw [2cell]  (1.6,-0.75) to node[auto,labelsize] {$w$} (1.6,-1.25);
      \node at (1.25,-1) {$\sim$};
      \node(x) at (-1,1) {$X$};
      \node(y) at (-1,-1) {$Y$};
      \draw [->]  (x) to node[auto, swap, labelsize] {$f$} (y);
\end{tikzpicture}
\]

For example, if we regard a functor $f \colon X \to Y$ between ordinary categories as a strict $\omega$-functor between weak $\omega$-categories (here we define the underlying globular set of a category $X$ by setting $s^X_k,t^X_k\colon X_{k+1}\to X_k$ to be the identity function for all $k\geq 1$), then $f$ is
\begin{itemize}
    \item essentially $0$-surjective if and only if it is essentially surjective on objects (in the usual sense),
    \item essentially $1$-surjective if and only if it is essentially surjective on objects and full, and
    \item essentially $2$-surjective if and only if it is essentially surjective on objects, full, and faithful (because the existence of $w$ as above is equivalent to $fu=fv$, and similarly for $\overline{w}$), if and only if it is essentially $\omega$-surjective.
\end{itemize}

\begin{definition}
\label{def:omega-weak-eq}
    An \emph{$\omega$-weak equivalence} is an essentially $\omega$-surjective strict $\omega$-functor between weak $\omega$-categories.
\end{definition}

It is easy to see from the above rephrasing of essential $n$-surjectivity that, when restricted to those between \emph{strict} $\omega$-categories, our definition of $\omega$-weak equivalence recovers precisely \cite[Definition~4.7]{Lafont_Metayer_Worytkiewicz_folk_model_str_omega_cat}.

Notice that a globular map $f\colon X\to Y$ between weak $\omega$-categories is essentially $\omega$-surjective if and only if it is essentially $0$-surjective and the induced globular map $f_{x,x'}\colon X(x,x')\to Y(fx,fx')$ is essentially $\omega$-surjective for each $x,x'\in X_0$. The following is an immediate consequence of this fact.
\begin{proposition}
\label{prop:omega-weak-eq-coinductively}
    A strict $\omega$-functor $f\colon X\to Y$ between weak $\omega$-categories is an $\omega$-weak equivalence if and only if it is essentially $0$-surjective and the induced strict $\omega$-functor $f_{x,x'}\colon X(x,x')\to Y(fx,fx')$ is an $\omega$-weak equivalence for each $x,x'\in X_0$.
\end{proposition}

\begin{remark}
\label{rmk:omega-weak-eq-coinductively}
    The statement of \cref{prop:omega-weak-eq-coinductively} can be taken as an alternative, \emph{coinductive} definition of $\omega$-weak equivalences (see \cref{rmk:coinduction} for details of coinduction). More precisely, define the monotone map 
    \[\Psi\colon \mathcal{P}\bigl(\mathrm{mor}(\WkCats{\omega})\bigr)\to\mathcal{P}\bigl(\mathrm{mor}(\WkCats{\omega})\bigr)\]
    on the powerset lattice $\mathcal{P}\bigl(\mathrm{mor}(\WkCats{\omega})\bigr)$ of the set $\mathrm{mor}(\WkCats{\omega})$ of all strict $\omega$-functors between weak $\omega$-categories as follows. It maps each $S\subseteq \mathrm{mor}(\WkCats{\omega})$ to the set $\Psi(S)$ of all strict $\omega$-functors which are essentially $0$-surjective and locally in $S$, i.e.,
    \begin{multline*}
    \Psi(S)=\bigl\{\, (f\colon X\to Y)\in\mathrm{mor}(\WkCats{\omega})  \,\big\vert\, \text{$f$ is essentially $0$-surjective and}\\
    \text{we have $\bigl(f_{x,x'}\colon X(x,x')\to Y(fx,fx')\bigr)\in S$ for each $x,x'\in X_0$}\,\bigr\}.
    \end{multline*}
    Then the set $E$ of all $\omega$-weak equivalences is the largest (post-)fixed point $\nu\Psi$ of $\Psi$. Indeed,
    \cref{prop:omega-weak-eq-coinductively} says that we have $E=\Psi(E)$, which implies $E\subseteq \nu\Psi$.
    On the other hand, given any $S\subseteq \mathrm{mor}(\WkCats{\omega})$ with $S\subseteq \Psi(S)$, it is easy to see that all strict $\omega$-functors in $S$ are essentially $n$-surjective, by induction on $n\in\N$. 
    Thus we have $S\subseteq E$, and hence $E$ is the largest (post-)fixed point $\nu\Psi$ of $\Psi$.
\end{remark}

We now prove two thirds of the 2-out-of-3 property. We prove these under rather weak assumptions, partly to illuminate the properties we actually use, and partly because we use some of the extra generality later.

The following implies \ref{item:A} of the 2-out-of-3 property by \cref{prop:strict-omega-functor-preserves-invertible-cells}.
\begin{lemma}
\label{prop:2-out-of-3-gf}
    Let $n\in\N$, $X,Y,Z$ be weak $\omega$-categories, $f \colon X \to Y$ be a globular map, and $g \colon Y \to Z$ be a globular map preserving invertible cells.  
    If $f$ and $g$ are essentially $n$-surjective, then so is $gf$. 
\end{lemma}
\begin{proof}
    We prove this by induction on $n$. 

    The base case $n=0$ is shown as follows. Take any $z\in Z_0$. Then since $g$ is essentially $0$-surjective, there exists $y\in Y_0$ with $gy\sim z$. Since $f$ is essentially $0$-surjective, there exists $x\in X_0$ with $fx\sim y$. Then we have $gfx\sim gy\sim z$ since $g$ preserves invertible cells.

    For the inductive step, let $n>0$. For each $x,x'\in X_0$, we have $(gf)_{x,x'}=g_{fx,fx'}\circ f_{x,x'}$. Here, $f_{x,x'}$ and $g_{fx,fx'}$ are essentially $(n-1)$-surjective, and $g_{fx,fx'}$ preserves invertible cells. Hence $(gf)_{x,x'}$ is essentially $(n-1)$-surjective by the inductive hypothesis.
\end{proof}

The following implies \ref{item:B} of the 2-out-of-3 property, by \cref{lem:weak-eq-refl-inv} below.
\begin{lemma}
\label{prop:2-out-of-3-f}
    Let $n\in\N$, $X,Y,Z$ be weak $\omega$-categories, $f \colon X \to Y$ be a globular map, and $g \colon Y \to Z$ be a globular map reflecting invertible cells.
    If $g$ is essentially $(n+1)$-surjective and $gf$ is essentially $n$-surjective, then $f$ is essentially $n$-surjective.
\end{lemma}
\begin{proof}
    We prove this by induction on $n$.

    The base case $n=0$ is shown as follows. Take any $y\in Y_0$. Then we obtain $gy\in Z_0$. Since $gf$ is essentially $0$-surjective, there exist $x\in X_0$ and an invertible $1$-cell $u\colon gfx\to gy$ in $Z$. Since $g$ is essentially $1$-surjective, there exists a $1$-cell $\overline u\colon fx\to y$ in $Y$ with $g\overline u\sim u$. By \cref{cor:invariance}, $g\overline u$ is invertible. Since $g$ reflects invertible cells, $\overline u$ is invertible. This shows $f$ is essentially $0$-surjective. 

    For the inductive step, let $n>0$. For each $x,x'\in X_0$, we have $(gf)_{x,x'}=g_{fx,fx'}\circ f_{x,x'}$. Here, $g_{fx,fx'}$ is essentially $n$-surjective and reflects invertible cells, and $(gf)_{x,x'}$ is essentially $(n-1)$-surjective. Hence  $f_{x,x'}$ is essentially $(n-1)$-surjective by the inductive hypothesis.
\end{proof}

\begin{proposition}
    \label{lem:weak-eq-refl-inv}
    Any $\omega$-weak equivalence reflects invertible cells.
    More generally, let $f \colon X \to Y$ be a globular map between (the underlying globular sets of) weak $\omega$-categories $X$ and $Y$.
    Suppose that
    \begin{itemize}
        \item[(i)] $f$ is essentially $\omega$-surjective,
        \item[(ii)] for each $n$-cell $x$ in $X$ with $n \ge 0$, there exists an invertible $(n+2)$-cell $\id{n+1}{Y}{fx} \sim f\bigl(\id{n+1}{X}{x}\bigr)$ in $Y$, and
        \item[(iii)] for each pair of $n$-cells $u \colon x \to x'$ and $u' \colon x' \to x''$ in $X$ composable along the $(n-1)$-dimensional boundary, with $n \ge 1$, there exists an invertible $(n+1)$-cell $f(u \comp{n-1}{X} u') \sim fu \comp{n-1}{Y} fu'$ in $Y$.
    \end{itemize}
    Then $f$ reflects invertible cells.
\end{proposition}

\begin{proof}
    Let $f \colon X \to Y$ be a globular map between weak $\omega$-categories satisfying (i)--(iii).
    Let 
    \[
    W = \biggl\{u \in \coprod_{n\geq 1}X_n \bigg| fu\text{ is invertible in }Y\biggr\}.
    \]
    It suffices to show $W \subseteq \Phi^X(W)$ (cf.~\cref{rmk:coinduction}).
    
    Let $u \colon x \to x'$ be an $n$-cell in $W$.
    Then there exist an $n$-cell $v \colon fx' \to fx$ and invertible $(n+1)$-cells $p \colon fu\comp{n-1}{Y}v \to \id{n}{Y}{fx}$ and $q \colon v\comp{n-1}{Y}fu \to \id{n}{Y}{fx'}$ in $Y$.
    By (i), there exists $\overline v \colon x' \to x$ in $X$ such that $f \overline v \sim v$ in $Y$.
    Let $p'$ be a composite of
    \[
    f(u\comp{n-1}{X}\overline v) \sim fu\comp{n-1}{Y}f \overline v \sim fu\comp{n-1}{Y}v \xrightarrow{p} \id{n}{Y}{fx} \sim f\bigl(\id{n}{X}{x}\bigr)
    \]
    where the first and the last factors are instances of (iii) and (ii) respectively, and the second factor is invertible by \cref{thm:pasting-invertible-cells}. 
    It follows from (i) that there exists an $(n+1)$-cell $\overline p \colon u\comp{n-1}{X}\overline v \to \id{n}{X}{x}$ such that $f \overline p \sim p'$.
    Note that $f \overline p$ is invertible in $Y$ by \cref{thm:pasting-invertible-cells} and \cref{cor:invariance}, so $\overline p \in W$.
    Similarly, there is an $(n+1)$-cell $\overline q \colon \overline v\comp{n-1}{X}u \to \id{n}{X}{x'}$ in $W$ that is sent by $f$ to a composite of
    \[
    f(\overline v\comp{n-1}{X}u) \sim f \overline v\comp{n-1}{Y}fu \sim v\comp{n-1}{Y} fu \xrightarrow{q} \id{n}{Y}{fx'} \sim f\bigl(\id{n}{X}{x'}\bigr)
    \]
    modulo the equivalence relation $\sim$.
    This completes the proof.
\end{proof}

We conclude this subsection with an informal explanation of why the remaining third \ref{item:C} of the 2-out-of-3 property is much more difficult to prove than \ref{item:A} and \ref{item:B}.

\begin{remark}
The difficulty of proving \ref{item:C} lies in the fact that $\omega$-weak equivalences are defined in terms of \emph{essential} surjectivity rather than actual surjectivity (the latter gives the definition of \emph{trivial fibrations} (\cref{def:alg-triv-fib,defn:TheQ})). In particular, the $\omega$-weak equivalence $f$ may not be surjective on the nose, and one is led to consider the following sort of diagram in an attempt to prove \ref{item:C}:
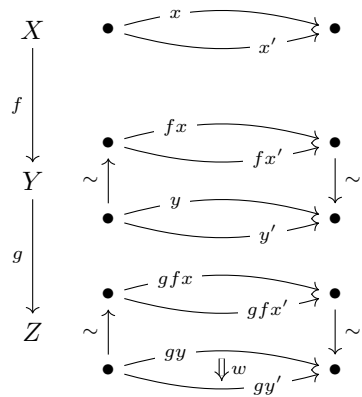
\[
    \begin{tikzpicture}[baseline=-\the\dimexpr\fontdimen22\textfont2\relax ]
      \node(0a) at (0,2) {$\bullet$};
      \node(1a) at (3,2) {$\bullet$};
      \draw [->,bend left=15]  (0a) to node[near start,fill=white, labelsize] {$x$} (1a);
      \draw [->,bend right=15] (0a) to node[swap,near end,fill=white,labelsize] {$x'$} (1a);
      \node(2) at (0,-0.5) {$\bullet$};
      \node(3) at (3,-0.5) {$\bullet$};
      \node(2a) at (0,0.5) {$\bullet$};
      \node(3a) at (3,0.5) {$\bullet$};
      \draw [->,bend left=15]  (2) to node[near start,fill=white,labelsize] {$y$} (3);
      \draw [->,bend right=15] (2) to node[near end,fill=white, swap,labelsize] {$y'$} (3);
      \draw [->,bend left=15]  (2a) to node[near start,fill=white, labelsize] {$fx$} (3a);
      \draw [->,bend right=15] (2a) to node[swap,near end,fill=white,labelsize] {$fx'$} (3a);
      \draw [<-]  (2a) to node[auto,swap,labelsize] {$\sim$} (2);
      \draw [<-]  (3) to node[auto,swap,labelsize] {$\sim$} (3a);
      \node(4) at (0,-2.5) {$\bullet$};
      \node(5) at (3,-2.5) {$\bullet$};
      \node(4a) at (0,-1.5) {$\bullet$};
      \node(5a) at (3,-1.5) {$\bullet$};
      \draw [->,bend left=15]  (4) to node[near start,fill=white,labelsize] {$gy$} (5);
      \draw [->,bend right=15] (4) to node[near end,fill=white, swap,labelsize] {$gy'$} (5);
      \draw [->,bend left=15]  (4a) to node[near start,fill=white, labelsize] {$gfx$} (5a);
      \draw [->,bend right=15] (4a) to node[swap,near end,fill=white,labelsize] {$gfx'$} (5a);
      \draw [2cell]  (1.5,-2.35) to node[auto,labelsize] {$w$} (1.5,-2.65);
      \draw [<-]  (4a) to node[auto,swap,labelsize] {$\sim$} (4);
      \draw [<-]  (5) to node[auto,swap,labelsize] {$\sim$} (5a);
      \node(x) at (-1,2) {$X$};
      \node(y) at (-1,0) {$Y$};
      \node(z) at (-1,-2) {$Z$};
      \draw [->]  (x) to node[auto, swap, labelsize] {$f$} (y);
      \draw [->]  (y) to node[auto, swap, labelsize] {$g$} (z);
\end{tikzpicture}
    \]
That is, starting with a parallel pair of $(n-1)$-cells $(y,y')$ in $Y$ and an $n$-cell $w\colon gy\to gy'$ in $Z$, one would hope to use the assumption that $gf$ is an $\omega$-weak equivalence to obtain some $v\colon y\to y'$ with $gv\sim w$. However, this is not immediately possible since $(y,y')$ may not be of the form $(fx,fx')$ for any parallel pair of $(n-1)$-cells $(x,x')$ in $X$. Instead, the assumption that $f$ is an $\omega$-weak equivalence yields cylinder-like diagrams as above. 

In the case of \emph{strict} $\omega$-categories \cite{Lafont_Metayer_Worytkiewicz_folk_model_str_omega_cat}, the authors construct a globular set $\Gamma Y$ of \emph{cylinders} in $Y$ in order to prove \ref{item:C}.\footnote{They then endow it with a suitable strict $\omega$-category structure, but this is not strictly necessary for proving \ref{item:C}.}
We avoid the use of $\Gamma Y$ by essentially unravelling their \emph{transport} construction \cite[Lemma 4.22]{Lafont_Metayer_Worytkiewicz_folk_model_str_omega_cat} into a series of \emph{whiskerings} (introduced in the next subsection) and interleaving them into a suitable inductive argument.
This modification repackages the proof, but we still need similar ingredients to those used in \cite{Lafont_Metayer_Worytkiewicz_folk_model_str_omega_cat}.
In particular, we need a weak $\omega$-categorical counterpart of \cite[Lemma~4.6]{Lafont_Metayer_Worytkiewicz_folk_model_str_omega_cat}, which is \cref{lem:invertible-u-ess-surj} below.
\end{remark}

\subsection{Whiskering with an invertible cell}
\label{subsec:whiskering}
The purpose of this subsection is to establish a key intermediate result (\cref{lem:invertible-u-ess-surj}) for proving \ref{item:C}, which concerns \emph{whiskering maps} defined as follows.

\begin{definition}
    For each weak $\omega$-category $X$, object $z\in X_0$, and 1-cell $u\colon x\to y$ in $X$, the \emph{whiskering map} 
\[
    X(u,z)=u\comp{0}{X}(-)\colon X(y,z)\to X(x,z)
\]
is defined by mapping each $n$-cell $v$ of $X(y,z)$ to the $n$-cell $u\comp{0}{X}v$ of $X(x,z)$.  Dually, for each object $x\in X_0$ and 1-cell $v\colon y\to z$ in $X$, we have the whiskering map
\[
X(x,v)=(-)\comp{0}{X}v\colon X(x,y)\to X(x,z).\qedhere
\]
\end{definition}

\begin{remark}
\label{rmk:whiskering-preserves-invertible-cells}
    Let $X$ be a weak $\omega$-category, $z\in X_0$, and $u\colon x\to y$ a $1$-cell of $X$.
    Then $X(u,z)\colon X(y,z)\to X(x,z)$ is a globular map by \cref{prop:s-and-t-for-id-and-ast}, but is not necessarily a strict $\omega$-functor unless $X$ is a \emph{strict} $\omega$-category.\footnote{In fact, we suspect that $X(u,z)$ is the \emph{underlying globular map} of a weak $\omega$-functor in the sense of \cref{underlying}, but this is beyond the scope of this paper.}
    It also preserves invertible cells by \cref{thm:pasting-invertible-cells} and \cref{rmk:hom-weak-omega-cat}. 
\end{remark}

\cref{lem:invertible-u-ess-surj} below states that the whiskering map $X(u,z)\colon X(y,z)\to X(x,z)$ is essentially $\omega$-surjective whenever the 1-cell $u\colon x\to y$ is invertible in $X$.
We note that its proof constitutes the most technical part of this paper, and if the reader is willing to take this theorem for granted, they may safely skip the rest of this subsection.

The proof of 
\cref{lem:invertible-u-ess-surj}
turns out to require the following notion of \emph{essential $\omega$-injectivity}.

\begin{definition}\label{def:ess-inj}
    Let $X$ and $Y$ be weak $\omega$-categories and $f\colon X\to Y$ be a globular map between the underlying globular sets of $X$ and $Y$.
    \begin{itemize}
        \item $f$ is \emph{essentially $0$-injective} if, for each $x,x'\in X_0$, $fx\sim fx'$ implies $x\sim x'$.
        \item $f$ is \emph{essentially $(n+1)$-injective} ($n\in\N$) if $f$ is essentially $0$-injective and for each $x,x'\in X_0$, the induced globular map $f_{x,x'}\colon X(x,x')\to Y(fx,fx')$ is essentially $n$-injective.
        \item $f$ is \emph{essentially $\omega$-injective} if it is essentially $n$-injective for all $n\in\N$.\qedhere
    \end{itemize}
\end{definition}
See \cref{lem:surj-implies-inj} for a relationship between the notions of essential $n$-injectivity and essential $n$-surjectivity.
Unravelling the induction, $f$ is essentially $n$-injective ($n\geq 0$) if and only if, for any $0\leq k\leq n$ and any parallel pair $(u,v)$ of $k$-cells in $X$, $fu\sim fv$ in $Y$ implies $u\sim v$ in $X$.

Our proof of \cref{lem:invertible-u-ess-surj} proceeds in the following four steps.
\begin{itemize}
    \labeleditem{(i)}\label{step1} Prove that $X(u,z)\colon X(y,z)\to X(x,z)$ is essentially $\omega$-injective in the special case $u = \id{1}{X}{y}$.
    \labeleditem{(ii)}\label{step2} Extend \ref{step1} to the general case of an arbitrary invertible $1$-cell $u$.
    \labeleditem{(iii)}\label{step3} Prove that $X(u,z)\colon X(y,z)\to X(x,z)$ is essentially $\omega$-surjective in the special case $u = \id{1}{X}{y}$.
    \labeleditem{(iv)}\label{step4} Extend \ref{step3} to the general case of an arbitrary invertible $1$-cell $u$.
\end{itemize}

Before embarking on the actual proof, let us discuss why even \ref{step1} is non-trivial.
For example, given a parallel pair of $2$-cells
\[
\begin{tikzpicture}[baseline=-\the\dimexpr\fontdimen22\textfont2\relax ]
      \node(20) at (0,0) {$y$};
      \node(21) at (3,0) {$z$};

      \draw [->,bend left=30]  (20) to node[auto, labelsize] {$s$} (21);
      \draw [->,bend right=30] (20) to node[auto, swap,labelsize] {$t$} (21);

      \draw [2cell]  (1.2,0.3) to node[left,labelsize] {$u$} (1.2,-0.3);
      \draw [2cell]  (1.8,0.3) to node[right,labelsize] {$v$} (1.8,-0.3);
\end{tikzpicture}
\]
in a weak $\omega$-category $(X,\xi)$ with $\id{1}{}{y} \comp{0}{}u \sim \id{1}{}{y} \comp{0}{}v$, one must prove $u \sim v$.
One may be tempted to resort to the ``unit law'' and argue
\[
u \sim \id{1}{}{y} \comp{0}{}u \sim \id{1}{}{y} \comp{0}{}v \sim v,
\]
but in general $u \colon s \to t$ and $\id{1}{}{y} \comp{0}{}u \colon \id{1}{}{y} \comp{0}{}s \to \id{1}{}{y} \comp{0}{}t$ are not even parallel, so ``$u \sim \id{1}{}{y} \comp{0}{}u$'' does not make sense.
(Note that \cref{lem:unit-law} is not applicable here.)

A key observation here is that both $u$ and $\id{1}{}{y}\comp{0}{}u$ can be obtained from the same pasting diagram $[u]\in (TX)_2$ by pasting according to different pasting instructions, 
namely $\chi=\widetilde e_2\in (L1)_2$ and $\phi=\id{1}{L1}{\widetilde e_0} \comp{0}{L1}\widetilde e_2\in (L1)_2$, respectively. (See the proof of \cref{lem:i-ess-inj}.)
The following construction, which transforms a cell of the form
\[
w \colon \xi(\phi,\uu) \to \xi(\phi,\vv)
\]
(such as an invertible $3$-cell witnessing $\id{1}{}{y} \comp{0}{}u \sim \id{1}{}{y} \comp{0}{}v$) into one of the form
\[
w' \colon \xi(\chi,\uu) \to \xi(\chi,\vv)
\]
(such as an invertible $3$-cell witnessing $u \sim v$)
by inserting ``pads'' in each dimension, is designed to deal with this kind of situation.
Note that if $\phi$ and $\chi$ were parallel cells in $L1$, this could be achieved by means of suitable coherence cells (see \cref{cor:coherence}); $\id{1}{L1}{\widetilde e_0} \comp{0}{L1}\widetilde e_2$ and $\widetilde e_2$ are not parallel, however.

\begin{construction}\label{con:padding}
    Let $n \ge 1$ and suppose we are given $\phi,\chi \in (L1)_n$ such that $\ar(\phi) = \ar(\chi) = \kk$.
    Note that $\phi$ and $\chi$ are not necessarily parallel.
    Let $\phi_0 = \phi$. 
    We will construct a sequence of $n$-cells $\phi_1,\dots,\phi_n$ in $L1$ where each $\phi_i$ is to be thought of as ``$\phi$ after inserting pads of dimension $\le i$ so that it has the $(i-1)$-boundary of $\chi$.''
    The following diagram describes $\phi_2$ with $\phi,\chi\in(L1)_2$.
    \begin{equation}\label{eqn:padding}
\begin{tikzpicture}[baseline=-\the\dimexpr\fontdimen22\textfont2\relax ]
      \node(0) at (0,0) {$s_0^{L1}(\chi)$};
      \node(1) at (2,0) {$s_0^{L1}(\phi)$};
      \node(2) at (4,0) {$t_0^{L1}(\phi)$};
      \node(3) at (6,0) {$t_0^{L1}(\chi)$};
      \draw [->]  (0) to node[auto, labelsize] {$\lambda_1$}   (1); 
      \draw [->,bend left=50]  (1) to node[auto, labelsize] {$s_1^{L1}(\phi)$}   (2);      
      \draw [->,bend right=50] (1) to node[auto, swap,labelsize] {$t_1^{L1}(\phi)$}  (2); 
      \draw [->,bend left=60]  (0) to node[auto, labelsize] {$s_1^{L1}(\chi)$}   (3);   
      \draw [->,bend right=60]  (0) to node[auto, swap,labelsize] {$t_1^{L1}(\chi)$}   (3); 
      \draw [->]  (2) to node[auto, labelsize] {$\rho_1$}   (3); 
      \draw [->,2cell] (3,1.5) to node[auto,labelsize] {$\lambda_2$} (3,1.2);
      \draw [->,2cell] (3,0.3) to node[auto,labelsize] {$\phi$} (3,-0.3);
      \draw [->,2cell] (3,-1.2) to node[auto,labelsize] {$\rho_2$} (3,-1.5);
\end{tikzpicture}
\end{equation}
    
    For $1 \le i \le n$, we set:
    \begin{itemize}
        \item $\lambda_i = \kappa\bigl(\langle s^{L1}_{i-1}(\chi),s^{L1}_{i-1}(\phi_{i-1})\rangle,\ \id{i}{T1}{s^{T1}_{i-1}(\kk)}\bigr)\in (L1)_i$,
        \item $\rho_i = \kappa\bigl(\langle t^{L1}_{i-1}(\phi_{i-1}),t^{L1}_{i-1}(\chi)\rangle,\ \id{i}{T1}{t^{T1}_{i-1}(\kk)}\bigr)\in (L1)_i$, and
        \item $\phi_i = \lambda_i\comp{i-1}{L1}(\phi_{i-1}\comp{i-1}{L1}\rho_i)\in(L1)_n$.
    \end{itemize}
    To see that these expressions indeed define valid cells in $L1$, we must check
    \begin{itemize}
        \item $\ar(\phi_{i}) = \kk$,
        \item $s_{i-1}^{L1}(\phi_{i}) = s_{i-1}^{L1}(\chi)$, and
        \item $t_{i-1}^{L1}(\phi_{i}) = t_{i-1}^{L1}(\chi)$. 
    \end{itemize}
    The first condition holds because $\ar \colon L1 \to T1$, being a strict $\omega$-functor between the weak $\omega$-categories $(L1,\mu^L_1)$ and $(T1,\mu^T_1\circ \ar_{T1})$ (the latter is induced by the strict $\omega$-category $(T1,\mu^T_1)$), preserves $\comp{i-1}{}$ and sends $\lambda_i,\rho_i$ to identities.
    The second and third hold by construction of $\phi_{i}$ and \cref{prop:s-and-t-for-id-and-ast}.
    Note that, in particular, $\phi_{n}$ and $\chi$ are parallel $n$-cells in $L1$ with the same arity, so the above formulas for $\lambda_i$ and $\rho_i$ make sense for $i=n+1$ too: 
    \begin{itemize}
        \item $\lambda_{n+1} = \kappa\bigl(\langle\chi,\phi_{n}\rangle,\id{n+1}{T1}{\kk}\bigr)\in(L1)_{n+1}$ and
        \item $\rho_{n+1} = \kappa\bigl(\langle \phi_{n},\chi\rangle,\id{n+1}{T1}{\kk}\bigr)\in(L1)_{n+1}$.
    \end{itemize}

    Now suppose further that we are given a weak $\omega$-category $(X,\xi)$, a parallel pair of pasting diagrams $\uu, \vv \in (TX)_n$ of shape $\kk$, and an $(n+1)$-cell $w \colon\xi(\phi,\uu) \to \xi(\phi,\vv)$.
    Let $w_0 = w$.
    For $1 \le i \le n+1$, we set
    \begin{itemize}
        \item $\ell_i = \xi\bigl(\lambda_i,\id{i}{TX}{s_{i-1}^{TX}(\uu)}\bigr)\in X_i$,
        \item $r_i = \xi\bigl(\rho_i,\id{i}{TX}{t_{i-1}^{TX}(\vv)}\bigr)\in X_i$, and
        \item $w_i =\ell_i\comp{i-1}{X} (w_{i-1}\comp{i-1}{X}r_i)\in X_{n+1}$.
    \end{itemize}
    (When $i=n+1$, we interpret the expressions $s_{n}^{TX}(\uu)$ and $t_{n}^{TX}(\vv)$ in the definitions of $\ell_{n+1}$ and $r_{n+1}$ as $\uu$ and $\vv$, respectively.) We may inductively prove 
    that the expression in the definition of $w_i$ is well typed
    by checking $s_n^X(w_i) = \xi(\phi_i,\uu)$ and $t_n^X(w_i) = \xi(\phi_i,\vv)$ for $1 \le i \le n$.
    We only spell out the former:
    \begin{align*}
        s_n^X(w_i) &= s_n^X\bigl(\ell_i\comp{i-1}{X}(w_{i-1}\comp{i-1}{X}r_i )\bigr)
        \tag*{\text{(definition of $w_i$)}}\\
        &= \ell_i\comp{i-1}{X}\bigl(s_n^X(w_{i-1})\comp{i-1}{X}r_i\bigr)
        \tag*{\text{(\cref{prop:s-and-t-for-id-and-ast})}}\\
        &= \xi\bigl(\lambda_i,\id{i}{TX}{s_{i-1}^{TX}(\uu)}\bigr)\comp{i-1}{X}\Bigl(\xi(\phi_{i-1},\uu)\comp{i-1}{X}\xi\bigl(\rho_i,\id{i}{TX}{t_{i-1}^{TX}(\vv)}\bigr)\Bigr)
        \tag*{\text{(definitions of $\ell_i$ and $r_i$, and inductive hypothesis)}}\\
        &= \xi\bigl(\lambda_i,\id{i}{TX}{s_{i-1}^{TX}(\uu)}\bigr)\comp{i-1}{X}\Bigl(\xi(\phi_{i-1},\uu)\comp{i-1}{X}\xi\bigl(\rho_i,\id{i}{TX}{t_{i-1}^{TX}(\uu)}\bigr)\Bigr)
        \tag*{\text{($\uu$ and $\vv$ are parallel)}}\\
        &= \xi\Bigl(\bigl(\lambda_i,\id{i}{TX}{s_{i-1}^{TX}(\uu)}\bigr)\comp{i-1}{LX} \Bigl((\phi_{i-1},\uu)\comp{i-1}{LX}\bigl(\rho_i,\id{i}{TX}{t_{i-1}^{TX}(\uu)}\bigr)\Bigr)\Bigr)
        \tag*{\text{($\xi\colon (LX,\mu^L_X)\to (X,\xi)$ is a strict $\omega$-functor)}}\\
        &= \xi\bigl(\lambda_i\comp{i-1}{L1}(\phi_{i-1}\comp{i-1}{L1}\rho_i), \uu\bigr)
        \tag*{\text{(action of $\comp{LX}{i-1}$ is determined by $\comp{L1}{i-1}$ and $\comp{TX}{i-1}$; see \cite[Sections~2.2 and 2.5]{FHM1})}}\\
        &= \xi(\phi_i,\uu).
        \tag*{\text{(definition of $\phi_i$)}}
    \end{align*}
    We call the resulting $(n+1)$-cell $w_{n+1} \colon \xi(\chi,\uu) \to \xi(\chi,\vv)$ the \emph{$\phi$-to-$\chi$ padding of $w$} 
    (with respect to $(\uu,\vv)$).
    Note that all the pads ($\ell_i$ and $r_i$)
    are coherence cells (\cref{cor:coherence}), and hence
    are invertible in $X$.
    This in particular implies that, if $w$ is invertible then so is its $\phi$-to-$\chi$ padding (see \cref{thm:pasting-invertible-cells}).
\end{construction}
 
Using this padding construction, \ref{step1} is straightforward.

\begin{lemma}
\label{lem:i-ess-inj}
    Let $(X,\xi)$ be a weak $\omega$-category and $y,z \in X_0$.
    Then the whiskering map
    \[
    X\bigl(\id{1}{X}{y},z\bigr) \colon X(y,z) \to X(y,z)
    \]
    is essentially $\omega$-injective.
\end{lemma}

\begin{proof}
    Let $n\geq 1$ and $u,v \in X_n$ be parallel $n$-cells with $s_0^X(u) = s_0^X(v) = y$ and $t_0^X(u) = t_0^X(v) = z$.
    Suppose that there exists an invertible $(n+1)$-cell $w \colon \id{1}{X}{y}\comp{0}{X}u \to \id{1}{X}{y}\comp{0}{X}v$ in $X$.
    Observe that we have
    \[
    \id{1}{X}{y}\comp{0}{X}u = \xi\bigl(\id{1}{L1}{\widetilde{e}_0}\comp{0}{L1}\widetilde e_n, [u] \bigr)
    \quad\text{and}\quad
    \id{1}{X}{y}\comp{0}{X}v = \xi\bigl(\id{1}{L1}{\widetilde{e}_0}\comp{0}{L1}\widetilde e_n, [v] \bigr).
    \]
    For example, the former can be proved as follows: 
    \begin{align*}
        \id{1}{X}{y}\comp{0}{X}u 
        &= \xi\bigl(\id{1}{L1}{\widetilde e_0}, \id{1}{TX}{[y]}\bigr)\comp{0}{X}\xi(\widetilde e_n,[u])\\
        &= \xi\Bigl(\bigl(\id{1}{L1}{\widetilde e_0}, \id{1}{TX}{[y]}\bigr)\comp{0}{LX}(\widetilde e_n,[u])\Bigr)
        \tag*{\text{($\xi\colon (LX,\mu^L_X)\to (X,\xi)$ is a strict $\omega$-functor)}}\\
        &= \xi\bigl(\id{1}{L1}{\widetilde{e}_0}\comp{0}{L1}\widetilde e_n, [u] \bigr).
        \tag*{\text{(definitions of $\mu^L_X$ and $\mu^T_X$; see \cite[Sections~2.2 and 2.5]{FHM1})}}
    \end{align*}
    (We shall omit similar verifications in what follows.)
    Thus the $\bigl(\id{1}{L1}{\widetilde{e}_0}\comp{0}{L1}\widetilde e_n\bigr)$-to-$\widetilde e_n$ padding of $w$ witnesses $u \sim v$.
\end{proof}

The following lemma, which corresponds to \ref{step2}, generalises the ``weak uniqueness'' part of \cite[Lemma~4.6]{Lafont_Metayer_Worytkiewicz_folk_model_str_omega_cat} from strict $\omega$-categories to weak $\omega$-categories.
\begin{lemma}
\label{lem:whiskering-u-ess-inj}
    Let $X$ be a weak $\omega$-category, $x,y,z \in X_0$, and $u \colon x \to y$ an invertible $1$-cell.
    Then the whiskering map
    \[
    X(u,z) \colon X(y,z) \to X(x,z)
    \]
    is essentially $\omega$-injective.
\end{lemma}
\begin{remark}
    We will prove, by induction on $n$, that $X(u,z)$ is essentially $n$-injective.
    Before presenting the general proof, let us illustrate how the inductive step proceeds by considering the lowest-dimensional case.
    
    Suppose that we are given a parallel pair $(v_1,v_2)$ of $1$-cells in $X(y,z)$ which are $\sim$-related after whiskering with $u$.
    We may visualise the situation as:
    \[
    \begin{tikzpicture}[baseline = -2]
        \node (2) at (0,0) {$x$};
        \node (3) at (2,0) {$y$};
        \node (4) at (4,0) {$z$};

        \draw[->] (2) to node[auto,labelsize] {$u$} (3);
        \draw[->, bend left = 50] (3) to node[auto,labelsize] {$s$} (4);
        \draw[->, bend right = 50] (3) to node[auto,swap,labelsize] {$t$} (4);

        \draw[->,2cell] (3,0.3) to node[auto,labelsize] {$v_1$} (3,-0.3);
    \end{tikzpicture}
    \qquad\sim\qquad
    \begin{tikzpicture}[baseline = -2]
        \node (2) at (0,0) {$x$};
        \node (3) at (2,0) {$y$};
        \node (4) at (4,0) {$z$};

        \draw[->] (2) to node[auto,labelsize] {$u$} (3);
        \draw[->, bend left = 50] (3) to node[auto,labelsize] {$s$} (4);
        \draw[->, bend right = 50] (3) to node[auto,swap,labelsize] {$t$} (4);

        \draw[->,2cell] (3,0.3) to node[auto,labelsize] {$v_2$} (3,-0.3);
    \end{tikzpicture}
    \]
    Now we can whisker both sides with an inverse $\check u$ of $u$, and then with one of the witness invertible $2$-cells $q\colon \check u\comp{0}{X}u\to\id{1}{X}{y}$, and obtain:
    \[
    \begin{tikzpicture}[baseline = -36]
        \node (1) at (-2,0) {$y$};
        \node (2) at (0,0) {$x$};
        \node (3) at (2,0) {$y$};
        \node (4) at (4,0) {$z$};
        
        \draw[->] (1) to node[auto,labelsize] {$\check u$} (2);
        \draw[->] (2) to node[auto,labelsize] {$u$} (3);
        \draw[->, bend left = 50] (3) to node[auto,labelsize] {$s$} (4);
        \draw[->, bend right = 50] (3) to node[auto,swap,labelsize] {$t$} (4);

        \draw[->,2cell] (3,0.3) to node[auto,labelsize] {$v_1$} (3,-0.3);

        \node at (1,-1) {$\comp{1}{X}$};

        \node (5) at (-2,-2) {$y$};
        \node (6) at (0,-2) {$x$};
        \node (7) at (2,-2) {$y$};
        \node (8) at (4,-2) {$z$};
        
        \draw[->] (5) to node[auto,labelsize] {$\check u$} (6);
        \draw[->] (6) to node[auto,labelsize] {$u$} (7);
        \draw[->] (7) to node[auto,labelsize] {$t$} (8);
        \draw[->, bend right = 50] (5) to node[auto,labelsize,swap] {$\id{1}{X}{y}$} (7);
        \draw[->,2cell] (0,-2.3) to node[auto,labelsize] {$q$} (0,-2.9);
    \end{tikzpicture}
    \qquad\sim\qquad
    \begin{tikzpicture}[baseline = -36]
        \node (1) at (-2,0) {$y$};
        \node (2) at (0,0) {$x$};
        \node (3) at (2,0) {$y$};
        \node (4) at (4,0) {$z$};
        
        \draw[->] (1) to node[auto,labelsize] {$\check u$} (2);
        \draw[->] (2) to node[auto,labelsize] {$u$} (3);
        \draw[->, bend left = 50] (3) to node[auto,labelsize] {$s$} (4);
        \draw[->, bend right = 50] (3) to node[auto,swap,labelsize] {$t$} (4);

        \draw[->,2cell] (3,0.3) to node[auto,labelsize] {$v_2$} (3,-0.3);

        \node at (1,-1) {$\comp{1}{X}$};

        \node (5) at (-2,-2) {$y$};
        \node (6) at (0,-2) {$x$};
        \node (7) at (2,-2) {$y$};
        \node (8) at (4,-2) {$z$};
        
        \draw[->] (5) to node[auto,labelsize] {$\check u$} (6);
        \draw[->] (6) to node[auto,labelsize] {$u$} (7);
        \draw[->] (7) to node[auto,labelsize] {$t$}  (8);
        \draw[->, bend right = 50] (5) to node[auto,labelsize,swap] {$\id{1}{X}{y}$} (7);
        \draw[->,2cell] (0,-2.3) to node[auto,labelsize] {$q$} (0,-2.9);
    \end{tikzpicture}
    \]
    By padding, we can turn this invertible $3$-cell in $X$ into one of type:
    \[
    \begin{tikzpicture}[baseline = 30]
        \node (1) at (-2,0) {$y$};
        \node (3) at (2,0) {$y$};
        \node (4) at (4,0) {$z$};
        
        \draw[->] (1) to node[auto,labelsize] {$\id{1}{X}{y}$}(3);
        \draw[->, bend left = 50] (3) to node[auto,labelsize] {$s$} (4);
        \draw[->, bend right = 50] (3) to node[auto,swap,labelsize] {$t$} (4);

        \draw[->,2cell] (3,0.3) to node[auto,labelsize] {$v_1$} (3,-0.3);

        \node at (1,1) {$\comp{1}{X}$};

        \node (5) at (-2,2) {$y$};
        \node (6) at (0,3.2) {$x$};
        \node (7) at (2,2) {$y$};
        \node (8) at (4,2) {$z$};
        
        \draw[->] (5) to node[auto,labelsize] {$\check u$} (6);
        \draw[->] (6) to node[auto,labelsize] {$u$} (7);
        \draw[->] (7) to node[auto,labelsize] {$s$} (8);
        \draw[->] (5) to node[auto,labelsize,swap] {$\id{1}{X}{y}$} (7);
        \draw[->,2cell] (0,2.8) to node[auto,labelsize] {$q$} (0,2.2);
    \end{tikzpicture}
    \qquad\sim\qquad
    \begin{tikzpicture}[baseline = 30]
        \node (1) at (-2,0) {$y$};
        \node (3) at (2,0) {$y$};
        \node (4) at (4,0) {$z$};
        
        \draw[->] (1) to node[auto,labelsize] {$\id{1}{X}{y}$}(3);
        \draw[->, bend left = 50] (3) to node[auto,labelsize] {$s$} (4);
        \draw[->, bend right = 50] (3) to node[auto,swap,labelsize] {$t$} (4);

        \draw[->,2cell] (3,0.3) to node[auto,labelsize] {$v_2$} (3,-0.3);

        \node at (1,1) {$\comp{1}{X}$};

        \node (5) at (-2,2) {$y$};
        \node (6) at (0,3.2) {$x$};
        \node (7) at (2,2) {$y$};
        \node (8) at (4,2) {$z$};
        
        \draw[->] (5) to node[auto,labelsize] {$\check u$} (6);
        \draw[->] (6) to node[auto,labelsize] {$u$} (7);
        \draw[->] (7) to node[auto,labelsize] {$s$} (8);
        \draw[->] (5) to node[auto,labelsize,swap] {$\id{1}{X}{y}$} (7);
        \draw[->,2cell] (0,2.8) to node[auto,labelsize] {$q$} (0,2.2);
    \end{tikzpicture}
    \]
    The common upper cell in both sides may be regarded as an invertible $1$-cell in $X(y,z)$, so the essential $0$-injectivity of whiskering with this cell (together with \cref{lem:i-ess-inj}) implies $v_1 \sim v_2$.
\end{remark}
\begin{proof}[Proof of \cref{lem:whiskering-u-ess-inj}]
    We prove the statement 
\begin{equation}
\label{eqn:IH-inj}
    \parbox{\dimexpr\linewidth-5em}{for any weak $\omega$-category $(X,\xi)$, any $x,y,z \in X_0$, and any invertible $1$-cell $u \colon x \to y$ in $X$, the whiskering map $X(u,z) \colon X(y,z) \to X(x,z)$ is essentially $n$-injective}
\end{equation}
    by induction on $n\in\N$ (note the universal quantification over $(X,\xi)$).

    For the base case $n=0$, let $(X,\xi),x,y,z$, and $u$ be as in \cref{eqn:IH-inj}, and let $\check u\colon y\to x$ be an inverse of $u$. Take any $v_1,v_2\in \bigl(X(y,z)\bigr)_0$ with $u\comp{0}{X}v_1 \sim u\comp{0}{X}v_2$.
    Since we have
    \[
    v_1 \sim \id{1}{X}{y}\comp{0}{X}v_1 \sim (\check u\comp{0}{X} u)\comp{0}{X}v_1 \sim \check u\comp{0}{X} (u\comp{0}{X}v_1) 
    \]
    by the unit law (\cref{lem:unit-law}),
    and similarly $v_2 \sim \check u \comp{0}{X}(u\comp{0}{X}v_2)$, we can deduce $v_1 \sim v_2$.

    For the inductive step, let $n \ge 1$ and let $(X,\xi),x,y,z$, and $u$ be as in \cref{eqn:IH-inj}, and let $\check u\colon y\to x$ be an inverse of $u$, with invertible 2-cell $q\colon \check u\comp{0}{X}u\to\id{1}{X}{y}$.     
    Let $v_1,v_2 \in \bigl(X(y,z)\bigr)_{n}$ be parallel $n$-cells.
    Suppose that there exists an invertible $(n+1)$-cell
    \[
    w \colon u\comp{0}{X}v_1 \to u\comp{0}{X}v_2
    \]
    in $X(x,z)$.
    Then we can whisker it with $\check u$ and obtain an invertible $(n+1)$-cell
    \[
    \check u\comp{0}{X}w \colon \check u\comp{0}{X}(u\comp{0}{X}v_1) \to \check u\comp{0}{X}(u\comp{0}{X}v_2)
    \]
    in $X(y,z)$.
    Let $\phi = \widetilde e_1\comp{0}{L1}(\widetilde e_1\comp{0}{L1} \widetilde e_{n+1})\in(L1)_{n+1}$ and $\chi = (\widetilde e_1\comp{0}{L1} \widetilde e_1)\comp{0}{L1} \widetilde e_{n+1}\in(L1)_{n+1}$.
    Then the above cell $\check u\comp{0}{X}w$ is of type
    \[
    \xi\left(\phi, \begin{bmatrix}
        \check u & & u & & v_1\\
        & x & & y &
    \end{bmatrix}\right)
    \quad \to \quad 
    \xi\left(\phi, \begin{bmatrix}
        \check u & & u & & v_2\\
        & x & & y &
    \end{bmatrix}\right)
    \]
    (note that $v_1$ and $v_2$ are $(n+1)$-cells of $X$).
    Hence its $\phi$-to-$\chi$ padding is an invertible $(n+2)$-cell
    \[
    w'\colon (\check u\comp{0}{X}u)\comp{0}{X}v_1 \to (\check u\comp{0}{X}u)\comp{0}{X}v_2
    \]
    of $X$.
    We can now compose $w'$ with the invertible 2-cell ${q}\comp{0}{X}t_1^X(v_1) = {q}\comp{0}{X}t_1^X(v_2)$ of $X$ and obtain the invertible $(n+2)$-cell 
    \[
    w'\comp{1}{X}\bigl({q}\comp{0}{X}t^X_1(v_1)\bigr)\colon
    \bigl((\check u\comp{0}{X}u)\comp{0}{X}v_1\bigr)\comp{1}{X}\bigl({q}\comp{0}{X}t^X_1(v_1)\bigr) \to \bigl((\check u\comp{0}{X}u)\comp{0}{X}v_2\bigr)\comp{1}{X}\bigl({q}\comp{0}{X}t_1^X(v_2)\bigr)
    \]
    of $X$.
    Let $\phi' = (\widetilde e_1\comp{0}{L1}\widetilde e_{n+1})\comp{1}{L1}(\widetilde e_2\comp{0}{L1}\widetilde e_1)\in(L1)_{n+1}$ and $\chi'= (\widetilde e_2\comp{0}{L1}\widetilde e_1)\comp{1}{L1}(\widetilde e_1\comp{0}{L1}\widetilde e_{n+1})\in(L1)_{n+1}$.
    Then $w'\comp{1}{X}({q}\comp{0}{X}t_1^X(v_1))$ has the type
    \[
    \xi\left(\phi',\begin{bmatrix}
        {q} & & v_1 \\
        & y &
    \end{bmatrix}\right)
    \quad \to \quad
    \xi\left(\phi',\begin{bmatrix}
        {q} & & v_2 \\
        & y &
    \end{bmatrix}\right)
    \]
    and so its $\phi'$-to-$\chi'$ padding is an invertible $(n+2)$-cell of type
    \[
    \bigl({q}\comp{0}{X}s_1^X(v_1)\bigr)\comp{1}{X}\bigl(\id{1}{X}{y}\comp{0}{X}v_1\bigr)\to 
    \bigl({q}\comp{0}{X}s_1^X(v_2)\bigr)\comp{1}{X}\bigl(\id{1}{X}{y}\comp{0}{X}v_2\bigr),
    \]
    which can also be written as
    \[
    \bigl({q}\comp{0}{X}s_1^X(v_1)\bigr)\comp{0}{X(y,z)}\bigl(\id{1}{X}{y}\comp{0}{X}v_1\bigr)\to 
    \bigl({q}\comp{0}{X}s_1^X(v_2)\bigr)\comp{0}{X(y,z)}\bigl(\id{1}{X}{y}\comp{0}{X}v_2\bigr).
    \]
    Using the inductive hypothesis to the weak $\omega$-category $X(y,z)$ and the invertible 1-cell $q\comp{0}{X}s_1^X(v_1)$ in $X(y,z)$, we obtain 
    $\id{1}{X}{y}\comp{0}{X}v_1\sim \id{1}{X}{y}\comp{0}{X}v_2$.
    By \cref{lem:i-ess-inj}, 
    we conclude $v_1\sim v_2$.
\end{proof}

Observe that, thanks to \cref{cor:composition-in-hom}, the above lemma immediately generalises to maps of the form $u \comp{k-1}{X} (-)$ where $u$ is an invertible $k$-cell in $X$ with $k\geq 1$: for any $n\geq k$ and any parallel pair $(v_1,v_2)$ of $n$-cells of $X$ such that $s^X_{k-1}(v_1)=s^X_{k-1}(v_2)=t^X_{k-1}(u)$, $u\comp{k-1}{X}v_1\sim u\comp{k-1}{X}v_2$ implies $v_1\sim v_2$.

Before moving on to \ref{step3}, we need to prove the following lemmas.

\begin{lemma}\label{lem:padding-a-pasting}
    In the situation of \cref{con:padding}, suppose further that we are given:
    \begin{itemize}
        \item an $(n+1)$-dimensional pasting scheme $\dot\kk \colon \kk \to \kk$,
        \item $(n+1)$-dimensional pasting instructions $\dot \phi \colon \phi \to \phi, \dot \chi \colon \chi \to \chi$ of arity $\dot \kk$, and
        \item an $(n+1)$-dimensional pasting diagram $\ww \colon \uu \to \vv$ of shape $\dot\kk$ in $X$
    \end{itemize}
    such that $w = \xi(\dot\phi,\ww)$.
    Then there exists an invertible $(n+2)$-cell in $X$ from the $\phi$-to-$\chi$ padding $w_{n+1}$ to $\xi(\dot\chi,\ww)$.
\end{lemma}

\begin{remark}
    We can think of this lemma as stating a kind of naturality of the padding construction as a ``transformation'' from ``paste according to $\dot\phi$'' to ``paste according to $\dot \chi$.''
    For example, the case $n=1$ may be visualised as follows:
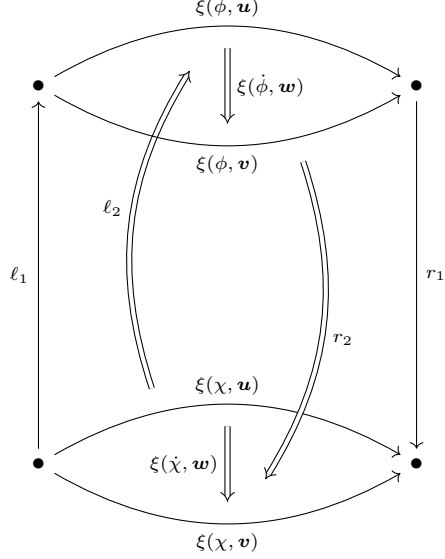
\[
\begin{tikzpicture}
    \node (ul) at (0,5) {$\bullet$};
    \node (ur) at (5,5) {$\bullet$};
    \node (ll) at (0,0) {$\bullet$};
    \node (lr) at (5,0) {$\bullet$};
    
    \draw[->, 2cell, bend left = 25] (1.5,1) to node[auto,labelsize] {$\ell_2$} (2,5.2);
    \draw[->] (ll) to node[auto,labelsize] {$\ell_1$} (ul);
    \draw[->] (ur) to node[auto,labelsize] {$r_1$} (lr);
    \draw[->, bend right = 30] (ll) to node[auto,labelsize,swap] {$\xi(\chi,\vv)$} (lr);
    \draw[->, bend left = 30] (ll) to node[auto,labelsize] {$\xi(\chi,\uu)$} (lr);
    \draw[->, bend right = 30] (ul) to node[auto,labelsize,swap] {$\xi(\phi,\vv)$} (ur);
    \draw[->, bend left = 30] (ul) to node[auto,labelsize] {$\xi(\phi,\uu)$} (ur);
    \draw[->, 2cell] (2.5,5.5) -- node[auto,labelsize] {$\xi(\dot\phi,\ww)$} (2.5,4.5);
    \draw[->, 2cell] (2.5,0.5) -- node[auto,labelsize,swap] {$\xi(\dot\chi,\ww)$} (2.5,-0.5);
    \draw[->, 2cell, bend left = 25] (3.5,4) to node[auto,labelsize] {$r_2$} (3,-0.2);
\end{tikzpicture}
\]
The lemma yields an invertible $3$-cell from the $\phi$-to-$\chi$ padding of $w=\xi(\dot\phi,\ww)$ (obtained by pasting all faces but the bottom) to $\xi(\dot\chi,\ww)$ (the bottom face), filling the interior of this ``naturality'' cylinder from $\xi(\dot\phi,-)$ to $\xi(\dot\chi,-)$.
\end{remark}

\begin{proof}[Proof of \cref{lem:padding-a-pasting}]
    Recall the cells $\lambda_i,\rho_i\in (L1)_i$, $\phi_i\in (L1)_n$, $\ell_i,r_i\in X_i$, and $w_i\in X_{n+1}$ ($1\leq i\leq n+1$) defined in \cref{con:padding}.
    Let $\dot\phi_0 = \dot\phi\colon \phi\to\phi$.
    By induction on $1 \le i \le n+1$, we define
    \[
    \dot\phi_{i} = \lambda_i\comp{i-1}{L1} (\dot\phi_{i-1}\comp{i-1}{L1}\rho_i);
    \]
    the expression on the right-hand side is well typed
    because we have $\dot\phi_i\colon \phi_i\to \phi_i$ in $L1$ for each $1\leq i\leq n$, as can be seen by applying \cref{prop:s-and-t-for-id-and-ast} inductively. When $i=n+1$, we have $\dot\phi_{n+1}\colon \chi\to\chi$. 
    We also have $\ar(\dot\phi_i) = \dot\kk$ and $w_i = \xi(\dot\phi_i,\ww)$ for each $1 \leq i\leq n+1$;
    the latter can be checked inductively as follows:
    \begin{align*}
        w_i &= \ell_i\comp{i-1}{X}(w_{i-1}\comp{i-1}{X}r_i )\\
        &= \xi\bigl(\lambda_i,\id{i}{TX}{s^{TX}_{i-1}(\ww)}\bigr)\comp{i-1}{X}\Bigl(\xi(\dot\phi_{i-1},\ww)\comp{i-1}{X}\xi\bigl(\rho_i,\id{i}{TX}{t^{TX}_{i-1}(\ww)}\bigr)\Bigr)\\
        &= \xi\Bigl(\bigl(\lambda_i,\id{i}{TX}{s^{TX}_{i-1}(\ww)}\bigr)\comp{i-1}{LX} \Bigl((\dot\phi_{i-1},\ww)\comp{i-1}{LX}\bigl(\rho_i,\id{i}{TX}{t^{TX}_{i-1}(\ww)}\bigr)\Bigr)\Bigr)\\
        &= \xi\bigl(\lambda_i\comp{i-1}{L1}(\dot\phi_{i-1}\comp{i-1}{L1}\rho_i),\ww\bigr)\\
        &= \xi(\dot\phi_i,\ww).
    \end{align*}
    In particular, the $\phi$-to-$\chi$ padding of $w$ can be written as $w_{n+1} = \xi(\dot\phi_{n+1},\ww)$.
    Since $\dot\chi$ is parallel to $\dot\phi_{n+1}$ in $L1$, the desired invertible $(n+2)$-cell may be obtained using coherence (\cref{cor:coherence}).
\end{proof}

\begin{lemma}\label{lem:padding-is-injective}
    In the situation of \cref{con:padding}, suppose we are given another $(n+1)$-cell $w' \colon \xi(\phi,\uu) \to \xi(\phi,\vv)$.
    Then the $\phi$-to-$\chi$ paddings of $w$ and $w'$ are connected by an invertible $(n+2)$-cell if and only if $w$ and $w'$ themselves are.
\end{lemma}
\begin{proof}
    We are whiskering $w$ and $w'$ with exactly the same (invertible) pads, so the ``if'' and ``only if'' directions follow from \cref{thm:pasting-invertible-cells} and \cref{lem:whiskering-u-ess-inj} respectively.
\end{proof}

Now we are ready to prove \ref{step3}.

\begin{lemma}\label{lem:i-ess-surj}
    Let $(X,\xi)$ be a weak $\omega$-category and let $y,z \in X_0$.
    Then the whiskering map
    \[
    X\bigl(\id{1}{X}{y},z\bigr) \colon X(y,z) \to X(y,z)
    \]
    is essentially $\omega$-surjective.
\end{lemma}

\begin{proof}
    The unit law (\cref{lem:unit-law}) implies that $X\bigl(\id{1}{X}{y},z\bigr)$ is essentially $0$-surjective. 
    Let $n\geq1$, $s,t \in X_n$ be parallel $n$-cells which are in $X(y,z)$, and suppose there exists an $(n+1)$-cell
    \[
    u \colon \id{1}{X}{y}\comp{0}{X}s \to \id{1}{X}{y}\comp{0}{X}t
    \]
    in $X$.
    Define $v \colon s \to t$ to be the $\bigl(\id{1}{L1}{\widetilde{e}_0}\comp{0}{L1}\widetilde{e}_{n}\bigr)$-to-$\widetilde e_{n}$ padding of $u$.
    We wish to prove
    \[
    \id{1}{X}{y}\comp{0}{X}v \sim u.
    \]
    
    By \cref{lem:padding-is-injective}, it suffices to provide an invertible $(n+2)$-cell between the $\bigl(\id{1}{L1}{\widetilde e_0}\comp{0}{L1}\widetilde e_{n}\bigr)$-to-$\widetilde e_{n}$ paddings of both sides, but this padding of the right-hand side is precisely $v$.
    Since we can write $v$ as $v = \xi(\widetilde e_{n+1},[v])$, the desired invertible $(n+2)$-cell can be obtained by applying \cref{lem:padding-a-pasting} to
    \begin{itemize}
        \item $\dot\phi = \id{1}{L1}{\widetilde e_0}\comp{0}{L1} \widetilde e_{n+1}$,
        \item $\dot \chi = \widetilde e_{n+1}$, and
        \item $\ww = [v]$.\qedhere
    \end{itemize}
\end{proof}

Finally we can prove \ref{step4}.
We note that combining it with \cref{lem:whiskering-u-ess-inj} yields a generalisation of \cite[Lemma~4.6]{Lafont_Metayer_Worytkiewicz_folk_model_str_omega_cat} from strict $\omega$-categories to weak $\omega$-categories.

\begin{theorem}
\label{lem:invertible-u-ess-surj}
    Let $X$ be a weak $\omega$-category, $x,y,z \in X_0$, and $u \colon x \to y$ an invertible $1$-cell.
    Then the whiskering map
    \[
    X(u,z) \colon X(y,z) \to X(x,z)
    \]
    is essentially $\omega$-surjective.
\end{theorem}
\begin{remark}
The following diagrams summarise the idea of the inductive step in the proof:
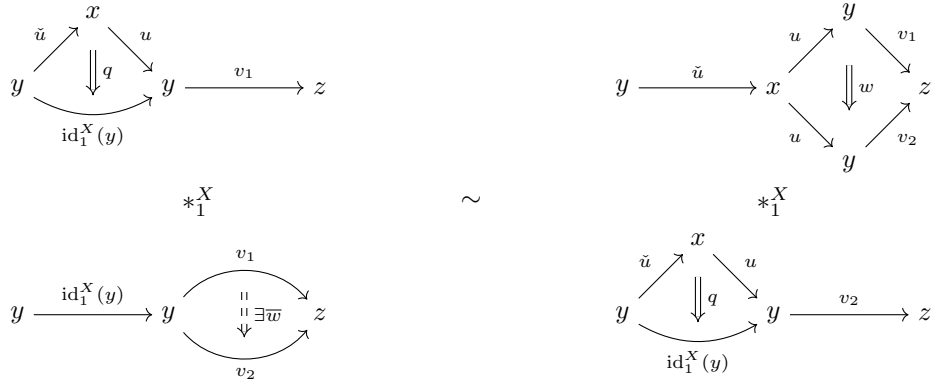
\[
\begin{tikzpicture}
    \node (ul1) at (5,1.5) {$y$};
    \node (ul2) at (6,2.5) {$x$};
    \node (ul3) at (7,1.5) {$y$};
    \node (ul4) at (9,1.5) {$z$};
    
    \draw[->] (ul1) to node[auto,labelsize] {$\check u$} (ul2);
    \draw[->] (ul2) to node[auto,labelsize] {$u$} (ul3);
    \draw[->,bend right = 30] (ul1) to node[auto,labelsize,swap] {$\id{1}{X}{y}$} (ul3);
    \draw[->] (ul3) to node[auto,labelsize] {$v_1$} (ul4);
    \draw[->,2cell] (6,2) to node[auto,labelsize] {$q$} (6,1.4);

    \node at (7.4,0) {$\comp{1}{X}$};

    \node (ll1) at (5,-1.5) {$y$};
    \node (ll2) at (7,-1.5) {$y$};
    \node (ll4) at (9,-1.5) {$z$};

    \draw[->] (ll1) to node[auto,labelsize] {$\id{1}{X}{y}$} (ll2);
    \draw[->, bend left = 50] (ll2) to node[auto,labelsize] {$v_1$} (ll4);
    \draw[->, bend right = 50] (ll2) to node[auto,labelsize,swap] {$v_2$} (ll4);
    \draw[->,2cell,dashed] (8,-1.2) to node[auto,labelsize] {$\exists \overline{w}$} (8,-1.8);
    
    \node at (11,0) {$\sim$};
    
    \node (ur1) at (13,1.5) {$y$};
    \node (ur2) at (15,1.5) {$x$};
    \node (ur3) at (16,2.5) {$y$};
    \node (ur4) at (16,0.5) {$y$};
    \node (ur5) at (17,1.5) {$z$};

    \draw[->] (ur1) to node[auto,labelsize] {$\check u$} (ur2);
    \draw[->] (ur2) to node[auto,labelsize] {$u$} (ur3);
    \draw[->] (ur2) to node[auto,labelsize,swap] {$u$} (ur4);
    \draw[->] (ur3) to node[auto,labelsize] {$v_1$} (ur5);
    \draw[->] (ur4) to node[auto,labelsize,swap] {$v_2$} (ur5);
    \draw[->,2cell] (16,1.8) to node[auto,labelsize] {$w$} (16,1.2);
    
    \node at (15,0) {$\comp{1}{X}$};

    \node (lr1) at (13,-1.5) {$y$};
    \node (lr2) at (14,-0.5) {$x$};
    \node (lr3) at (15,-1.5) {$y$};
    \node (lr4) at (17,-1.5) {$z$};
    
    \draw[->] (lr1) to node[auto,labelsize] {$\check u$} (lr2);
    \draw[->] (lr2) to node[auto,labelsize] {$u$} (lr3);
    \draw[->,bend right = 30] (lr1) to node[auto,labelsize,swap] {$\id{1}{X}{y}$} (lr3);
    \draw[->] (lr3) to node[auto,labelsize] {$v_2$} (lr4);
    \draw[->,2cell] (14,-1) to node[auto,labelsize] {$q$} (14,-1.6);
\end{tikzpicture}
\]
More precisely, they depict the two sides of \cref{eqn:surj-5} below in the case $n=1$ with all paddings suppressed.
\end{remark}
\begin{proof}[Proof of \cref{lem:invertible-u-ess-surj}]
    We prove the statement 
\begin{equation}
\label{eqn:IH-surj}
    \parbox{\dimexpr\linewidth-5em}{for any weak $\omega$-category $(X,\xi)$, any $x,y,z \in X_0$, and any invertible $1$-cell $u \colon x \to y$ in $X$, the whiskering map $X(u,z) \colon X(y,z) \to X(x,z)$ is essentially $n$-surjective}
\end{equation}
    by induction on $n\in\N$.

    For the base case $n=0$, let $(X,\xi),x,y,z$, and $u$ be as in \cref{eqn:IH-surj}, and let $\check u\colon y\to x$ be an inverse of $u$. Take any $w\in \bigl(X(x,z)\bigr)_0$. Then we have 
    \[
    w\sim \id{1}{X}{x}\comp{0}{X}w\sim (u\comp{0}{X}\check u)\comp{0}{X}w\sim u\comp{0}{X}(\check u\comp{0}{X}w)
    \]
    and this completes the base case.
    (Note that, for a higher dimensional $w$, no two of the four expressions in the above displayed line necessarily yield parallel cells.
    In particular, the two sides of the second ``$\sim$'' may not be parallel even when $X$ is a strict $\omega$-category.
    Our inductive step below overcomes this problem in exactly the same way as the strict case \cite[Lemma 4.6]{Lafont_Metayer_Worytkiewicz_folk_model_str_omega_cat}.)
    
    For the inductive step, let $n \ge 1$ and let $(X,\xi),x,y,z$, and $u$ be as in \cref{eqn:IH-surj}, and let $\check u\colon y\to x$ be an inverse of $u$, with invertible 2-cell $q\colon \check u\comp{0}{X}u\to\id{1}{X}{y}$ of $X$.     
    Let $(v_1,v_2)$ be a parallel pair of $(n-1)$-cells of $X(y,z)$  and let $w\colon u\comp{0}{X}v_1\to u\comp{0}{X}v_2$ be an $n$-cell of $X(x,z)$. We wish to find an $n$-cell $\overline w\colon v_1\to v_2$ of $X(y,z)$ such that 
    \begin{equation}
    \label{eqn:surj-1}
        u\comp{0}{X}\overline w\sim w.
    \end{equation}
    To this end, we shall consider conditions \cref{eqn:surj-2}--\cref{eqn:surj-5} on an $n$-cell $\overline w\colon v_1\to v_2$ of $X(y,z)$, show that they are all equivalent to \cref{eqn:surj-1}, and observe that there exists $\overline w$ satisfying the last condition \cref{eqn:surj-5} (hence also \cref{eqn:surj-1}).

    By the essential $\omega$-injectivity of whiskering with $\check u$ (\cref{lem:whiskering-u-ess-inj}), \cref{eqn:surj-1} is equivalent to 
    \begin{equation}
        \label{eqn:surj-2}
        \check u\comp{0}{X}(u\comp{0}{X}\overline w)\sim \check u \comp{0}{X}w.
    \end{equation}
    
    Let $\phi = \widetilde e_1 \comp{0}{L1} (\widetilde e_1 \comp{0}{L1} \widetilde e_n)$ and $\chi=(\widetilde e_1 \comp{0}{L1} \widetilde e_1) \comp{0}{L1} \widetilde e_n$.
    Then \cref{eqn:surj-2} relates two cells of type
    \[
    \xi\left(\phi,\begin{bmatrix}
        \check u & & u  & & v_1\\
        & x & & y &
    \end{bmatrix}\right) \quad \to \quad 
    \xi\left(\phi,\begin{bmatrix}
        \check u & & u  & & v_2\\
        & x & & y &
    \end{bmatrix}\right).
    \]
    Hence, thanks to \cref{lem:padding-is-injective}, \cref{eqn:surj-2} is equivalent to the existence of an invertible $(n+2)$-cell between their $\phi$-to-$\chi$ paddings.
    Combining this observation with \cref{lem:padding-a-pasting}, we see that \cref{eqn:surj-2} is equivalent to
    \begin{equation}
    \label{eqn:surj-3}
        (\check u \comp{0}{X} u) \comp{0}{X} \overline w\sim v',
    \end{equation}
    where $v'$ is the $\phi$-to-$\chi$ padding of the right-hand side of \cref{eqn:surj-2}.

    By the essential $\omega$-injectivity of whiskering with ${q} \comp{0}{X}t_1^X(\overline w)$ (\cref{lem:whiskering-u-ess-inj} applied to $X(y,z)$), \cref{eqn:surj-3} is in turn equivalent to
    \begin{equation}\label{eqn:surj-4}
    \bigl((\check u \comp{0}{X} u) \comp{0}{X} \overline w \bigr)\comp{1}{X} \bigl({q} \comp{0}{X} t_1^X(\overline w)\bigr) \sim v' \comp{1}{X}\bigl({q} \comp{0}{X} t_1^X(\overline w)\bigr).
    \end{equation}
    Note that $t^X_1(\overline w)$ is either equal to $v_2$ (when $n=1$) or to $t^X_1(v_1)=t^X_1(v_2)$ (when $n\geq 2$), and hence the right-hand side of \cref{eqn:surj-4} does not depend on $\overline w$.
    
    Observe that the left-hand side of \cref{eqn:surj-4} can be written as
    \[
    \xi\left((\widetilde e_1 \comp{0}{L1}\widetilde e_{n+1}) \comp{1}{L1}(\widetilde e_2 \comp{0}{L1} \widetilde e_1), \ 
    \begin{bmatrix}
        {q} & & \overline w\\
        & y &
    \end{bmatrix}
    \right).
    \]
    Thus again by applying \cref{lem:padding-is-injective,lem:padding-a-pasting}, this time to the pasting instructions
    \begin{align*}
    \phi' &= (\widetilde e_1 \comp{0}{L1}\widetilde e_{n}) \comp{1}{L1}(\widetilde e_2 \comp{0}{L1} \widetilde e_1)\text{ and}\\
    \chi' &= (\widetilde e_2 \comp{0}{L1}\widetilde e_1) \comp{1}{L1}(\widetilde e_1 \comp{0}{L1} \widetilde e_n),
    \end{align*}
    we see that \cref{eqn:surj-4} is equivalent to
    \begin{equation}\label{eqn:surj-5}
    \bigl({q} \comp{0}{X} s_1^X(\overline w)\bigr) \comp{1}{X}\bigl(\id{1}{X}{y} \comp{0}{X} \overline w\bigr) \sim v''
    \end{equation}
    where $v''$ is the $\phi'$-to-$\chi'$ padding of the right-hand side of \cref{eqn:surj-4}.
    But $\comp{1}{X}$ in the left-hand side can be replaced by $\comp{0}{X(y,z)}$ thanks to \cref{cor:composition-in-hom}, so the essential $(n-1)$-surjectivity of whiskering with ${q} \comp{0}{X} s_1^X(\overline w)$ (in $X(y,z)$; note again that $s_1^X(\overline w)$ is either equal to $v_1$ or to $s^X_1(v_1)=s^X_1(v_2)$) and \cref{lem:i-ess-surj} imply that we can construct $\overline w \colon v_1 \to v_2$ satisfying \cref{eqn:surj-5}.
    This completes the proof.
\end{proof}

We can also prove the converse to this theorem (although we will not use this converse in what follows).

\begin{proposition}
\label{prop:ess-surj-implies-invertible}
    Let $(X,\xi)$ be a weak $\omega$-category and $u \colon x \to y$ a $1$-cell in $X$.
    Suppose that the whiskering map
    \[
    X(u,z) \colon X(y,z) \to X(x,z) 
    \]
    is essentially $\omega$-surjective for all $z \in X_0$.
    Then $u$ is invertible in $X$.
\end{proposition}

\begin{proof}
    Since $X(u,x) \colon X(y,x) \to X(x,x)$ is essentially $0$-surjective, there exists $\check u \colon y \to x$ such that $u \comp{0}{X}\check u \sim \id{1}{X}{x}$.
    We wish to show $\check u \comp{0}{X} u \sim \id{1}{X}{y}$.
    
    Observe that the images of $\check u \comp{0}{X} u$ and $\id{1}{X}{y}$ under $X(u,y) \colon X(y,y) \to X(x,y)$ are connected by the following chain of invertible $2$-cells:
    \[
    u \comp{0}{X}(\check u \comp{0}{X} u) \sim (u \comp{0}{X} \check u) \comp{0}{X} u \sim \id{1}{X}{x} \comp{0}{X} u \sim u \sim u \comp{0}{X} \id{1}{X}{y}
    \]
    It thus suffices to prove that $X(u,y)$ reflects invertible cells.

    We will apply \cref{lem:weak-eq-refl-inv} to $X(u,y)$.
    Since we already know that $X(u,y)$ is essentially $\omega$-surjective, we only need to check that it satisfies (ii) and (iii) of that proposition.
    They both are straightforward consequences of coherence (\cref{cor:coherence}); for example, given a pair $v,v'$ of $n$-cells in $X(y,y)$ composable along the $(n-1)$-dimensional boundary, we have
    \begin{align*}
        u \comp{0}{X}\bigl(v \comp{n-1}{X(y,y)} v'\bigr) &= u \comp{0}{X}(v \comp{n}{X} v')\\
        & = \xi\left(\widetilde e_1 \comp{0}{L1} (\widetilde e_{n+1} \comp{n}{L1} \widetilde e_{n+1}), \begin{bmatrix}
            u & & v & & v'\\
            & y & & t^X_n(v) & 
        \end{bmatrix} \right)\\
        & \sim \xi\left((\widetilde e_1 \comp{0}{L1} \widetilde e_{n+1}) \comp{n}{L1} (\widetilde e_1 \comp{0}{L1} \widetilde e_{n+1}), \begin{bmatrix}
            u & & v & & v'\\
            & y & & t^X_n(v) & 
        \end{bmatrix} \right)\\
        &= (u \comp{0}{X} v) \comp{n}{X} (u \comp{0}{X} v')\\
        &= (u \comp{0}{X} v) \comp{n-1}{X(x,y)} (u \comp{0}{X} v')
    \end{align*}
    which proves that $X(u,y)$ satisfies (iii).
    The calculation for (ii) is similar.
\end{proof}

\subsection{The 2-out-of-3 property}
\label{subsec:2-out-of-3}
Now we are ready to complete the proof of the 2-out-of-3 property.

\begin{lemma}
\label{lem:functoriality}
    Let $X,Y$ be weak $\omega$-categories,
    $f\colon X\to Y$ be a strict $\omega$-functor, $u\colon x\to x'$ be a $1$-cell of $X$, and $x''$ be a $0$-cell of $X$. Then the following diagram of globular maps commutes:
    \[
    \begin{tikzcd}[row sep = large]
        X(x',x'')
        \arrow [r, "f_{x',x''}"]
        \arrow [d, "{X(u,x'')}", swap] &
        Y(fx',fx'')
        \arrow [d, "{Y(fu,fx'')}"]\\
        X(x,x'')
        \arrow [r, "f_{x,x''}", swap] &
        Y(fx,fx'')
    \end{tikzcd}
    \]
\end{lemma}
\begin{proof}
    Given any cell $v$ in $X(x',x'')$, we have 
    \begin{align*}
    \big(Y(fu,fx'')\circ f_{x',x''}\big)(v)
    &=fu \comp{0}{Y} fv\\
    &=f(u \comp{0}{X}v)\\
    &=\big(f_{x,x''}\circ X(u,x'')\big)(v).\qedhere
    \end{align*}
\end{proof}

The following implies \ref{item:C} of the 2-out-of-3 property.
\begin{lemma}
\label{prop:2-out-of-3-g}
    Let $n\in\N$, $X,Y,Z$ be weak $\omega$-categories, $f \colon X \to Y$ be a globular map, and $g \colon Y \to Z$ be a strict $\omega$-functor.
    If $f$ and $gf$ are essentially $n$-surjective, then so is $g$. 
\end{lemma}
\begin{proof}
    We prove this by induction on $n$.
    
    The base case $n=0$ is easy and only requires the essential $0$-surjectivity of $gf$. 

    For the inductive step, let $n>0$. Let $y,y'\in Y_0$. We want to show that $g_{y,y'}\colon Y(y,y')\to Z(gy,gy')$ is essentially $(n-1)$-surjective. Since $f$ is essentially $0$-surjective, there exist $x,x'\in X_0$ and invertible $1$-cells $u\colon y\to fx$ and $u'\colon fx'\to y'$ in $Y$.
    Consider the following diagram of globular maps.
    \[
    \begin{tikzcd}[row sep = large]
        X(x,x')
        \arrow [r, "f_{x,x'}"]
        \arrow [rr, bend left, "(gf)_{x,x'}"]&
        Y(fx,fx')
        \arrow [d, "{Y(u,u')}", swap]
        \arrow [r, "g_{fx,fx'}"]
        &
        Z(gfx,gfx')
        \arrow [d, "{Z(gu,gu')}"]\\
        &
        Y(y,y')
        \arrow [r, "g_{y,y'}"] 
         &
        Z(gy,gy')
    \end{tikzcd}
    \]
    Here, $Y(u,u')=Y(y,u')\circ Y(u,fx')$ and $Z(gu,gu')=Z(gy,gu')\circ Z(gu,gfx')$. Then the above diagram commutes by \cref{lem:functoriality} and its dual. 

    Since $f_{x,x'}$ and $(gf)_{x,x'}$ are essentially $(n-1)$-surjective, $g_{fx,fx'}$ is essentially $(n-1)$-surjective by the inductive hypothesis.
    Since $Z(gu,gu')$ is essentially $\omega$-surjective by \cref{lem:invertible-u-ess-surj}, $Z(gu,gu')\circ g_{fx,fx'}=g_{y,y'}\circ Y(u,u')$ is essentially $(n-1)$-surjective by \cref{rmk:whiskering-preserves-invertible-cells} and \cref{prop:2-out-of-3-gf}. Since $Y(u,u')$ is essentially $\omega$-surjective by \cref{lem:invertible-u-ess-surj}, $g_{y,y'}$ is essentially $(n-1)$-surjective by the inductive hypothesis.
\end{proof}

\begin{theorem}[2-out-of-3]
\label{thm:2-out-of-3}
    Let $f \colon X \to Y$ and $g \colon Y \to Z$ be strict $\omega$-functors between weak $\omega$-categories.
    If two of $f$, $g$, and $gf$ are $\omega$-weak equivalences, then so is the third.
\end{theorem}
\begin{proof}
    This follows from \cref{prop:2-out-of-3-gf,prop:2-out-of-3-f,prop:2-out-of-3-g}.
\end{proof}

We can also establish the stronger 2-out-of-6 property (so we have a \emph{homotopical category} in the sense of \cite{Dwyer_Hirschhorn_Kan_Smith}).

\begin{corollary}[2-out-of-6]
\label{cor:2-out-of-6}
    Let $f \colon X \to Y$, $g \colon Y \to Z$, and $h \colon Z \to W$ be strict $\omega$-functors between weak $\omega$-categories.
    Suppose that $gf$ and $hg$ are $\omega$-weak equivalences.
    Then $f,g,h$, and $hgf$ are also $\omega$-weak equivalences.
\end{corollary}

\begin{proof}
    We will show that $hgf$ is an $\omega$-weak equivalence; the rest will follow by the 2-out-of-3 property.

    To see that $hgf$ is essentially $0$-surjective, let $w \in W_0$.
    Since $hg$ is essentially $0$-surjective, there exists $y \in Y_0$ with $hgy \sim w$.
    Moreover, since $gf$ is essentially $0$-surjective, there exists $x \in X_0$ such that $gfx \sim gy$.
    Then we have $hgfx \sim hgy \sim w$.

    By applying the above argument to the hom weak $\omega$-categories, we can deduce that $hgf$ is essentially $n$-surjective for all $n \in \N$.
\end{proof}

\subsection{Other properties of \texorpdfstring{$\omega$}{ω}-weak equivalences}
\label{subsec:other-properties-omega-weak-eq}
Here we collect miscellaneous properties of $\omega$-weak equivalences. 

Recall the notion of essential $n$-injectivity from \cref{def:ess-inj}.

\begin{lemma}\label{lem:surj-implies-inj}
    Let $n\in\N$, $X,Y$ be weak $\omega$-categories, and $f \colon X \to Y$ be a globular map reflecting invertible cells. If $f$ is essentially $(n+1)$-surjective, then $f$ is essentially $n$-injective.
\end{lemma}
\begin{proof}
    We prove this by induction on $n$.

    The base case $n=0$ can be proved as follows. Suppose that $f\colon X\to Y$ is essentially $1$-surjective. Let $x,x'\in X_0$ be $0$-cells such that there exists an invertible $1$-cell $u\colon fx\to fx'$ in $Y$. Then there exists a $1$-cell $\overline u\colon x\to x'$ in $X$ such that $f\overline u\sim u$. By \cref{cor:invariance}, $f\overline u$ is invertible, and since $f$ reflects invertible cells, $\overline u$ is invertible. Therefore we have $x\sim x'$. 

    For the inductive step, let $n>0$ and suppose that $f\colon X\to Y$ is essentially $(n+1)$-surjective.  Then for each $x,x'\in X_0$, $f_{x,x'}$ is essentially $n$-surjective. By the inductive hypothesis, $f_{x,x'}$ is essentially $(n-1)$-injective. Hence $f$ is essentially $n$-injective.
\end{proof}

\begin{proposition}\label{prop:ess-inj}
    Any $\omega$-weak equivalence is essentially $\omega$-injective.
\end{proposition}
\begin{proof}
    This follows from \cref{lem:weak-eq-refl-inv} and \cref{lem:surj-implies-inj}.
\end{proof}

\begin{proposition}[{Cf.~\cite[Lemma~4.12]{Lafont_Metayer_Worytkiewicz_folk_model_str_omega_cat} and \cite[Proposition~20.1.18]{Ara_etal_polygraphs}}]
\label{prop:weak-eq-retract}
    The class of $\omega$-weak equivalences is closed under retracts in the arrow category $\WkCats{\omega}^\mathbf{2}$. 
    More generally, if $n\in\N$ and if
\begin{equation*}
\begin{tikzpicture}[baseline=-\the\dimexpr\fontdimen22\textfont2\relax ]
      \node(00) at (0,1) {$X$};
      \node(01) at (2,1) {$X'$};
      \node(10) at (0,-1) {$Y$};
      \node(11) at (2,-1) {$Y'$};
      \node(02) at (4,1) {$X$};
      \node(12) at (4,-1) {$Y$};
      
      \draw [->] (00) to node[auto, swap, labelsize] {$i$} (01); 
      \draw [->] (01) to node[auto, labelsize] {$f'$} (11); 
      \draw [->] (00) to node[auto,swap,labelsize] {$f$} (10); 
      \draw [->] (10) to node[auto,labelsize] {$j$} (11);  
      \draw [->] (01) to node[auto, swap, labelsize] {$p$} (02); 
      \draw [->] (11) to node[auto, labelsize] {$q$} (12); 
      \draw [->] (02) to node[auto,labelsize] {$f$} (12); 
      \draw [->, bend left=30] (00) to node[auto, labelsize] {$1_X$} (02); 
      \draw [->, bend right=30] (10) to node[auto, swap, labelsize] {$1_Y$} (12); 
\end{tikzpicture}
\end{equation*} 
    is a commutative diagram of globular maps between (the underlying globular sets of) weak $\omega$-categories such that $q$ preserves invertible cells, then $f$ is essentially $n$-surjective whenever $f'$ is.
\end{proposition}
\begin{proof}
We prove the latter statement by induction on $n$.

    The base case $n=0$ is shown as follows. Take any $y\in Y_0$. Then we obtain $jy\in Y'_0$. Since $f'$ is essentially $0$-surjective, there exist $x'\in X'_0$ and an invertible $1$-cell $u\colon f'x'\to jy$ in $Y'$. The $1$-cell $qu\colon qf'x'\to qjy$ in $Y$ is invertible by the assumption. Since $qf'x'=fpx'$ and $qjy=y$, we have $fpx'\sim y$ for $px'\in X_0$, as desired. 

    For the inductive step, let $n>0$. For each $x,x'\in X_0$, we have the commutative diagram 
\[
\begin{tikzpicture}[baseline=-\the\dimexpr\fontdimen22\textfont2\relax ]
      \node(00) at (0,1) {$X(x,x')$};
      \node(01) at (3,1) {$X'(ix,ix')$};
      \node(10) at (0,-1) {$Y(fx,fx')$};
      \node(11) at (3,-1) {$Y'(f'ix,f'ix')$};
      \node(02) at (6,1) {$X(x,x')$};
      \node(12) at (6,-1) {$Y(fx,fx')$};
      
      \draw [->] (00) to node[auto, swap, labelsize] {$i_{x,x'}$} (01); 
      \draw [->] (01) to node[auto, labelsize] {$f'_{ix,ix'}$} (11); 
      \draw [->] (00) to node[auto,swap,labelsize] {$f_{x,x'}$} (10); 
      \draw [->] (10) to node[auto,labelsize] {$j_{fx,fx'}$} (11);  
      \draw [->] (01) to node[auto, swap, labelsize] {$p_{ix,ix'}$} (02); 
      \draw [->] (11) to node[auto, labelsize] {$q_{f'ix,f'ix'}$} (12); 
      \draw [->] (02) to node[auto,labelsize] {$f_{x,x'}$} (12); 
      \draw [->, bend left=25] (00) to node[auto, labelsize] {$1_{X(x,x')}$} (02); 
      \draw [->, bend right=25] (10) to node[auto, swap, labelsize] {$1_{Y(fx,fx')}$} (12); 
\end{tikzpicture}
\]
    of globular maps between weak $\omega$-categories. Since $f'_{ix,ix'}$ is essentially $(n-1)$-surjective, so is $f_{x,x'}$, by the inductive hypothesis.
\end{proof}

\begin{proposition}[{Cf.~\cite[Lemma~4.12]{Lafont_Metayer_Worytkiewicz_folk_model_str_omega_cat} and \cite[Proposition~20.1.19]{Ara_etal_polygraphs}}]
    The class of $\omega$-weak equivalences is closed under filtered colimits in the arrow category $\WkCats{\omega}^\mathbf{2}$ and transfinite compositions in $\WkCats{\omega}$.
\end{proposition}
\begin{proof}
    The proof is almost identical to that of \cite[Proposition~20.1.19]{Ara_etal_polygraphs}.
    The only thing we must check is that the forgetful functor $\WkCats{\omega} \to \GSet$ preserves filtered colimits. This follows from the fact that the functor $L \colon \GSet \to \GSet$ preserves filtered colimits, which we established in \cref{L-is-finitary}.
\end{proof}

\section{Weak \texorpdfstring{$\omega$}{ω}-weak equivalences}\label{sec:weak-weak}
Whereas the notion of $\omega$-weak equivalence we have studied so far is natural from the viewpoint of model category theory (cf.~\cite{Lafont_Metayer_Worytkiewicz_folk_model_str_omega_cat}), the requirement that an $\omega$-weak equivalence should be a \emph{strict} $\omega$-functor is too restrictive as a fully general notion of ``equivalence'' between weak $\omega$-categories, i.e., as a suitable higher-dimensional analogue of a biequivalence between bicategories or a triequivalence between tricategories. 
In this section, we define and study a more general class of \emph{weak $\omega$-weak equivalences}, which are \emph{weak $\omega$-functors} \cite{Garner_homomorphisms} with suitable essential $\omega$-surjectivity. 
In particular, we show that they also satisfy the 2-out-of-3 property.
For the sake of clarity, in this section we call $\omega$-weak equivalences in the sense of \cref{def:omega-weak-eq} \emph{strict $\omega$-weak equivalences}.

\subsection{Weak \texorpdfstring{$\omega$}{ω}-functors}
\label{subsec:weak-omega-functors}
We first recall the definition of weak $\omega$-functor in the sense of Garner \cite{Garner_homomorphisms}, which is based on the theory of \emph{algebraic weak factorisation systems} (\textsc{awfs}) \cite{Garner_understanding,Bourke_Garner_1,Bourke_Garner_2}. However, since this definition depends only on a small portion of the data of \textsc{awfs}, here we give a direct definition without introducing the full structure of \textsc{awfs}. 
The following explanation is largely self-contained; the only serious exception is \cref{prop:loc-pres-ATF-initial}, where the existence of a certain left adjoint is derived via the theory of \textsc{awfs}.

\begin{definition}
\label{def:alg-triv-fib}
	Let $\mathbf{C}$ be a category and $\mathcal{I}=(\iota_i)_{i\in I}$ be an indexed family of morphisms in $\mathbf{C}$.

\begin{enumerate}
    \item A \emph{trivial fibration} with respect to $\mathcal{I}$ is a morphism in $\mathbf{C}$ that has the right lifting property against $\iota_i$ for each $i\in I$.
    \item An \emph{algebraic trivial fibration} with respect to $\mathcal{I}$ is a morphism $f\colon X \to Y$ in $\mathbf{C}$ equipped with an $\mathcal{I}$\emph{-lifting operation}
	$\overline\kappa$ (cf.~\cite[Proposition~3.8]{Garner_understanding}); the latter is a function which assigns to each commutative square with $i\in I$ as in the solid part below, a morphism $\overline \kappa(i;u,v)$ making both triangles commutative:
\begin{equation*}
\begin{tikzpicture}[baseline=-\the\dimexpr\fontdimen22\textfont2\relax ]
      \node(00) at (0,1) {$\bullet$};
      \node(01) at (2,1) {$X$};
      \node(10) at (0,-1) {$\bullet$};
      \node(11) at (2,-1) {$Y$.};
      
      \draw [->] (00) to node[auto, labelsize] {$u$} (01); 
      \draw [->] (01) to node[auto, labelsize] {$f$} (11); 
      \draw [->] (00) to node[auto,swap,labelsize] {$\iota_i$} (10); 
      \draw [->] (10) to node[auto,swap,labelsize] {$v$} (11);   
      \draw [->, dashed] (10) to node[midway,fill=white,labelsize] {$\overline \kappa(i;u,v)$} (01);
\end{tikzpicture}
\end{equation*}
	Given algebraic trivial fibrations $(f\colon X\to Y,\overline \kappa)$ and $(g\colon Z\to W, \overline \lambda)$, a \emph{morphism of algebraic trivial fibrations}
	from the former to the latter is a morphism $(h\colon X\to Z,k\colon Y\to W)\colon f\to g$ in the arrow category
	$\mathbf{C}^\mathbf{2}$ such that, for any $i\in I$ and any $(u,v)\colon \iota_i\to f$ in $\mathbf{C}^\mathbf{2}$,
	we have $h\circ \overline \kappa (i;u,v)=\overline \lambda (i;h\circ u,k\circ v)$.
\begin{equation*}
\begin{tikzpicture}[baseline=-\the\dimexpr\fontdimen22\textfont2\relax ]
      \node(00) at (0,2) {$\bullet$};
      \node(01) at (2,2) {$X$};
      \node(10) at (0,0) {$\bullet$};
      \node(11) at (2,0) {$Y$};
      \node(02) at (5,2) {$Z$};
      \node(12) at (5,0) {$W$};
      
      \draw [->] (00) to node[auto, labelsize] {$u$} (01); 
      \draw [->] (01) to node[auto, labelsize] {$h$} (02); 
      \draw [->] (01) to node[auto,near start, labelsize] {$f$} (11); 
      \draw [->] (02) to node[auto, labelsize] {$g$} (12); 
      \draw [->] (11) to node[auto,swap,labelsize] {$k$} (12);  
      \draw [->] (00) to node[auto,swap,labelsize] {$\iota_i$} (10); 
      \draw [->] (10) to node[auto,swap,labelsize] {$v$} (11);   
      \draw [->] (10) to node[midway,fill=white,labelsize] {$\overline \kappa(i;u,v)$} (01);
      \draw [->] (10) to node[near end,fill=white,labelsize] {$\overline \lambda(i;h\circ u,k\circ v)$} (02);
\end{tikzpicture}
\end{equation*}
	We write $\ATF(\mathbf{C};\mathcal{I})$ for the category of algebraic trivial fibrations and morphisms between them.
    Note that we have the functor $\mathrm{cod}\colon \ATF(\mathbf{C};\mathcal{I})\to\mathbf{C}$ sending an object $(f\colon X\to Y,\overline \kappa)$ to $Y$ and a morphism $(h,k)\colon (f,\overline \kappa)\to (g,\overline \lambda)$ to $k$. 
    \item Now suppose that the functor $\mathrm{cod}\colon \ATF(\mathbf{C};\mathcal{I})\to\mathbf{C}$ has a \emph{left adjoint right inverse} (\emph{lari}), i.e., a left adjoint with identity unit. 
    Explicitly, this means the following: for any object $X\in\mathbf{C}$, there exists an algebraic trivial fibration $(\varepsilon_X\colon QX\to X,\overline\kappa_X)$ over $X$ such that, for any algebraic trivial fibration $(g\colon Z\to W,\overline \lambda)$ and any morphism $k\colon X\to W$ in $\mathbf{C}$, there exists a unique morphism $h\colon QX\to Z$ in $\mathbf{C}$ such that $(h, k)$ is a morphism of algebraic trivial fibrations $(\varepsilon_X,\overline\kappa_X)\to (g,\overline\lambda)$. 
    (This is equivalent to the condition that the unique morphism in the slice 2-category $\mathbf{CAT}/\mathbf{C}$ from the object $\cod\colon \ATF(\mathbf{C};\mathcal{I})\to \mathbf{C}$ to the terminal object $1_\mathbf{C}\colon \mathbf{C}\to \mathbf{C}$ has a left adjoint \emph{in} $\mathbf{CAT}/\mathbf{C}$.)
    Under this situation, we can make the following definitions. 
    \begin{itemize}
        \item For each $X\in\mathbf{C}$,
        let $\delta_X\colon QX\to Q^2X$ $\bigl(=Q(QX)\bigr)$ be the unique morphism in $\mathbf{C}$ such that
        \begin{equation}
\begin{tikzpicture}[baseline=-\the\dimexpr\fontdimen22\textfont2\relax ]
      \node(01) at (0,0.75) {$QX$};
      \node(11) at (0,-0.75) {$X$};
      \node(02) at (1.5,0.75) {$Q^2X$};
      \node(12) at (1.5,-0.75) {$X$};
      \draw [->] (01) to node[auto, labelsize] {$\delta_X$} (02); 
      \draw [->] (01) to node[auto,swap, labelsize] {$(\varepsilon_X,\overline\kappa_X)$} (11); 
      \draw [->] (02) to node[auto, labelsize] {$(\varepsilon_X\circ\varepsilon_{QX},\overline\lambda_X)$} (12); 
      \draw [->] (11) to node[auto,swap,labelsize] {$1_X$} (12);
\end{tikzpicture}\tag{D}\label{eqn:delta}
\end{equation}
        is a morphism of algebraic trivial fibrations, where $\overline\lambda_X$ is the $\mathcal{I}$-lifting operation on $\varepsilon_X\circ\varepsilon_{QX}$ defined by $\overline\lambda_X(i;u,v)=\overline\kappa_{QX}\bigl(i;u,\overline\kappa_X(i;\varepsilon_{QX}\circ u, v)\bigr)$
        for each $i\in I$ and $(u,v)\colon \iota_i \to \varepsilon_X\circ\varepsilon_{QX}$ in $\mathbf{C}^\mathbf{2}$ (depicted below).
        \begin{equation}
        \label{eqn:vertical-composition}
        \begin{tikzpicture}[baseline=-\the\dimexpr\fontdimen22\textfont2\relax ]
              \node(00) at (0,1.5) {$\bullet$};
              \node(01) at (4,1.5) {$Q^2X$};
              \node(10) at (0,-1.5) {$\bullet$};
              \node(11) at (4,-1.5) {$X$};
              \node(12) at (4,0) {$QX$};
              
              \draw [->] (00) to node[auto, labelsize] {$u$} (01); 
              \draw [->] (01) to node[auto, labelsize] {$\varepsilon_{QX}$} (12); 
              \draw [->] (12) to node[auto, labelsize] {$\varepsilon_X$} (11); 
              \draw [->] (00) to node[auto,swap,labelsize] {$\iota_i$} (10); 
              \draw [->] (10) to node[auto,swap,labelsize] {$v$} (11);   
              \draw [->] (10) to (12);
              \draw [->] (10) to (01);    
              \node[fill=white,labelsize]  at (2.5,-0.8) {$\overline \kappa_X(i;\varepsilon_{QX}\circ u,v)$};
              \node[fill=white,labelsize]  at (2,0.5) {$\overline \kappa_{QX}\bigl(i;u,\overline \kappa_X(\varepsilon_{QX}\circ u,v)\bigr)$};
        \end{tikzpicture}
        \end{equation}
      \item For each morphism $f\colon X\to Y$ in $\mathbf{C}$, let $Qf\colon QX\to QY$ be the unique morphism in $\mathbf{C}$ such that
      \begin{equation}
    \begin{tikzpicture}[baseline=-\the\dimexpr\fontdimen22\textfont2\relax ]
          \node(01) at (0,0.75) {$QX$};
          \node(11) at (0,-0.75) {$X$};
          \node(02) at (1.5,0.75) {$QY$};
          \node(12) at (1.5,-0.75) {$Y$};
          \draw [->] (01) to node[auto, labelsize] {$Qf$} (02); 
          \draw [->] (01) to node[auto,swap, labelsize] {$(\varepsilon_X,\overline\kappa_X)$} (11); 
          \draw [->] (02) to node[auto, labelsize] {$(\varepsilon_Y,\overline \kappa_Y)$} (12); 
          \draw [->] (11) to node[auto,swap,labelsize] {$f$} (12);
    \end{tikzpicture}\tag{Q}\label{eqn:Qf}
    \end{equation}
    is a morphism of algebraic trivial fibrations. \qedhere
    \end{itemize}
\end{enumerate}
\end{definition}

\begin{remark}
\label{rmk:fibres-of-ATF}
    Let $\mathbf{C}$ be a category and $\mathcal{I}=(\iota_i)_{i\in I}$ be an indexed family of morphisms in $\mathbf{C}$. 
    For each $X\in\mathbf{C}$, let $\ATF(\mathbf{C};\mathcal{I})_X$ be the fibre over $X$ of the functor $\mathrm{cod}\colon \ATF(\mathbf{C};\mathcal{I})\to\mathbf{C}$, i.e., the pullback
     \begin{equation*}
\begin{tikzpicture}[baseline=-\the\dimexpr\fontdimen22\textfont2\relax ]
      \node(00) at (0,1) {$\ATF(\mathbf{C};\mathcal{I})_X$};
      \node(01) at (3,1) {$\ATF(\mathbf{C};\mathcal{I})$};
      \node(10) at (0,-1) {$1$};
      \node(11) at (3,-1) {$\mathbf{C}$};
      \draw  (0.3,0.5) to (0.5,0.5) to (0.5,0.7); 
      \draw [->] (00) to node[auto, labelsize] {} (01); 
      \draw [->] (01) to node[auto, labelsize] {$\mathrm{cod}$} (11); 
      \draw [->] (00) to node[auto,swap,labelsize] {} (10); 
      \draw [->] (10) to node[auto,swap,labelsize] {$X$} (11);
\end{tikzpicture}
\end{equation*}
    in $\mathbf{CAT}$. Notice that the category $\ATF(\mathbf{C};\mathcal{I})_X$ always has the terminal object $(1_X\colon X\to X,\overline\omega_X)$, where $\overline\omega_X$ is the (unique) $\mathcal{I}$-lifting operation on $1_X$ defined by $\overline\omega_X(i;u,v)=v$. 
    For any $(f,\overline\kappa) \in \ATF(\mathbf{C};\mathcal{I})_X$, the unique map is given by
    \begin{equation}
    \begin{tikzpicture}[baseline=-\the\dimexpr\fontdimen22\textfont2\relax ]
          \node(01) at (0,0.75) {$\bullet$};
          \node(11) at (0,-0.75) {$X$};
          \node(02) at (1.5,0.75) {$X$};
          \node(12) at (1.5,-0.75) {$X$.};
          \draw [->] (01) to node[auto, labelsize] {$f$} (02); 
          \draw [->] (01) to node[auto,swap, labelsize] {$(f,\overline\kappa)$} (11); 
          \draw [->] (02) to node[auto, labelsize] {$(1_X,\overline\omega_X)$} (12); 
          \draw [->] (11) to node[auto,swap,labelsize] {$1_X$} (12);
    \end{tikzpicture}\tag{I}\label{eqn:terminal}
    \end{equation}

    If $\mathrm{cod}\colon \ATF(\mathbf{C};\mathcal{I})\to\mathbf{C}$ has a lari, then the value $(\varepsilon_X\colon QX\to X,\overline\kappa_X)$ of the lari at $X\in\mathbf{C}$ is an initial object in $\ATF(\mathbf{C};\mathcal{I})_X$. 
    The converse holds if the functor $\mathrm{cod}\colon \ATF(\mathbf{C};\mathcal{I})\to\mathbf{C}$ is a (Grothendieck) fibration; that is, the existence of an initial object in each fibre $\ATF(\mathbf{C};\mathcal{I})_X$ implies that $\mathrm{cod}\colon \ATF(\mathbf{C};\mathcal{I})\to\mathbf{C}$ has a lari in that case. 
    A sufficient condition for $\mathrm{cod}\colon \ATF(\mathbf{C};\mathcal{I})\to\mathbf{C}$ to be a fibration is that $\mathbf{C}$ has pullbacks along trivial fibrations with respect to $\mathcal{I}$. 
\end{remark}

\begin{proposition}
\label{prop:univ-cofib-repl-comonad}
    In the situation of (3) of \cref{def:alg-triv-fib}, the data $(Q,\varepsilon,\delta)$ give rise to a comonad on $\mathbf{C}$. 
\end{proposition}
\begin{proof}
    In short, one can verify the functoriality of $Q$, the naturality of $\varepsilon$ and $\delta$, and the comonad axioms, using the universality of $(\varepsilon_X,\overline\kappa_X)$. 

    Below we give more details. The functoriality of $Q$ and the naturality of $\varepsilon$ are immediate from the definitions. To show the rest, 
    the following observation (mentioned in \cite[Section~5.2]{Bourke_Garner_1}) will be useful:
    there exists a (strict) double category $\AATF(\mathbf{C};\mathcal{I})$ whose objects are the objects of $\mathbf{C}$, whose vertical morphisms are the algebraic trivial fibrations with respect to $\mathcal{I}$, whose horizontal morphisms are the morphisms of $\mathbf{C}$, and for each configuration
        \begin{equation}
        \label{eqn:booundary-in-AATF}
\begin{tikzpicture}[baseline=-\the\dimexpr\fontdimen22\textfont2\relax ]
      \node(01) at (0,0.75) {$X$};
      \node(11) at (0,-0.75) {$Y$};
      \node(02) at (1.5,0.75) {$Z$};
      \node(12) at (1.5,-0.75) {$W$};
      \draw [->] (01) to node[auto, labelsize] {$h$} (02); 
      \draw [->] (01) to node[auto,swap, labelsize] {$(f,\overline\kappa)$} (11); 
      \draw [->] (02) to node[auto, labelsize] {$(g,\overline\lambda)$} (12); 
      \draw [->] (11) to node[auto,swap,labelsize] {$k$} (12);
\end{tikzpicture}
\end{equation}
    in $\AATF(\mathbf{C};\mathcal{I})$, if $(h,k)$ is a morphism $(f,\overline\kappa)\to (g,\overline\lambda)$ in $\ATF(\mathbf{C}; \mathcal{I})$, then there exists precisely one square in $\AATF(\mathbf{C};\mathcal{I})$ having \cref{eqn:booundary-in-AATF} as the boundary, and otherwise, there exists no such square. The $\mathcal{I}$-lifting operation on the composite of vertical morphisms is defined as in \cref{eqn:vertical-composition}. 
    
    To show the naturality of $\delta$, it suffices to show the following: for any morphism $f\colon X\to Y$ in $\mathbf{C}$, both $(\delta_Y\circ Qf,f)$ and $(Q^2f\circ \delta_X,f)$ are morphisms of algebraic trivial fibrations $(\varepsilon_X,\overline\kappa_X) \to (\varepsilon_Y\circ \varepsilon_{QY},\overline\lambda_Y)$.
    This can be done by forming the following pasting composites in $\AATF(\mathbf{C};\mathcal{I})$: 
        \begin{equation*}
        \begin{tikzpicture}[baseline=-\the\dimexpr\fontdimen22\textfont2\relax ]
            \node (00) at (0,0.75) {$QX$};
            \node (10) at (2,0.75) {$QY$};
            \node (20) at (4,0.75) {$Q^2Y$};
            \node (01) at (0,-0.75) {$X$};
            \node (11) at (2,-0.75) {$Y$};
            \node (21) at (4,-0.75) {$Y$};
            \draw[->] (00) to node[auto,labelsize] {$Qf$} (10);
            \draw[->] (10) to node[auto,labelsize] {$\delta_Y$} (20);
            \draw[->] (01) to node[auto,swap,labelsize] {$f$} (11);
            \draw[->] (11) to node[auto,swap,labelsize] {$1_Y$} (21);
            \draw[->] (00) to node[auto,swap,labelsize] {$(\varepsilon_X,\overline\kappa_X)$} (01);
            \draw[->] (10) to node[midway,fill=white,labelsize] {$(\varepsilon_Y,\overline\kappa_Y)$} (11);
            \draw[->] (20) to node[auto,labelsize] {$(\varepsilon_Y\circ \varepsilon_{QY},\overline\lambda_Y)$} (21);
            \node at (1,0) {(\ref{eqn:Qf})};
            \node at (3,0) {(\ref{eqn:delta})};
        \end{tikzpicture}
        \qquad
        \begin{tikzpicture}[baseline=-\the\dimexpr\fontdimen22\textfont2\relax ]
              \node(00) at (0,1.5)   {$QX$};
              \node(02) at (0,-1.5)  {$X$};
              \node(10) at (2,1.5) {$Q^2X$};
              \node(11) at (2,0)   {$QX$};
              \node(12) at (2,-1.5){$X$};
              \node(20) at (4,1.5)   {$Q^2Y$};
              \node(21) at (4,0)     {$QY$};
              \node(22) at (4,-1.5)  {$Y$.};
              \draw [->] (00) to node[auto,swap, labelsize] {$(\varepsilon_X,\overline\kappa_X)$} (02); 
              \draw [->] (10) to node[midway,fill=white, labelsize] {$(\varepsilon_{QX},\overline\kappa_{QX})$} (11); 
              \draw [->] (11) to node[midway,fill=white, labelsize] {$(\varepsilon_{X},\overline\kappa_{X})$} (12); 
              \draw [->] (20) to node[auto, labelsize] {$(\varepsilon_{QY},\overline\kappa_{QY})$} (21); 
              \draw [->] (21) to node[auto, labelsize] {$(\varepsilon_{Y},\overline\kappa_{Y})$} (22); 
              \draw [->] (00) to node[auto, labelsize] {$\delta_X$} (10); 
              \draw [->] (10) to node[auto, labelsize] {$Q^2f$} (20); 
              \draw [->] (11) to node[auto, labelsize] {$Qf$} (21); 
              \draw [->] (02) to node[auto,swap, labelsize] {$1_X$} (12); 
              \draw [->] (12) to node[auto,swap, labelsize] {$f$} (22);
              \node at (1,0) {(\ref{eqn:delta})};
              \node at (3,0.75) {(\ref{eqn:Qf})};
              \node at (3,-0.75) {(\ref{eqn:Qf})};
        \end{tikzpicture}
        \end{equation*}
    
    For the axiom $\varepsilon_{QX}\circ \delta_X=1_{QX}= Q\varepsilon_X\circ \delta_X$, it suffices to show that both $(\varepsilon_{QX},1_X)$ and $(Q\varepsilon_X,1_X)$ are morphisms $(\varepsilon_X\circ \varepsilon_{QX},\overline\lambda_X)\to (\varepsilon_X,\overline\kappa_X)$ in $\ATF(\mathbf{C};\mathcal{I})$ (then $(\varepsilon_{QX}\circ \delta_X,1_X)$, $(1_{QX},1_X)$, and $(Q\varepsilon_X\circ \delta_X,1_X)$ are parallel morphisms with domain $(\varepsilon_{X},\overline\kappa_X)$ over the same morphism $1_X$,
    and hence have to coincide).
    These can be seen by the following composites in $\AATF(\mathbf{C};\mathcal{I})$: 
    \begin{equation*}
    \begin{tikzpicture}[baseline=-\the\dimexpr\fontdimen22\textfont2\relax ]
              \node(10) at (1.5,1.5) {$Q^2X$};
              \node(11) at (1.5,0)   {$QX$};
              \node(12) at (1.5,-1.5){$X$};
              \node(20) at (3,1.5)   {$QX$};
              \node(21) at (3,0)     {$QX$};
              \node(22) at (3,-1.5)  {$X$};
              \draw [->] (10) to node[auto,swap, labelsize] {$(\varepsilon_{QX},\overline\kappa_{QX})$} (11); 
              \draw [->] (11) to node[auto,swap, labelsize] {$(\varepsilon_{X},\overline\kappa_{X})$} (12); 
              \draw [->] (20) to node[auto, labelsize] {$(1_{QX},\overline\omega_{QX})$} (21); 
              \draw [->] (21) to node[auto, labelsize] {$(\varepsilon_{X},\overline\kappa_{X})$} (22); 
              \draw [->] (10) to node[auto, labelsize] {$\varepsilon_{QX}$} (20); 
              \draw [->] (11) to node[auto, labelsize] {$1_{QX}$} (21); 
              \draw [->] (12) to node[auto,swap, labelsize] {$1_X$} (22); 
              \node at (2.25,0.75) {(\ref{eqn:terminal})};
              \node at (2.25,-0.75) {$=$};
        \end{tikzpicture}
        \qquad
            \begin{tikzpicture}[baseline=-\the\dimexpr\fontdimen22\textfont2\relax ]
              \node(10) at (1.5,1.5) {$Q^2X$};
              \node(11) at (1.5,0)   {$QX$};
              \node(12) at (1.5,-1.5){$X$};
              \node(20) at (3,1.5)   {$QX$};
              \node(21) at (3,0)     {$X$};
              \node(22) at (3,-1.5)  {$X$};
              \draw [->] (10) to node[auto,swap, labelsize] {$(\varepsilon_{QX},\overline\kappa_{QX})$} (11); 
              \draw [->] (11) to node[auto,swap, labelsize] {$(\varepsilon_{X},\overline\kappa_{X})$} (12); 
              \draw [->] (20) to node[auto, labelsize] {$(\varepsilon_{X},\overline\kappa_{X})$} (21); 
              \draw [->] (21) to node[auto, labelsize] {$(1_{X},\overline\omega_{X})$} (22); 
              \draw [->] (10) to node[auto, labelsize] {$Q\varepsilon_{X}$} (20); 
              \draw [->] (11) to node[auto, labelsize] {$\varepsilon_{X}$} (21); 
              \draw [->] (12) to node[auto,swap, labelsize] {$1_X$} (22); 
              \node at (2.25,0.75) {(\ref{eqn:Qf})};
              \node at (2.25,-0.75) {(\ref{eqn:terminal})};
        \end{tikzpicture}
        \end{equation*}
        
    Finally, for the axiom $\delta_{QX}\circ\delta_X=Q\delta_X\circ\delta_X$, it suffices to show that $(\delta_{QX},1_X)$ and $(Q\delta_X,1_X)$ are parallel morphisms in $\ATF(\mathbf{C};\mathcal{I})$, which can be seen as follows: 
\begin{equation*}
       \begin{tikzpicture}[baseline=-\the\dimexpr\fontdimen22\textfont2\relax ]
              \node(10) at (1.5,4.5) {$Q^2X$};
              \node(12) at (1.5,1.5){$QX$};
              \node(13) at (1.5,0){$X$};
              \node(20) at (3,4.5)   {$Q^3X$};
              \node(21) at (3,3)   {$Q^2X$};
              \node(22) at (3,1.5)  {$QX$};
              \node(23) at (3,0)  {$X$};
              \draw [->] (10) to node[auto,swap, labelsize] {$(\varepsilon_{QX},\overline\kappa_{QX})$} (12);  
              \draw [->] (12) to node[auto,swap, labelsize] {$(\varepsilon_{X},\overline\kappa_{X})$} (13); 
              \draw [->] (20) to node[auto, labelsize] {$(\varepsilon_{Q^2X},\overline\kappa_{Q^2X})$} (21); 
              \draw [->] (21) to node[auto, labelsize] {$(\varepsilon_{QX},\overline\kappa_{QX})$} (22); 
              \draw [->] (22) to node[auto, labelsize] {$(\varepsilon_{X},\overline\kappa_{X})$} (23); 
              \draw [->] (10) to node[auto, labelsize] {$\delta_{QX}$} (20); 
              \draw [->] (12) to node[auto, labelsize] {$1_{QX}$} (22); 
              \draw [->] (13) to node[auto,swap, labelsize] {$1_X$} (23); 
              \node at (2.25,3) {(\ref{eqn:delta})};
              \node at (2.25,0.75) {$=$};
        \end{tikzpicture}
        \qquad
              \begin{tikzpicture}[baseline=-\the\dimexpr\fontdimen22\textfont2\relax ]
              \node(10) at (1.5,4.5) {$Q^2X$};
              \node(11) at (1.5,3) {$QX$};
              \node(13) at (1.5,0){$X$};
              \node(20) at (3,4.5)   {$Q^3X$};
              \node(21) at (3,3)   {$Q^2X$};
              \node(22) at (3,1.5)  {$QX$};
              \node(23) at (3,0)  {$X$};
              \draw [->] (10) to node[auto,swap, labelsize] {$(\varepsilon_{QX},\overline\kappa_{QX})$} (11); 
              \draw [->] (11) to node[auto,swap, labelsize] {$(\varepsilon_{X},\overline\kappa_{X})$} (13);  
              \draw [->] (20) to node[auto, labelsize] {$(\varepsilon_{Q^2X},\overline\kappa_{Q^2X})$} (21); 
              \draw [->] (21) to node[auto, labelsize] {$(\varepsilon_{QX},\overline\kappa_{QX})$} (22); 
              \draw [->] (22) to node[auto, labelsize] {$(\varepsilon_{X},\overline\kappa_{X})$} (23); 
              \draw [->] (10) to node[auto, labelsize] {$Q\delta_{X}$} (20); 
              \draw [->] (11) to node[auto, labelsize] {$\delta_{X}$} (21); 
              \draw [->] (13) to node[auto,swap, labelsize] {$1_X$} (23); 
              \node at (2.25,3.75) {(\ref{eqn:Qf})};
              \node at (2.25,1.5) {(\ref{eqn:delta})};
        \end{tikzpicture}\qedhere
\end{equation*}
\end{proof}

A sufficient condition for $\cod\colon \ATF(\mathbf{C};\mathcal{I})\to \mathbf{C}$ to have a lari (and hence for the comonad $Q$ as in \cref{prop:univ-cofib-repl-comonad} to exist) is given by the following.

\begin{proposition}
\label{prop:loc-pres-ATF-initial}
    Let $\mathbf{C}$ be a locally presentable category and $\mathcal{I}=(\iota_i)_{i\in I}$ be an indexed family of morphisms in $\mathbf{C}$ whose indexing set $I$ is small. Then the functor $\cod\colon \ATF(\mathbf{C};\mathcal{I})\to \mathbf{C}$ has a lari. 
\end{proposition}
\begin{proof}
	We may apply the algebraic small object argument (see \cite[Theorem~4.4]{Garner_understanding} or \cite[Proposition~16]{Bourke_Garner_1}) to $\mathcal{I}$ and obtain
	an \textsc{awfs} $(\mathsf{L},\mathsf{R})$ on $\mathbf{C}$.
	Then $(\mathsf{L},\mathsf{R})$ is \emph{cofibrantly generated} by $\mathcal{I}$ in the sense of \cite[Section~5.2]{Bourke_Garner_1} (or \emph{algebraically-free} on $\mathcal{I}$ in the sense of \cite[Definition~3.9]{Garner_understanding}),
	which means that a certain canonical functor $K\colon \mathsf{R}\mhyphen\mathbf{Alg}\to \ATF(\mathbf{C};\mathcal{I})$, making the triangle
 \[
\begin{tikzpicture}[baseline=-\the\dimexpr\fontdimen22\textfont2\relax ]
      \node(00) at (0,1) {$\mathsf{R}\mhyphen\mathbf{Alg}$};
      \node(01) at (2,-1) {$\mathbf{C^2}$};
      \node(10) at (4,1) {$\ATF(\mathbf{C};\mathcal{I})$};
      
      \draw [->] (00) to node[auto,swap, labelsize] {forgetful} (01);  
      \draw [->] (00) to node[auto,labelsize] {$K$} (10); 
      \draw [->] (10) to node[auto,labelsize] {forgetful} (01);   
\end{tikzpicture}
\]
    commute, is an isomorphism of categories; here, $\mathsf{R}\mhyphen\mathbf{Alg}$ is the Eilenberg--Moore category of the monad $\mathsf{R}$ on $\mathbf{C}^\mathbf{2}$. Hence it suffices to show that the composite of the forgetful functor $\mathsf{R}\mhyphen\mathbf{Alg}\to \mathbf{C^2}$ and $\cod\colon\mathbf{C^2}\to\mathbf{C}$ admits a lari, which follows from the definition of \textsc{awfs}; the lari sends an object $X$ of $\mathbf{C}$ to the free $\mathsf{R}$-algebra on the object $!_X\colon \emptyset\to X$ in $\mathbf{C^2}$, where $\emptyset$ is the initial object of $\mathbf{C}$.
\end{proof}

\begin{remark}
    In \cite{Garner_homomorphisms,Bourke_Garner_2}, the \emph{cofibrant replacement comonad} associated to an \textsc{awfs} $(\mathsf{L},\mathsf{R})$ on a category $\mathbf{C}$ with an initial object $\emptyset$ is defined by restricting the comonad $\mathsf{L}$ on $\mathbf{C^2}$ to $\emptyset/\mathbf{C}\cong \mathbf{C}$. 
    Under the situation of \cref{prop:loc-pres-ATF-initial}, the comonad $(Q,\varepsilon,\delta)$ induced from $\mathcal{I}$ as in \cref{prop:univ-cofib-repl-comonad} agrees with the cofibrant replacement comonad $(Q',\varepsilon',\delta')$ of the \textsc{awfs} $(\mathsf{L},\mathsf{R})$ cofibrantly generated by $\mathcal{I}$ (see the proof of \cref{prop:loc-pres-ATF-initial}).

    To see this, first observe that $\varepsilon'_X$ is obtained by applying the monad $\mathsf{R}$ to the object $!_X \colon \emptyset \to X$ in $\mathbf{C^2}$, which agrees with $\varepsilon_X$ as we have mentioned at the end of the proof of \cref{prop:loc-pres-ATF-initial}. Next, \cite[Proposition~2.8]{Garner_homomorphisms} (or \cite[(2.11)]{Bourke_Garner_1}) implies that $\delta'_X$ is determined as in (3) of \cref{def:alg-triv-fib}. 
	Finally, given a morphism $f \colon X \to Y$, the morphism $Q'f\colon Q'X\to Q'Y$ is determined as in (3) of \cref{def:alg-triv-fib}, since $(Q'f,f)$ is obtained by applying $\mathsf{R}$ to the morphism $(1_\emptyset,f) \colon {!}_X \to {!}_Y$ in $\mathbf{C}^\mathbf{2}$, and hence is a morphism in $\mathsf{R}\mhyphen\mathbf{Alg}\cong \ATF(\mathbf{C};\mathcal{I})$.
\end{remark}

Motivated by the above remark, in the situation of \cref{prop:loc-pres-ATF-initial}, we call the comonad $Q=(Q,\varepsilon,\delta)$ on $\mathbf{C}$ the \emph{cofibrant replacement comonad} with respect to $\mathcal{I}$.

\begin{definition}
	\label{defn:TheQ}
	Consider the monadic adjunction
	\begin{equation}
	\label{eqn:free-forgetful-adjunction}
	\begin{tikzpicture}[baseline=-\the\dimexpr\fontdimen22\textfont2\relax ]
				\node(0) at (0,0) {$\GSet$};
				\node(1) at (3,0) {$\WkCats{\omega}$};
				
				\draw [->,transform canvas={yshift=4pt}] (0) to node[auto, labelsize] 
				{$F$} (1); 
				\draw [<-,transform canvas={yshift=-4pt}] (0) to node[auto, 
				swap,labelsize] 
				{$U$} (1); 
		\path(0) to node[rotate=-90] {$\dashv$} (1);
	\end{tikzpicture}
	\end{equation}
	and set
	\[
		\mathcal{I'} = (\iota_n\colon \partial G^n \to G^n)_{n\in\mathbb{N}}
		\quad\text{ and }\quad
		\mathcal{I} = (F\iota_n\colon F\partial G^n \to FG^n)_{n\in\mathbb{N}}.
	\]
	We define $Q$ as the cofibrant replacement comonad on $\WkCats{\omega}$ with respect to $\mathcal{I}$.
	As before, we write $\overline\kappa_X$ for the canonical $\mathcal{I}$-lifting operation on the counit $\varepsilon_X\colon QX\to X$ for each $X\in\WkCats{\omega}$.
        Whenever we speak of a trivial fibration between weak $\omega$-categories, we mean that with respect to $\mathcal{I}$.
\end{definition}
	Note that an $\mathcal{I}$-lifting operation on a strict $\omega$-functor canonically corresponds to an $\mathcal{I'}$-lifting operation on its underlying globular map.
\begin{remark}
    Notice that $\mathcal{I'}$-lifting operations on globular maps are different from contractions on them, such as $\kappa$ on $\ar\colon L1\to T1$ in \cref{subsec:T,subsec:basic-operations} (see \cite[Section~9.1]{Leinster_book} for a general definition), since $n$ ranges over the set of nonnegative integers in the former, whereas the case $n=0$ is omitted in the latter.
\end{remark}

\begin{definition}[\cite{Garner_homomorphisms}]
\label{def:weak-functor}
    Let $X$ and $Y$ be weak $\omega$-categories. 
    A \emph{weak $\omega$-functor} from $X$ to $Y$ is a strict $\omega$-functor $QX\to Y$. 
    
    The category $\WkCat{\omega}$ of weak $\omega$-categories and weak $\omega$-functors is defined as the Kleisli category of the comonad $Q$. 
\end{definition}

The idea is that $QX$ is the \emph{weak $\omega$-functor classifier} for $X$.
We write a weak $\omega$-functor $p$ from $X$ to $Y$ as $p\colon X\rightsquigarrow Y$. 
By definition, it is a strict $\omega$-functor from $QX$ to $Y$; we call the latter the \emph{classifying strict $\omega$-functor} of $p$, and denote it by $\unds{p}\colon QX\to Y$.
The composite of weak $\omega$-functors $p\colon X\rightsquigarrow Y$ and $q\colon Y\rightsquigarrow Z$ is denoted by $q\Kcomp p\colon X\rightsquigarrow Z$, its classifying strict $\omega$-functor $\unds{(q\Kcomp p)}$ being the composite of
\[
\begin{tikzpicture}[baseline=-\the\dimexpr\fontdimen22\textfont2\relax ]
      \node(00) at (0,0) {$QX$};
      \node(01) at (6,0) {$Z$};
      \node(10) at (2,0) {$Q^2X$};
      \node(11) at (4,0) {$QY$};
    
      \draw [->] (11) to node[auto, labelsize] {$\unds{q}$} (01); 
      \draw [->] (00) to node[auto,labelsize] {$\delta_X$} (10); 
      \draw [->] (10) to node[auto,labelsize] {$Q\unds{p}$} (11);   
\end{tikzpicture}
\]
in $\WkCats{\omega}$. 

We have the identity-on-objects right adjoint functor $J\colon \WkCats{\omega}\to \WkCat{\omega}$ associated with the Kleisli category $\WkCat{\omega}$. It maps each strict $\omega$-functor $f\colon X\to Y$ to the weak $\omega$-functor $Jf\colon X\rightsquigarrow Y$ classified by $\unds{(Jf)}=f\circ \varepsilon_X\colon QX\to Y$. This allows us to view any strict $\omega$-functor $f$ as a weak $\omega$-functor $Jf$. 

\subsection{Weak \texorpdfstring{$\omega$}{ω}-weak equivalences}
\begin{definition}
    Let $X$ and $Y$ be weak $\omega$-categories. A \emph{weak $\omega$-weak equivalence} from $X$ to $Y$ is a weak $\omega$-functor $p\colon X\rightsquigarrow Y$ whose classifying strict $\omega$-functor $\unds{p}\colon QX\to Y$ is a strict $\omega$-weak equivalence.
\end{definition}

A strict $\omega$-functor $f\colon X\to Y$ is a strict $\omega$-weak equivalence if and only if the corresponding weak $\omega$-functor $Jf\colon X\rightsquigarrow Y$ is a weak $\omega$-weak equivalence. To see this, first observe that $\varepsilon_X$ is a strict $\omega$-weak equivalence because it is a trivial fibration; any trivial fibration is a strict $\omega$-weak equivalence since the binary relation $\sim$ on the set of cells of a weak $\omega$-category, defined by the existence of an invertible cell (see \cref{def:invertible}), is reflexive (see \cref{cor:eq-rel}). Therefore \cref{thm:2-out-of-3} shows that $f$ is a strict $\omega$-weak equivalence if and only if $\unds{(Jf)}=f\circ \varepsilon_X$ is. 

The comonad $(Q,\varepsilon,\delta)$ on $\WkCats{\omega}$ is compatible with the strict $\omega$-weak equivalences in the following sense. 

\begin{proposition}
\label{prop:Q-weak-eq}
For each weak $\omega$-category $X$, the components $\varepsilon_X\colon QX\to X$ and $\delta_X\colon QX\to Q^2X$ are strict $\omega$-weak equivalences, and the functor $Q\colon \WkCats{\omega}\to\WkCats{\omega}$ preserves and reflects strict $\omega$-weak equivalences.
\end{proposition}
\begin{proof}
We have already seen that $\varepsilon_X$ is a strict $\omega$-weak equivalence.
One of the axioms for a comonad states that the diagram 
\[
\begin{tikzpicture}[baseline=-\the\dimexpr\fontdimen22\textfont2\relax ]
      \node(00) at (0,1) {$QX$};
      \node(01) at (2,1) {$Q^2X$};
      \node(10) at (2,-1) {$QX$};
      
      \draw [->] (00) to node[auto, labelsize] {$\delta_X$} (01); 
      \draw [->] (01) to node[auto, labelsize] {$\varepsilon_{QX}$} (10); 
      \draw [->] (00) to node[auto,swap,labelsize] {$1_{QX}$} (10); 
\end{tikzpicture}
\]
commutes. Since $1_{QX}$ and $\varepsilon_{QX}$ are strict $\omega$-weak equivalences, so is $\delta_X$ by \cref{thm:2-out-of-3}. Moreover, for any strict $\omega$-functor $f\colon X\to Y$, the diagram 
    \[
\begin{tikzpicture}[baseline=-\the\dimexpr\fontdimen22\textfont2\relax ]
      \node(00) at (0,1) {$QX$};
      \node(01) at (0,-1) {$X$};
      \node(10) at (2,1) {$QY$};
      \node(11) at (2,-1) {$Y$};
      
      \draw [->] (00) to node[auto,swap, labelsize] {$\varepsilon_X$} (01); 
      \draw [->] (01) to node[auto,swap, labelsize] {$f$} (11); 
      \draw [->] (00) to node[auto,labelsize] {$Qf$} (10); 
      \draw [->] (10) to node[auto,labelsize] {$\varepsilon_Y$} (11);   
\end{tikzpicture}
\]
commutes. Hence by \cref{thm:2-out-of-3}, $f$ is a strict $\omega$-weak equivalence if and only if $Qf$ is. 
\end{proof}

Now it is straightforward to show that the class of weak $\omega$-weak equivalences has the 2-out-of-3 property (in the category $\WkCat{\omega}$). 

\begin{proposition}[2-out-of-3 for weak $\omega$-weak equivalences]
\label{prop:2-out-of-3-for-weak-omega-weak-eq}
    Let $p \colon X \rightsquigarrow Y$ and $q \colon Y \rightsquigarrow Z$ be weak $\omega$-functors between weak $\omega$-categories.
    If two of $p$, $q$, and $q\Kcomp p$ are weak $\omega$-weak equivalences, then so is the third.
\end{proposition}
\begin{proof}
    Consider the commutative diagram
\[
\begin{tikzpicture}[baseline=-\the\dimexpr\fontdimen22\textfont2\relax ]
      \node(00) at (0,1) {$QX$};
      \node(01) at (2,1) {$Z$};
      \node(10) at (0,-1) {$Q^2X$};
      \node(11) at (2,-1) {$QY$};
      
      \draw [->] (00) to node[auto, labelsize] {$\unds{(q\Kcomp p)}$} (01); 
      \draw [<-] (01) to node[auto, labelsize] {$\unds{q}$} (11); 
      \draw [->] (00) to node[auto,swap,labelsize] {$\delta_X$} (10); 
      \draw [->] (10) to node[auto,swap,labelsize] {$Q\unds{p}$} (11);   
\end{tikzpicture}
\]
    in $\WkCats{\omega}$. It is easy to see that if two of $\unds{p}$, $\unds{q}$, and $\unds{(q\Kcomp p)}$ are strict $\omega$-weak equivalences, then so is the third, using \cref{prop:Q-weak-eq} and \cref{thm:2-out-of-3}.
\end{proof}

\begin{corollary}[2-out-of-6 for weak $\omega$-weak equivalences]
    Let $p \colon X \rightsquigarrow Y$, $q \colon Y \rightsquigarrow Z$, and $r \colon Z \rightsquigarrow W$ be weak $\omega$-functors between weak $\omega$-categories.
    Suppose that $q\Kcomp p$ and $r\Kcomp q$ are weak $\omega$-weak equivalences.
    Then $p,q,r$, and $r\Kcomp q\Kcomp p$ are also weak $\omega$-weak equivalences.
\end{corollary}
\begin{proof}
    We will show that $r\Kcomp q\Kcomp p$ is a weak $\omega$-weak equivalence; the rest will follow by the 2-out-of-3 property for weak $\omega$-weak equivalences.

    The classifying strict $\omega$-functor $\unds{(r\Kcomp q\Kcomp p)}$ of $r\Kcomp q\Kcomp p$ is the composite of 
    \[
\begin{tikzpicture}[baseline=-\the\dimexpr\fontdimen22\textfont2\relax ]
      \node(0) at (0,0) {$QX$};
      \node(1) at (2,0) {$Q^2X$};
      \node(2) at (4,0) {$Q^3X$};
      \node(3) at (6,0) {$Q^2Y$};
      \node(4) at (8,0) {$QZ$};
      \node(5) at (10,0){$W$};
    
      \draw [->] (0) to node[auto,labelsize] {$\delta_X$} (1); 
      \draw [->] (1) to node[auto,labelsize] {$Q\delta_X$} (2); 
      \draw [->] (2) to node[auto,labelsize] {$Q^2\unds{p}$} (3); 
      \draw [->] (3) to node[auto,labelsize] {$Q\unds{q}$} (4); 
      \draw [->] (4) to node[auto,labelsize] {$\unds{r}$} (5); 
\end{tikzpicture}
\]
in $\WkCats{\omega}$. 
Since $\delta_X$ and $Q\delta_X$ are strict $\omega$-weak equivalences by \cref{prop:Q-weak-eq}, it suffices to show that $\unds{r}\circ Q\unds{q}\circ Q^2\unds{p}$ is a strict $\omega$-weak equivalence. 
Since $q\Kcomp p$ is a weak $\omega$-weak equivalence, $\unds{q}\circ Q\unds{p}$ is a strict $\omega$-weak equivalence. Hence so is $Q\unds{q}\circ Q^2\unds{p}$, by \cref{prop:Q-weak-eq}. Since $r\Kcomp q$ is a weak $\omega$-weak equivalence, $\unds{r}\circ Q\unds{q}$ is a strict $\omega$-weak equivalence. 
Therefore $\unds{r}\circ Q\unds{q}\circ Q^2\unds{p}$ is a strict $\omega$-weak equivalence by \cref{cor:2-out-of-6}.
\end{proof}

We also have a characterisation of weak $\omega$-weak equivalences along the lines of \cref{prop:omega-weak-eq-coinductively}.
We first note the following fact.
\begin{proposition}
\label{prop:epsilonBij}
For each weak $\omega$-categry $X$, the strict $\omega$-functor $\varepsilon_X\colon QX\to X$ is bijective on objects.
\end{proposition}
\begin{proof}
    See \cref{apx:epsilon-bo}.
\end{proof}
From now on, we identify the objects of $QX$ with those of $X$ via $\varepsilon_X$.
We recall from \cite[Section~5]{Cottrell_Fujii_hom} that, given any weak $\omega$-functor $p\colon X\rightsquigarrow Y$ between weak $\omega$-categories and $x,x'\in X_0$, we obtain a weak $\omega$-functor $p_{x,x'}\colon X(x,x')\rightsquigarrow Y(px,px')$ between the hom weak $\omega$-categories as follows.
For each $x,x'\in X_0$, we obtain an algebraic trivial fibration
\[
\bigl((\varepsilon_X)_{x,x'}\colon (QX)(x,x')\to X(x,x'), (\overline{\kappa}_X)_{x,x'}\bigr)
\]
by restricting $(\varepsilon_X,\overline\kappa_X)$. 
Indeed, a commutative square
\[
\begin{tikzcd}
    F\partial G^n
    \ar[r]
    \ar[d,"F\iota_n"'] &
    (QX)(x,x')
    \ar[d,"{(\varepsilon_X)_{x,x'}}"]\\
    FG^n
    \ar[r]&
    X(x,x')
\end{tikzcd}
\]
can be identified with a commutative square
\[
\begin{tikzcd}
    F\partial G^{n+1}
    \ar[r,"{\ppair{u,v}}"]
    \ar[d,"F\iota_{n+1}"']
        &
        QX
        \ar[d,"{\varepsilon_X}"]
    \\
    FG^{n+1}
    \ar[r,"w", swap]
        &
        X
\end{tikzcd}
\]
with $s^X_0(u)=s^X_0(v)=s^X_0(w)=x$ and $t^X_0(u)=t^X_0(v)=t^X_0(w)=x'$.
The latter square admits a lift
\[
k = \overline\kappa_X\bigl({n+1};\ppair{u,v},w\bigr)
\]
which necessarily satisfies $s^X_0(k)=x$ and $t^X_0(k)=x'$, and we can set the lift of the former square to be the map corresponding to this $k$.

On the other hand, the initial algebraic trivial fibration over $X(x,x')$ is
\[
\bigl(\varepsilon_{X(x,x')}\colon Q\bigl(X(x,x')\bigr)\to X(x,x'),\overline \kappa_{X(x,x')}\bigr),
\]
and hence we obtain a unique strict $\omega$-functor $(\alpha_X)_{x,x'}\colon Q\bigl(X(x,x')\bigr)\to (QX)(x,x')$ such that 
\[
\bigl((\alpha_X)_{x,x'},1_{X(x,x')}\bigr)\colon \bigl(\varepsilon_{X(x,x')},\overline \kappa_{X(x,x')}\bigr)\to\bigl((\varepsilon_X)_{x,x'}, (\overline{\kappa}_X)_{x,x'}\bigr)
\]
is a morphism of algebraic trivial fibrations. 
In particular, $(\alpha_X)_{x,x'}$ makes the diagram
\[
\begin{tikzpicture}[baseline=-\the\dimexpr\fontdimen22\textfont2\relax ]
      \node(00) at (0,1) {$Q\bigl(X(x,x')\bigr)$};
      \node(01) at (2,-1) {$X(x,x')$};
      \node(10) at (4,1) {$(QX)(x,x')$};
      
      \draw [->] (00) to node[auto,swap, labelsize] {$\varepsilon_{X(x,x')}$} (01);  
      \draw [->] (00) to node[auto,labelsize] {$(\alpha_X)_{x,x'}$} (10); 
      \draw [->] (10) to node[auto,labelsize] {$(\varepsilon_X)_{x,x'}$} (01);   
\end{tikzpicture}
\]
commute, and hence is a strict $\omega$-weak equivalence by \cref{prop:omega-weak-eq-coinductively,thm:2-out-of-3}. The weak $\omega$-functor $p_{x,x'}\colon X(x,x')\rightsquigarrow Y(px,px')$ is defined by setting its classifying strict $\omega$-functor $\unds{(p_{x,x'})}$ to be the composite of 
\[
\begin{tikzpicture}[baseline=-\the\dimexpr\fontdimen22\textfont2\relax ]
      \node(0) at (0,0) {$Q\bigl(X(x,x')\bigr)$};
      \node(1) at (3,0) {$(QX)(x,x')$};
      \node(2) at (6,0) {$Y(px,px')$.};
    
      \draw [->] (0) to node[auto,labelsize] {$(\alpha_X)_{x,x'}$} (1); 
      \draw [->] (1) to node[auto,labelsize] {$(\unds{p})_{x,x'}$} (2); 
\end{tikzpicture}
\]

We now give the characterisation of weak $\omega$-weak equivalences alluded to just above \cref{prop:epsilonBij}.
We say that a weak $\omega$-functor is \emph{essentially $0$-surjective} if its classifying strict $\omega$-functor is so. 
\begin{proposition}
\label{prop:weak-omega-weak-eq-coinductively}
A weak $\omega$-functor $p\colon X\rightsquigarrow Y$ between weak $\omega$-categories is a weak $\omega$-weak equivalence if and only if it is essentially $0$-surjective and the induced weak $\omega$-functor $p_{x,x'}\colon X(x,x')\rightsquigarrow Y(px,px')$ is a weak $\omega$-weak equivalence for each $x,x'\in X_0$.    
\end{proposition}
\begin{proof}
    Since $(\alpha_X)_{x,x'}$ is a strict $\omega$-weak equivalence, $(\unds{p})_{x,x'}$ is a strict $\omega$-weak equivalence if and only if $\unds{(p_{x,x'})}$ is, by \cref{thm:2-out-of-3}. Therefore the claim follows from \cref{prop:omega-weak-eq-coinductively}.
\end{proof}

\begin{remark}
    Just like the strict case mentioned in \cref{rmk:omega-weak-eq-coinductively}, 
    the statement of \cref{prop:weak-omega-weak-eq-coinductively} can be taken as a coinductive definition of weak $\omega$-weak equivalences.
    Define the monotone map
    \[
    \Psi\colon \mathcal{P}\bigl(\mathrm{mor}(\WkCat{\omega})\bigr)\to\mathcal{P}\bigl(\mathrm{mor}(\WkCat{\omega})\bigr)
    \]
    by mapping $S\subseteq \mathrm{mor}(\WkCat{\omega})$ to the set $\Psi(S)$ of all weak $\omega$-functors which are essentially $0$-surjective and locally in $S$, i.e.,
    \begin{multline*}
    \Psi(S)=\bigl\{\, (p\colon X\rightsquigarrow Y)\in\mathrm{mor}(\WkCat{\omega})  \,\big\vert\, \text{$p$ is essentially $0$-surjective and}\\
    \text{we have $\bigl(p_{x,x'}\colon X(x,x')\rightsquigarrow Y(fx,fx')\bigr)\in S$ for each $x,x'\in X_0$}\,\bigr\}.
    \end{multline*}
    Then the set $E$ of all weak $\omega$-weak equivalences is the largest (post-)fixed point $\nu\Psi$ of $\Psi$.  
    Indeed, \cref{prop:weak-omega-weak-eq-coinductively} says that we have $E=\Psi(E)$, which implies $E\subseteq \nu\Psi$.
    On the other hand, given $S\subseteq \mathrm{mor}(\WkCat{\omega})$ with $S\subseteq \Psi(S)$, one can show that the classifying strict $\omega$-functor $\unds{p}$ of each $p\in S$ is essentially $n$-surjective, by induction on $n\in\N$. 
    (In the inductive step, apply \cref{prop:2-out-of-3-g} to the commutative triangle 
\[
\begin{tikzpicture}[baseline=-\the\dimexpr\fontdimen22\textfont2\relax ]
      \node(0) at (0,-1) {$Q\bigl(X(x,x')\bigr)$};
      \node(1) at (2,1) {$(QX)(x,x')$};
      \node(2) at (4,-1) {$Y(px,px')$};
    
      \draw [->] (0) to node[auto,labelsize] {$(\alpha_X)_{x,x'}$} (1); 
      \draw [->] (1) to node[auto,labelsize] {$(\unds{p})_{x,x'}$} (2); 
      \draw [->] (0) to node[auto,labelsize] {$\unds{(p_{x,x'})}$} (2); 
\end{tikzpicture}
\]
    in $\WkCats{\omega}$.)
    Therefore we have $S\subseteq E$, and hence $E$ is the greatest (post-)fixed point $\nu\Psi$ of $\Psi$.
\end{remark}

\subsection{Weak \texorpdfstring{$\omega$}{ω}-weak equivalences via underlying globular maps}
\label{underlying}
Let us consider the relationship between the weak $\omega$-weak equivalences and similar known notions in low dimensions, such as \emph{biequivalences} between bicategories and \emph{triequivalences} between tricategories. 
Recall that in low dimensions, notions of weak functors (such as \emph{homomorphisms} of bicategories \cite[Section~4]{Benabou_bicat} or \emph{trihomomorphisms} of tricategories \cite[Definition~3.2]{GPS-tricat}) are commonly defined as globular maps between the underlying globular sets equipped with coherence cells.
(We define the underlying globular set of a weak $n$-dimensional category $X$ for $n\in\N$ by setting $s^X_k,t^X_k\colon X_{k+1}\to X_k$ to be the identity function for all $k\geq n$.)
For example, a homomorphism $p\colon X\rightsquigarrow Y$ of bicategories is a globular map $\undg{p}\colon X\to Y$ between the underlying globular sets equipped with coherence 2-cells, such as an invertible 2-cell $\undg{p}(u)\comp{0}{Y} \undg{p}(u) \to \undg{p}(u\comp{0}{X} v)$ in $Y$ for each composable pair of 1-cells $(u,v)$ in $X$. 
However, the definition of biequivalence only refers to the underlying globular map of a homomorphism: a homomorphism $p\colon X\rightsquigarrow Y$ is a biequivalence if and only if its underlying globular map $\undg{p}\colon X\to Y$ is essentially $3$-surjective \cite[(1.33)]{Street-fib-bicat}, which implies that $\undg{p}$ is essentially $\omega$-surjective. Similarly, a trihomomorphism $p\colon X\rightsquigarrow Y$ is a triequivalence if and only if its underlying globular map $\undg{p}\colon X\to Y$ is essentially 4-surjective \cite[Definition~3.5]{GPS-tricat} (or equivalently, is essentially $\omega$-surjective).  

We characterise weak $\omega$-weak equivalences in a similar fashion: we define the \emph{underlying globular map} $\undg{p}\colon X\to Y$ of a weak $\omega$-functor $p\colon X\rightsquigarrow Y$ (which is a genuine globular map $X \to Y$ and is in particular different from the classifying strict $\omega$-functor $\unds{p}\colon QX\to Y$ of $p$), and show that a weak $\omega$-functor $p\colon X\rightsquigarrow Y$ is a weak $\omega$-weak equivalence if and only if its underlying globular map $\undg{p}\colon X\to Y$ is essentially $\omega$-surjective.

Hence our first goal is to define the operation of taking the underlying globular map of a weak $\omega$-functor. We would like this operation to be functorial, i.e., giving rise to a functor $V \colon \WkCat{\omega}\to\GSet$. 
This turns out to be a consequence of an appropriate ``change of base'' result for the notions explained in \cref{def:alg-triv-fib} and \cref{prop:univ-cofib-repl-comonad}. (The same goal can also be achieved via the theory of \textsc{awfs}; see \cref{rmk:sigma-via-awfs}.)

\begin{definition}
\label{def:base-change-for-Q}
    Let 
         \[
    	\begin{tikzpicture}[baseline=-\the\dimexpr\fontdimen22\textfont2\relax ]
    				\node(0) at (0,0) {$\mathbf{D}$};
    				\node(1) at (3,0) {$\mathbf{C}$};
    				
    				\draw [->,transform canvas={yshift=4pt}] (0) to node[auto, labelsize] 
    				{$F$} (1); 
    				\draw [<-,transform canvas={yshift=-4pt}] (0) to node[auto, 
    				swap,labelsize] 
    				{$U$} (1); 
    		\path(0) to node[rotate=-90] {$\dashv$} (1);
    	\end{tikzpicture}
     \]
    be an adjunction between categories, and let $\mathcal{I'}=(\iota_i)_{i\in I}$ be an indexed family of morphisms in $\mathbf{D}$. We define the indexed family $\mathcal{I}$ of morphisms in $\mathbf{C}$ as $\mathcal{I}=(F\iota_i)_{i\in{I}}$. 
    \begin{enumerate}
        \item We have a functor $\ATF(U)\colon \ATF(\mathbf{C};\mathcal{I})\to \ATF(\mathbf{D};\mathcal{I'})$ defined by mapping each $\mathcal{I}$-algebraic trivial fibration $(f\colon X\to Y,\overline\kappa)$ to the $\mathcal{I'}$-algebraic trivial fibration $(Uf\colon UX\to UY, \overline\kappa^\dag)$, where the $\mathcal{I'}$-lifting operation $\overline\kappa^\dag$ on $Uf$ is determined by the following correspondence under the adjunction $F\dashv U$. 
        \begin{equation*}
\begin{tikzpicture}[baseline=-\the\dimexpr\fontdimen22\textfont2\relax ]
      \node(00) at (0,1) {$\bullet$};
      \node(01) at (2,1) {$UX$};
      \node(10) at (0,-1) {$\bullet$};
      \node(11) at (2,-1) {$UY$};
      
      \draw [->] (00) to node[auto, labelsize] {$u^\dag$} (01); 
      \draw [->] (01) to node[auto, labelsize] {$Uf$} (11); 
      \draw [->] (00) to node[auto,swap,labelsize] {$\iota_i$} (10); 
      \draw [->] (10) to node[auto,swap,labelsize] {$v^\dag$} (11);   
      \draw [->, dashed] (10) to node[midway,fill=white,labelsize] {$\overline \kappa^\dag(i;u^\dag,v^\dag)$} (01);
\end{tikzpicture}\qquad
\begin{tikzpicture}[baseline=-\the\dimexpr\fontdimen22\textfont2\relax ]
      \node(00) at (0,1) {$\bullet$};
      \node(01) at (2,1) {$X$};
      \node(10) at (0,-1) {$\bullet$};
      \node(11) at (2,-1) {$Y$};
      
      \draw [->] (00) to node[auto, labelsize] {$u$} (01); 
      \draw [->] (01) to node[auto, labelsize] {$f$} (11); 
      \draw [->] (00) to node[auto,swap,labelsize] {$F\iota_i$} (10); 
      \draw [->] (10) to node[auto,swap,labelsize] {$v$} (11);   
      \draw [->, dashed] (10) to node[midway,fill=white,labelsize] {$\overline \kappa(i;u,v)$} (01);
\end{tikzpicture}
\end{equation*}
    Notice that the following square is a pullback in $\mathbf{CAT}$: 
         \begin{equation}
         \label{eqn:proj-lift-CAT}
\begin{tikzpicture}[baseline=-\the\dimexpr\fontdimen22\textfont2\relax ]
      \node(00) at (0,1) {$\ATF(\mathbf{C};\mathcal{I})$};
      \node(01) at (3,1) {$\ATF(\mathbf{D};\mathcal{I'})$};
      \node(10) at (0,-1) {$\mathbf{C}^\mathbf{2}$};
      \node(11) at (3,-1) {$\mathbf{D}^\mathbf{2}$.};
      \draw  (0.3,0.5) to (0.5,0.5) to (0.5,0.7); 
      \draw [->] (00) to node[auto, labelsize] {$\ATF(U)$} (01); 
      \draw [->] (01) to node[auto, labelsize] {forgetful} (11); 
      \draw [->] (00) to node[auto,swap,labelsize] {forgetful} (10); 
      \draw [->] (10) to node[auto,swap,labelsize] {$U^\mathbf{2}$} (11);
\end{tikzpicture}
\end{equation}
    \item Now suppose that both functors $\cod\colon \ATF(\mathbf{C};\mathcal{I})\to \mathbf{C}$ and $\cod\colon \ATF(\mathbf{D};\mathcal{I'})\to \mathbf{D}$ have laris. As in (3) of \cref{def:alg-triv-fib}, we write $(\varepsilon_X\colon QX\to X,\overline{\kappa}_X)$ and $(\varepsilon'_A\colon Q'A\to A,\overline\kappa'_A)$ for the values under these laris of objects $X\in \mathbf{C}$ and $A\in \mathbf{D}$, respectively. For each object $X\in\mathbf{C}$, we define $\sigma_X\colon Q'UX\to UQX$ as the unique morphism in $\mathbf{D}$ such that
    \begin{equation}
    \begin{tikzpicture}[baseline=-\the\dimexpr\fontdimen22\textfont2\relax ]
          \node(01) at (0,0.75) {$Q'UX$};
          \node(11) at (0,-0.75) {$UX$};
          \node(02) at (2,0.75) {$UQX$};
          \node(12) at (2,-0.75) {$UX$};
          \draw [->] (01) to node[auto, labelsize] {$\sigma_X$} (02); 
          \draw [->] (01) to node[auto,swap, labelsize] {$(\varepsilon'_{UX},\overline\kappa'_{UX})$} (11); 
          \draw [->] (02) to node[auto, labelsize] {$(U\varepsilon_X,\overline\kappa_X^\dagger)$} (12); 
          \draw [->] (11) to node[auto,swap,labelsize] {$1_{UX}$} (12);
    \end{tikzpicture}\tag{S}\label{eqn:sigma}
    \end{equation}
    is a morphism in $\ATF(\mathbf{D};\mathcal{I'})$. \qedhere
    \end{enumerate}
\end{definition}

\begin{proposition}
    \label{prop:base-change-for-Q}
    In the situation of (2) of \cref{def:base-change-for-Q}, the family $(\sigma_X)_{X\in\mathbf{C}}$ is a natural transformation  	
    \[\begin{tikzpicture}[baseline=-\the\dimexpr\fontdimen22\textfont2\relax ]
		\node(00) at (0,1) {$\mathbf{C}$};
		\node(01) at (3,1) {$\mathbf{C}$};
		\node(10) at (0,-1) {$\mathbf{D}$};
		\node(11) at (3,-1) {$\mathbf{D}$};
		
		\draw [->] (00) to node[auto, labelsize] {$Q$} (01); 
		\draw [->] (01) to node[auto, labelsize] {$U$} (11); 
		\draw [->] (00) to node[auto,swap,labelsize] {$U$} (10); 
		\draw [->] (10) to node[auto,swap,labelsize] {$Q'$} (11);   
		\draw [2cell] (1.5,-0.3) to node[auto,swap,labelsize] {$\sigma$} (1.5,0.3); 
\end{tikzpicture}\]
	such that the pair $(U,\sigma)$ is a comonad opfunctor \cite[Section~4]{Street_FTM} from $(\mathbf{C},Q)$ to $(\mathbf{D},Q')$.
\end{proposition}
\begin{proof}
    As in the proof of \cref{prop:univ-cofib-repl-comonad}, one can verify the naturality of $\sigma$ and the compatibility of $\sigma$ with the comonads $Q$ and $Q'$ using the universality of $(\varepsilon'_{UX},\overline\kappa_{UX}')$; below are the details.

    The compatibility of $\sigma$ with the counits (i.e., that $\varepsilon'_{UX}= U\varepsilon_X\circ \sigma_X$ holds for any $X\in\mathbf{C}$) is immediate from the definition of $\sigma_X$. 

    To show the naturality of $\sigma$, it suffices to show that for any morphism $f\colon X\to Y$ in $\mathbf{C}$, both $(\sigma_Y\circ Q'Uf,Uf)$ and $(UQf\circ \sigma_X,Uf)$ are morphisms $(\varepsilon'_{UX},\overline\kappa'_{UX})\to (U\varepsilon_X,\overline\kappa^\dag_X)$ in $\ATF(\mathbf{D};\mathcal{I'})$. These can be seen by the following diagrams, which can be regarded either as composition of morphisms in the category $\ATF(\mathbf{D};\mathcal{I'})$ or horizontal composition of squares in the double category $\AATF(\mathbf{D};\mathcal{I'})$ introduced in the proof of \cref{prop:univ-cofib-repl-comonad}.
       \begin{equation*}
        \begin{tikzpicture}[baseline=-\the\dimexpr\fontdimen22\textfont2\relax ]
              \node(00) at (0,0.75)   {$Q'UX$};
              \node(01) at (0,-0.75)  {$UX$};
              \node(10) at (2,0.75) {$Q'UY$};
              \node(11) at (2,-0.75)   {$UY$};
              \node(20) at (4,0.75)   {$UQY$};
              \node(21) at (4,-0.75)     {$UY$};
              \draw [->] (00) to node[auto,swap, labelsize] {$(\varepsilon'_{UX},\overline\kappa'_{UX})$} (01); 
              \draw [->] (10) to node[midway,fill=white, labelsize] {$(\varepsilon'_{UY},\overline\kappa'_{UY})$} (11); 
              \draw [->] (20) to node[auto, labelsize] {$(U\varepsilon_{Y},\overline\kappa_{Y}^\dag)$} (21);  
              \draw [->] (00) to node[auto, labelsize] {$Q'Uf$} (10); 
              \draw [->] (10) to node[auto, labelsize] {$\sigma_Y$} (20); 
              \draw [->] (01) to node[auto,swap, labelsize] {$Uf$} (11); 
              \draw [->] (11) to node[auto,swap, labelsize] {$1_{UY}$} (21);  
              \node at (1,0) {(\ref{eqn:Qf})};
              \node at (3,0) {(\ref{eqn:sigma})};
        \end{tikzpicture}
        \qquad
        \begin{tikzpicture}[baseline=-\the\dimexpr\fontdimen22\textfont2\relax ]
              \node(00) at (0,0.75)   {$Q'UX$};
              \node(01) at (0,-0.75)  {$UX$};
              \node(10) at (2,0.75) {$UQX$};
              \node(11) at (2,-0.75)   {$UX$};
              \node(20) at (4,0.75)   {$UQY$};
              \node(21) at (4,-0.75)     {$UY$};
              \draw [->] (00) to node[auto,swap, labelsize] {$(\varepsilon'_{UX},\overline\kappa'_{UX})$} (01); 
              \draw [->] (10) to node[midway,fill=white, labelsize] {$(U\varepsilon_{X},\overline\kappa_{X}^\dag)$} (11); 
              \draw [->] (20) to node[auto, labelsize] {$(U\varepsilon_{Y},\overline\kappa_{Y}^\dag)$} (21);  
              \draw [->] (00) to node[auto, labelsize] {$\sigma_X$} (10); 
              \draw [->] (10) to node[auto, labelsize] {$UQf$} (20); 
              \draw [->] (01) to node[auto,swap, labelsize] {$1_{UX}$} (11); 
              \draw [->] (11) to node[auto,swap, labelsize] {$Uf$} (21); \node at (1,0) {(\ref{eqn:sigma})};
              \node at (3,0) {(\ref{eqn:Qf})};
        \end{tikzpicture}
        \end{equation*}

        Finally, the compatibility of $\sigma$ with the comultiplications means that for any object $X\in \mathbf{C}$, the diagram 
    \begin{equation}
    \label{eqn:sigma-delta}
		\begin{tikzcd}
			Q'UX
			\ar[rr,"\sigma_X"]
			\ar[d,"\delta'_{UX}"']
				&
					&
					UQX
					\ar[d,"U\delta_X"]
			\\
			Q'^2UX
			\ar[r,"Q'\sigma_X"']
				&
				Q'UQX
				\ar[r,"\sigma_{QX}"']
					&
					UQ^2X
		\end{tikzcd}
    \end{equation}
    commutes.
    To see this, observe that the two composites in \cref{eqn:sigma-delta} lift to the top horizontal composites in the following diagrams in the double category $\AATF(\mathbf{D};\mathcal{I'})$:
    \[
    \begin{tikzpicture}[baseline=-\the\dimexpr\fontdimen22\textfont2\relax ]
              \node(00) at (0,1.5)   {$Q'UX$};
              \node(01) at (0,-1.5)  {$UX$};
              \node(10) at (2,1.5) {$UQX$};
              \node(11) at (2,-1.5)   {$UX$};
              \node(20) at (4,1.5)   {$UQ^2X$};
              \node(22) at (4,0)     {$UQX$};
              \node(21) at (4,-1.5)     {$UX$};
              \draw [->] (00) to node[auto,swap, labelsize] {$(\varepsilon'_{UX},\overline\kappa'_{UX})$} (01); 
              \draw [->] (10) to node[midway,fill=white, labelsize] {$(U\varepsilon_{X},\overline\kappa_{X}^\dag)$} (11); 
              \draw [->] (20) to node[auto, labelsize] {$(U\varepsilon_{QX},\overline{\kappa}^\dag_{QX})$} (22);
              \draw [->] (22) to node[auto, labelsize] {$(U\varepsilon_{X},\overline{\kappa}^\dag_{X})$} (21);  
              \draw [->] (00) to node[auto, labelsize] {$\sigma_X$} (10); 
              \draw [->] (10) to node[auto, labelsize] {$U\delta_X$} (20); 
              \draw [->] (01) to node[auto,swap, labelsize] {$1_{UX}$} (11); 
              \draw [->] (11) to node[auto,swap, labelsize] {$1_{UX}$} (21); 
              \node at (1,0) {(\ref{eqn:sigma})};
              \node at (3,0) {(\ref{eqn:delta})};
    \end{tikzpicture}
    \]
    \[
    \begin{tikzpicture}[baseline=-\the\dimexpr\fontdimen22\textfont2\relax ]
              \node(00) at (0,1.5)   {$Q'UX$};
              \node(01) at (0,-1.5)  {$UX$};
              \node(10) at (3,1.5) {$Q'^2UX$};
              \node(11) at (3,-1.5)   {$UX$};
              \node(12) at (3,0)   {$Q'UX$};
              \node(20) at (6,1.5)   {$Q'UQX$};
              \node(21) at (6,-1.5)     {$UX$};
              \node(22) at (6,0)  {$UQX$};
              \node(30) at (9,1.5)   {$UQ^2X$};
              \node(31) at (9,-1.5)     {$UX$};
              \node(32) at (9,0)     {$UQX$};
              \draw [->] (00) to node[auto,swap, labelsize] {$(\varepsilon'_{UX},\overline\kappa'_{UX})$} (01); 
              \draw [->] (10) to node[midway,fill=white, labelsize] {$(\varepsilon'_{Q'UX},\overline\kappa'_{Q'UX})$} (12);
              \draw [->] (12) to node[midway,fill=white, labelsize] {$(\varepsilon'_{UX},\overline\kappa'_{UX})$} (11); 
              \draw [->] (20) to node[midway, fill=white, labelsize] {$(\varepsilon'_{UQX},\overline\kappa'_{UQX})$} (22);  
              \draw [->] (22) to node[midway, fill=white, labelsize] {$(U\varepsilon_{X},\overline\kappa_{X}^\dag)$} (21); 
              \draw [->] (30) to node[auto, labelsize] {$(U\varepsilon_{QX},\overline\kappa_{QX}^\dag)$} (32);
              \draw [->] (32) to node[auto, labelsize] {$(U\varepsilon_{X},\overline\kappa_{X}^\dag)$} (31); 
              \draw [->] (00) to node[auto, labelsize] {$\delta'_{UX}$} (10); 
              \draw [->] (10) to node[auto, labelsize] {$Q'\sigma_X$} (20); 
              \draw [->] (12) to node[auto, labelsize] {$\sigma_X$} (22); 
              \draw [->] (20) to node[auto, labelsize] {$\sigma_{QX}$} (30);
              \draw [->] (22) to node[auto, labelsize] {$1$} (32); 
              \draw [->] (01) to node[auto,swap, labelsize] {$1_{UX}$} (11); 
              \draw [->] (11) to node[auto,swap, labelsize] {$1_{UX}$} (21);  
              \draw [->] (21) to node[auto,swap, labelsize] {$1_{UX}$} (31); 
              \node at (1.5,0) {(\ref{eqn:delta})};
              \node at (4.5,0.75) {(\ref{eqn:Qf})};
              \node at (4.5,-0.75) {(\ref{eqn:sigma})};
              \node at (7.5,0.75) {(\ref{eqn:sigma})};
              \node at (7.5,-0.75) {$=$};
    \end{tikzpicture}
    \]
    (The square labelled by (\ref{eqn:delta}) in the first diagram indeed exists in $\AATF(\mathbf{D};\mathcal{I'})$ because $U$ induces a double functor
    \[
    \AATF(U)\colon \AATF(\mathbf{C};\mathcal{I})\to\AATF(\mathbf{D};\mathcal{I'})
    \]
    (cf.~\cite[Proposition~21]{Bourke_Garner_1}),
    and hence it preserves the vertical composition of algebraic trivial fibrations.)
    This means that the two composites of \cref{eqn:sigma-delta} define parallel morphisms in $\ATF(\mathbf{D};\mathcal{I'})_{UX}$ from the initial object (see \cref{rmk:fibres-of-ATF}), and hence they must be equal.
\end{proof}

\begin{remark}
\label{rmk:sigma-via-awfs}
    Let us consider the situation of \cref{def:base-change-for-Q}, and assume moreover that the categories $\mathbf{C}$ and $\mathbf{D}$ are locally presentable and the indexing set $I$ is small.
    Then by \cref{prop:loc-pres-ATF-initial}, the assumptions in (2) of \cref{def:base-change-for-Q} are satisfied. 
    We now explain that in this case, the result of \cref{prop:base-change-for-Q} can also be obtained via the theory of \textsc{awfs}.
    
    As explained in the proof of \cref{prop:loc-pres-ATF-initial}, there exist an \textsc{awfs} $(\mathsf{L},\mathsf{R})$ on $\mathbf{C}$ cofibrantly generated by $\mathcal{I}$ and an \textsc{awfs} $(\mathsf{L'},\mathsf{R'})$ on $\mathbf{D}$ cofibrantly generated by $\mathcal{I'}$. 
    (The \textsc{awfs} $(\mathsf{L'},\mathsf{R'})$ coincides with the one obtained by \emph{projectively lifting} $(\mathsf{L},\mathsf{R})$ along $U\colon \mathbf{C}\to \mathbf{D}$ \cite[Section~4.5]{Bourke_Garner_1}, i.e., the pullback square \cref{eqn:proj-lift-CAT} lifts to a suitable pullback of double categories.)
    By \cite[Proposition~22]{Bourke_Garner_1}, the universality of $(\mathsf{L'},\mathsf{R'})$ induces a canonical oplax morphism of \textsc{awfs} $(\mathsf{L'},\mathsf{R'})\to (\mathsf{L},\mathsf{R})$ whose functor part is $F$. This involves a comonad functor from 
    $(\mathbf{D^2},\mathsf{L'})$ to $(\mathbf{C^2},\mathsf{L})$ whose functor part is $F^\mathbf{2}$. Since $F$ preserves the initial object, we can restrict it and obtain a comonad functor
    \[\begin{tikzpicture}[baseline=-\the\dimexpr\fontdimen22\textfont2\relax ]
		\node(00) at (0,1) {$\mathbf{C}$};
		\node(01) at (3,1) {$\mathbf{C}$};
		\node(10) at (0,-1) {$\mathbf{D}$};
		\node(11) at (3,-1) {$\mathbf{D}$};
		
		\draw [->] (00) to node[auto, labelsize] {$Q$} (01); 
		\draw [<-] (01) to node[auto, labelsize] {$F$} (11); 
		\draw [<-] (00) to node[auto,swap,labelsize] {$F$} (10); 
		\draw [->] (10) to node[auto,swap,labelsize] {$Q'$} (11);   
		\draw [2cell] (1.5,0.3) to node[auto,labelsize] {$\tau$} (1.5,-0.3); 
\end{tikzpicture}\]
    from $(\mathbf{D},Q')$ to $(\mathbf{C},Q)$. 
    The comonad opfunctor $(U,\sigma)$ in \cref{prop:base-change-for-Q} and $(F,\tau)$ form a (doctrinal) adjunction \cite{Kelly:doctrinal}, i.e., $\sigma$ is the mate of $\tau$ under $F\dashv U$. 
    (To see this, observe that we also have a canonical lax morphism of \textsc{awfs} $(\mathsf{L},\mathsf{R})\to (\mathsf{L'},\mathsf{R'})$ whose functor part is $U$ by \cite[Section~2.10]{Bourke_Garner_1}. 
    One can also restrict this to obtain a comonad opfunctor from $Q$ to $Q'$, which coincides with $(U,\sigma)$ and is also the right adjoint of $(F,\tau)$.)
\end{remark}

A comonad opfunctor (such as the one in \cref{prop:base-change-for-Q}) induces a functor between the respective Kleisli categories.
The following proposition ensures that, in the case of our interest, the codomain of this induced functor is indeed $\GSet$.

\begin{proposition}
	\label{prop:every-globular-set-is-algebraically-cofibrant}
	The cofibrant replacement comonad on $\GSet$ with respect to $\mathcal{I'}$ in \Cref{defn:TheQ} is isomorphic to the identity comonad.
\end{proposition}
\begin{proof}
	In view of \cref{rmk:fibres-of-ATF}, it suffices to show the following: for each $A\in\GSet$, $(1_A,\overline\omega_A)$ is the initial object of the category $\ATF(\GSet;\mathcal{I'})_A$. (In other words, $\ATF(\GSet;\mathcal{I'})_A$ has a zero object.) 

	To this end, suppose that we are given an algebraic trivial fibration $(f \colon B \to A, \overline\kappa)$ over $A$.
    Then a morphism $h\colon A\to B$ in $\GSet$ determines a morphism $(h,1_A)\colon (1_A,\overline\omega_A)\to (f,\overline\kappa)$ in $\ATF(\GSet;\mathcal{I'})_A$ if and only if 
    \begin{itemize}
        \item for each $a\in A_0$, we have $h(a)=\overline\kappa(0;!_B,a)$, where $!_B\colon\emptyset\to B$; and
        \item for each $n>0$ and each $c\colon a\to b$ in $A_n$, we have $h(c)=\overline\kappa\bigl(n;\langle h(a),h(b)\rangle,c\bigr)$. 
    \end{itemize}
    In fact, these conditions \emph{define} $h$ inductively, which means that there exists precisely one such $h$. 
\end{proof}

Thus, applying \cref{prop:base-change-for-Q} to the situation of \cref{defn:TheQ}, we obtain a canonical natural transformation     
\[\begin{tikzpicture}[baseline=-\the\dimexpr\fontdimen22\textfont2\relax ]
      \node(00) at (0,1) {$\WkCats{\omega}$};
      \node(01) at (3,1) {$\WkCats{\omega}$};
      \node(10) at (0,-1) {$\GSet$};
      \node(11) at (3,-1) {$\GSet$,};
      
      \draw [->] (00) to node[auto, labelsize] {$Q$} (01); 
      \draw [->] (01) to node[auto, labelsize] {$U$} (11); 
      \draw [->] (00) to node[auto,swap,labelsize] {$U$} (10); 
      \draw [->] (10) to node[auto,swap,labelsize] {$1_{\GSet}$} (11);   
      \draw [2cell] (1.5,-0.3) to node[auto, swap,labelsize] {$\sigma$} (1.5,0.3); 
\end{tikzpicture}\]
such that $(U,\sigma)$ is a comonad opfunctor from $(\WkCats{\omega},Q)$ to $(\GSet,1_{\GSet})$. 
Concretely, for each weak $\omega$-category $X$, the globular map $\sigma_X\colon X\to QX$ is obtained by applying the function $\overline \kappa_X$ associated to $\varepsilon_X$ inductively:
\begin{itemize}
    \item A $0$-cell $x$ of $X$ is mapped to the $0$-cell $\sigma_X(x)=\overline\kappa_X(0;!_{QX},x)$ of $QX$ where $!_{QX}\colon \emptyset\to QX$.
    \item For $n \ge 1$, an $n$-cell $u\colon x\to y$ in $X$ is mapped to the $n$-cell $\sigma_X(u)=\overline{\kappa}_X\bigl(n;\ppair{\sigma_X(x),\sigma_X(y)},u\bigr)\colon \sigma_X(x)\to\sigma_X(y)$ in $QX$.
\end{itemize}

It follows that $(U,\sigma)$ induces a functor $V\colon \WkCat{\omega}\to \GSet$ between the Kleisli categories making the diagram 
    \[\begin{tikzpicture}[baseline=-\the\dimexpr\fontdimen22\textfont2\relax ]
      \node(00) at (0,1) {$\WkCats{\omega}$};
      \node(01) at (4,1) {$\WkCat{\omega}$};
      \node(10) at (2,-1) {$\GSet$};
      
      \draw [->] (00) to node[auto, labelsize] {$J$} (01); 
      \draw [->] (01) to node[auto, labelsize] {$V$} (10); 
      \draw [->] (00) to node[auto,swap,labelsize] {$U$} (10); 
\end{tikzpicture}\]
commute, where $J$ is the identity-on-objects functor mentioned at the end of \cref{subsec:weak-omega-functors}.
Explicitly, $V$ maps each weak $\omega$-category to its underlying globular set, and each weak $\omega$-functor $p\colon X\rightsquigarrow Y$ to the composite of the globular map $\sigma_X\colon X\to QX$ and the classifying strict $\omega$-functor $\unds{p}\colon QX\to Y$, which we call the \emph{underlying globular map} $\undg{p}$ of the weak $\omega$-functor $p$.
See also \cite[Sections 3 and 4]{Garner_homomorphisms}, where a notion of weak functor between tricategories in the style of \cref{def:weak-functor} is related to the classical definition of trihomomorphism, and in which the underlying globular map of a weak functor is given by a similar construction.

\begin{remark}
\label{rmk:SplEpi}
    One can also induce the functor $V\colon \WkCat{\omega}\to \GSet$ using the universal property of $\WkCat{\omega}$ given in \cite[Theorem 10]{Bourke_Garner_2}, which is different from its universal property as a Kleisli category. In this approach, it suffices to construct a double functor from $\AATF(\WkCats{\omega},\cat{I})$ to $\SSplEpi(\GSet)$, which is possible because there exists a canonical double functor $\AATF(\GSet,\cat{I}')\to \SSplEpi(\GSet)$. The latter corresponds to the identity functor on $\GSet$ (again by \cite[Theorem 10]{Bourke_Garner_2}); see \cref{prop:every-globular-set-is-algebraically-cofibrant}.
\end{remark}

\begin{proposition}
\label{prop:weak-omega-weak-eq-via-und-glob}
    Let $p\colon X\rightsquigarrow Y$ be a weak $\omega$-functor between weak $\omega$-categories. Then $p$ is a weak $\omega$-weak equivalence if and only if its underlying globular map $\undg{p}=\unds{p}\circ \sigma_X\colon X\to Y$ is essentially $\omega$-surjective.
\end{proposition}
\begin{proof}
    Since $\varepsilon_X\circ \sigma_X=1_X$ holds in $\GSet$, the map $\sigma_X\colon X\to QX$ is essentially $\omega$-surjective by \cref{prop:2-out-of-3-f}. Hence $\unds{p}\colon QX\to Y$ is essentially $\omega$-surjective if and only if $\unds{p}\circ \sigma_X\colon X\to Y$ is, by \cref{prop:2-out-of-3-g,prop:2-out-of-3-gf}.
\end{proof}

\begin{lemma}
    \label{lem:sigma-pres-refl-inv}
    For each weak $\omega$-category $X$, the globular map $\sigma_X\colon X\to QX$ preserves and reflects invertible cells.
\end{lemma}
\begin{proof}
    This follows from the fact that the retraction $\varepsilon_X$ of $\sigma_X$, being a strict $\omega$-weak equivalence, reflects and preserves invertible cells (\cref{lem:weak-eq-refl-inv,prop:strict-omega-functor-preserves-invertible-cells}).
\end{proof}

\begin{proposition}
\label{lem:und-glob-pres-inv}
    The underlying globular map of a weak $\omega$-functor preserves invertible cells.
\end{proposition}
\begin{proof}
    By \cref{prop:strict-omega-functor-preserves-invertible-cells,lem:sigma-pres-refl-inv}.
\end{proof}

\begin{proposition}
    The underlying globular map of a weak $\omega$-weak equivalence reflects invertible cells and is essentially $\omega$-injective.
\end{proposition}
\begin{proof}
    The first statement follows from \cref{lem:weak-eq-refl-inv,lem:sigma-pres-refl-inv}. The second statement follows from the first, \cref{prop:weak-omega-weak-eq-via-und-glob}, and \cref{lem:surj-implies-inj}. 
\end{proof}

\begin{proposition}
    The class of weak $\omega$-weak equivalences is closed under retracts in the arrow category $\WkCat{\omega}^\mathbf{2}$.
\end{proposition}
\begin{proof}
    By \cref{prop:weak-omega-weak-eq-via-und-glob}, we can work at the level of underlying globular maps. We can apply \cref{prop:weak-eq-retract} thanks to \cref{lem:und-glob-pres-inv}.
\end{proof}

\begin{remark}
    Given a weak $\omega$-weak equivalence $p \colon X \rightsquigarrow Y$, one might expect to be able to construct its (pseudo-)inverse $q \colon Y \rightsquigarrow X$.
    Although this is entirely plausible, it seems rather difficult to actually achieve it for the following two reasons.

    Firstly, we do not currently have a way of constructing \emph{any} weak $\omega$-functor $Y \rightsquigarrow X$ from just the data of $p$ (even assuming the axiom of choice).
    The most natural approach seems to be somehow factorising the strict $\omega$-weak equivalence $\unds{p} \colon QX \to Y$ as $\begin{tikzcd}
        QX \arrow [r, "i"] & Z \arrow [r, "f"] & Y
    \end{tikzcd}$ in such a way that $f$ is an (algebraic) trivial fibration and $i$ is a strict $\omega$-functor admitting a retraction $r$ in $\WkCats{\omega}$.
    The initiality of the algebraic trivial fibration $\varepsilon_Y \colon QY \to Y$ then induces the first factor of the following composite, which classifies the putative inverse $q$:
    \[
    \unds{q} \colon
    \begin{tikzcd}
        QY \arrow [r] & Z \arrow [r, "r"] & QX \arrow [r, "\varepsilon_X"] & X
    \end{tikzcd}
    \]
    The problem is obtaining the factorisation $\unds{p} = fi$.
    We can apply the small object argument to $\unds{p}$ with respect to $\mathcal{I}$ in \cref{defn:TheQ},
    but then we do not know if $i$ would have a retraction $r$.
    Instead of this ``(cofibration, trivial fibration)-factorisation,'' we might want to try the ``(trivial cofibration, fibration)-factorisation,'' but in this case we do not even know what the ``generating trivial cofibrations'' should be.

    Secondly, even if we can construct a weak $\omega$-functor $q \colon Y \rightsquigarrow X$, we do not know what it means for $q$ to be an inverse of $p$, let alone how to prove it.
    A potential approach might be to define a notion of pseudo-natural equivalence using \emph{cylinders} (and in particular a weak $\omega$-categorical version of $\Gamma X$ in \cite{Lafont_Metayer_Worytkiewicz_folk_model_str_omega_cat}), although we currently have no concrete ideas beyond that.
\end{remark}

\subsection*{Acknowledgements}
We thank John Bourke and Richard Garner for their helpful comments, and
the anonymous referees for their careful reading and detailed suggestions, which helped improve the presentation of this paper. 

The first-named author acknowledges the support of JSPS Overseas Research Fellowships, an Australian Research Council Discovery Project DP190102432, and the Grant Agency of the Czech Republic under the grant 22-02964S.
The second-named author is supported by JSPS Research Fellowship for Young Scientists and JSPS KAKENHI Grant Number JP23KJ1365.
The third-named author is grateful to the JSPS for the support of KAKENHI Grant Numbers JP21K20329, JP23K12960, and JP24KJ0126.

\appendix
\newcommand{\Mon}{\mathbf{Mon}}
\newcommand{\GOpd}{\mathbf{GOpd}}
\newcommand{\OC}{\mathbf{OC}}
\newcommand{\CC}{\mathbf{CC}}
\section{\texorpdfstring{$\varepsilon_X$}{ε\textunderscore X} is bijective on objects}
\label{apx:epsilon-bo}

Here we prove \cref{prop:epsilonBij}, which states that the strict $\omega$-functor $\varepsilon_X\colon QX\to X$ is bijective on objects for any weak $\omega$-category $X$.

We first show that the set $(L1)_0$ of $0$-cells of the globular set $L1$ is a singleton.
(Although it should be possible, as mentioned in \cite[Proof of Theorem 10]{Garner_univ}, to deduce this statement from results in \cite{Cheng_monad}, we found it rather non-trivial to actually do so, which is why give a different proof here.)
To this end, we recall the notion of \emph{globular operad} \cite{Leinster_book}, which is equivalent to that of cartesian monad over $T$. 

In general, let $\mathcal{E}$ be a finitely complete category and $S$ a cartesian monad on $\mathcal{E}$. 
We define $[\mathcal{E},\mathcal{E}]_{\mathrm{cart}}$ as the category of all pullback-preserving endofunctors on $\cat{E}$ and cartesian natural transformations between them. Note that $[\mathcal{E},\mathcal{E}]_{\mathrm{cart}}$ is a strict monoidal category under composition of endofunctors, and $S$ is a monoid object therein.
Now, the evaluation functor $\mathrm{ev}_1\colon[\mathcal{E},\mathcal{E}]_{\mathrm{cart}}/S\to\mathcal{E}/S1$ at the terminal object $1\in\mathcal{E}$ is an equivalence of categories \cite[3.1 and 3.2]{Kelly_club_data_type}, and hence we can transport the (strict) monoidal structure on the slice monoidal category $[\mathcal{E},\mathcal{E}]_{\mathrm{cart}}/S$ along $\mathrm{ev}_1$ to induce a (not necessarily strict) monoidal structure on $\mathcal{E}/S1$. The monoidal unit in $\cat{E}/S1$ is the component $\eta^S_1\colon 1\to S1$ of the unit $\eta^S$ of the monad $S$ at the terminal object $1$, and given two objects $f\colon X\to S1$ and $g\colon Y\to S1$ of $\cat{E}/S1$, their monoidal product $f\otimes g$ in $\cat{E}/S1$ is the composite of the top horizontal row in the following diagram.
  \[
    \begin{tikzcd}
      SX\times_{S1} Y
      \ar[r]
      \ar[d]
      \ar[rd,phantom,"\lrcorner"very near start]
        &
        SX
        \ar[d,"{S!}"]
        \ar[r,"{Sf}"] & 
        S^21 \ar[r,"{\mu^S_1}"]
        &
        S1
      \\
      Y
      \ar[r,"{g}"']
        &
        S1
    \end{tikzcd}
  \]
An \emph{$S$-operad} \cite{Leinster_book} is a monoid object in $\mathcal{E}/S1$; from the monoidal equivalence $\mathrm{ev}_1\colon[\mathcal{E},\mathcal{E}]_{\mathrm{cart}}/S\to\mathcal{E}/S1$, it is clear that an $S$-operad corresponds to a cartesian monad on $\cat{E}$ equipped with a cartesian monad morphism to $S$.

A \emph{globular operad} is a $T$-operad where $T$ is the free strict $\omega$-category monad on $\GSet$. Therefore a globular operad consists of an object $a\colon O\to T1$ and morphisms $e\colon \eta^T_1\to a$ and $m\colon a\otimes a\to a$ in $\GSet/T1$ satisfying the monoid axioms. In particular, $e$ is a globular map $e\colon 1\to O$, and hence it gives rise to an $n$-cell $e_n\in O_n$ of the globular set $O$ for each natural number $n$.
We define the \emph{normal part} $O^\mathrm{norm}$ of the $T$-operad $O=(a\colon O\to T1,e,m)$ as the globular subset $O^\mathrm{norm}\subseteq O$ defined as follows.
\begin{itemize}
    \item There is only one $0$-cell in $O^\mathrm{norm}$, namely $e_0$.
    \item For $n>0$, the set $O^\mathrm{norm}_n$ of all $n$-cells of $O^\mathrm{norm}$ consists of all $n$-cells $x\in O_n$ of $O$ satisfying $s^O_0(x)=e_0=t^O_0(x)$.
\end{itemize}
Note that to give a globular map $X\to O^\mathrm{norm}$ from $X\in \GSet$ is equivalent to giving a globular map $X\to O$ which maps every $0$-cell of $X$ to $e_0$.

We claim that $O^\mathrm{norm}$ (more precisely, the subobject $O^\mathrm{norm}\xhookrightarrow{} O\xrightarrow{a} T1$ of $a$ in $\GSet/T1$) is closed under the monoid structure of $O$, and hence defines a globular suboperad of $O$. 
It is clear that the globular map $e\colon 1\to O$ factors through $O^\mathrm{norm}$.
To see that $O^\mathrm{norm}$ is closed under $m$, observe that the monad $T$ is trivial at the dimension $0$, i.e., the unit component $\eta^T_X\colon X\to TX$ indues a bijection $X_0\cong (TX)_0$ for any globular set $X$. 
This implies that the functor $(-)_0\colon \GSet/T1\to \Set$, mapping each object $f\colon X\to T1$ of $\GSet/T1$ to the set $X_0$, is strong monoidal with respect to the cartesian monoidal structure on $\Set$. 
Hence the set $O_0$ of $0$-cells of the globular operad $O$ is an ordinary monoid, with $e_0\in O_0$ as the unit element. 
Now it is easy to see that $O^\mathrm{norm}$ is closed under $m$, using the fact that $m$ commutes with the source and target operations.

The notion of \emph{contraction} on an object in $\GSet/T1$ is defined just as in \cref{subsec:T}; see \cite[Section~9.1]{Leinster_book} for details.
If the underlying object $a\colon O\to T1$ of a globular operad $O=(a\colon O\to T1, e,m)$ has a contraction $\kappa$, then $\kappa$ restricts to the normal part $O^\mathrm{norm}$ of $O$ because the globular map $\iota_n\colon \partial G^n\to G^n$ is bijective on $0$-cells for all $n > 0$.

Now, the globular operad $O_L$ for weak $\omega$-categories (which corresponds to the monad $L$ for weak $\omega$-categories) is defined as the initial object in the category of globular operads with contraction; see \cite[Section~9.2]{Leinster_book} for details. 
The above observations imply that the normal part $O^\mathrm{norm}_L$ of $O_L$ is also a globular operad with contraction, and the inclusion map $O^{\mathrm{norm}}_L\to O_L$, being a monomorphism into an initial object, must then be an isomorphism. Therefore we have $O^{\mathrm{norm}}_L\cong O_L$, i.e., the set $(O_L)_0$ is a singleton. 
In terms of the monad $L$, this means that the set $(L1)_0$ is a singleton. 

It follows from the cartesianness of the unit $\eta^L$ that its component $\eta^L_X\colon X\to LX$ induces a bijection $X_0\cong (LX)_0$ for any globular set $X$.

\begin{proposition}
  \label{prop:cosk}
  The functor $(-)_0\colon\WkCats\omega\to\Set$, mapping each weak $\omega$-category $X$ to its set $X_0$ of objects, has a right adjoint $\mathrm{Cosk}\colon \Set\to \WkCats\omega$.
\end{proposition}
\begin{proof}
  Observe that we have a lax monad morphism (i.e., a \emph{monad functor} in the sense of \cite{Street_FTM}) $\bigl((-)_0,(-)_0\circ\eta^L\bigr)\colon(\GSet,L)\to(\Set,1_\Set)$, where $(-)_0\colon \GSet\to \Set$ is the functor mapping each globular set $X$ to its set $X_0$ of $0$-cells. The functor $(-)_0\colon\WkCats\omega\to\Set$ is induced by this lax monad morphism.
  Now, the observation immediately preceding the current proposition implies that $\bigl((-)_0,(-)_0\circ\eta^L\bigr)$ is in fact a \emph{strong} monad morphism.

  The underlying functor $(-)_0\colon \GSet\to \Set$ of this strong monad morphism has both left and right adjoints, given by left and right Kan extensions along the (fully faithful) inclusion $\ulcorner0\urcorner\colon1\to \G$.
  In particular, the right adjoint $\mathrm{Ran}_{\ulcorner{0}\urcorner}\colon \Set\to \GSet$ of $(-)_0\colon \GSet\to \Set$ induces (by \emph{doctrinal adjunction} \cite{Kelly:doctrinal}) a right adjoint of  $\bigl((-)_0,(-)_0\circ\eta^L\bigr)\colon(\GSet,L)\to(\Set,1_\Set)$ in the $2$-category of monads and lax monad morphisms, which induces the desired right adjoint $\mathrm{Cosk}\colon \Set\to \WkCats{\omega}$ of $(-)_0\colon\WkCats{\omega}\to \Set$.
\end{proof}
\begin{proof}[Proof of \cref{prop:epsilonBij}]
For any weak $\omega$-category $Y$, we denote the unique strict $\omega$-functors $\emptyset\to Y$ and $Y\to 1$ by $!_Y$ and $!^Y$, respectively.

Let $X$ be a weak $\omega$-category.
For any object $x\in X_0$ of $X$, let $\overline x\in (QX)_0$ be the object of $QX$ corresponding to the diagonal strict $\omega$-functor $\overline \kappa_X(0;!_{QX},x)$ below:
\begin{equation*}
\begin{tikzpicture}[baseline=-\the\dimexpr\fontdimen22\textfont2\relax ]
      \node(00) at (0,1) {$\emptyset$};
      \node(01) at (2,1) {$QX$};
      \node(10) at (0,-1) {$FG^0$};
      \node(11) at (2,-1) {$X$.};
      
      \draw [->] (00) to node[auto, labelsize] {$!_{QX}$} (01); 
      \draw [->] (01) to node[auto, labelsize] {$\varepsilon_X$} (11); 
      \draw [->] (00) to node[auto,swap,labelsize] {$F\iota_0$} (10); 
      \draw [->] (10) to node[auto,swap,labelsize] {$x$} (11);   
      \draw [->, dashed] (10) to node[midway,fill=white,labelsize] {$\overline \kappa_X(0;!_{QX},x)$} (01);
\end{tikzpicture}
\end{equation*}
We claim that the function $(\overline -)\colon X_0\to (QX)_0$, mapping $x\in X_0$ to $\overline x\in (QX)_0$, is an inverse of the function $(\varepsilon_X)_0\colon (QX)_0\to X_0$, mapping $y\in (QX)_0$ to $\varepsilon_X(y)\in X_0$.
For any $x\in X_0$, we clearly have $\varepsilon_X(\overline{x})=x$ by the commutativity of the above diagram.

It suffices to show that the function $(\overline -)\colon X_0\to (QX)_0$ is surjective.
To this end, consider a two-element set $S=\{a,b\}$, and let $E=\mathrm{Cosk}\,S\in\WkCats{\omega}$. 
The adjointness $(-)_0\dashv \mathrm{Cosk}$ implies that $E$ is right orthogonal to $\{F\iota_n\,|\,n\ge1\}$.
Hence, in order to equip the unique strict $\omega$-functor $!^E\colon E\to 1$ to the terminal weak $\omega$-category $1$ with a structure $\overline \lambda$ of an algebraic trivial fibration, we only need to specify $\overline\lambda(0;!_E,\ast)\colon G^0\to E$.
We set it to the transpose of the function $(G^0)_0\to S$ mapping the unique element of $(G^0)_0$ to $a$. 

By the universality of $(\varepsilon_X\colon QX\to X,\overline\kappa_X)$, the unique strict $\omega$-functor $!^X\colon X\to 1$ uniquely extends to a morphism $(h,!^X)\colon (\varepsilon_X,\overline{\kappa}_X)\to (!^E,\overline\lambda)$ in $\ATF(\WkCats\omega;\mathcal{I})$.
The strict $\omega$-functor $h\colon QX\to E$ is necessarily the transpose of some function $\hat{h}\colon (QX)_0\to S$, and one can check that the pair $(h,!^X)$ is a morphism $(\varepsilon_X,\overline{\kappa}_X)\to (!^E,\overline \lambda)$ in $\ATF(\WkCats\omega;\mathcal{I})$ if and only if we have $\hat{h}(y)=a$ for all $y$ in the image of the function $(\overline -)\colon X_0\to (QX)_0$.
Therefore the uniqueness of such $h$, or equivalently that of such $\hat h$, implies that $(\overline -)\colon X_0\to (QX)_0$ is surjective. 
\end{proof}

\bibliographystyle{plain}
\bibliography{mybib}

\end{document}